\documentclass[11pt, a4paper,leqno]{amsart}
\usepackage{amsmath,amsthm,amscd,amssymb,amsfonts, amsbsy}
\usepackage{latexsym}
\usepackage{txfonts}
\usepackage{exscale}
\usepackage{comment}

\calclayout
\allowdisplaybreaks

\newtheorem{theorem}{Theorem}[section]

\newtheorem{lemma}{Lemma}[section]
\newtheorem{corollary}{Corollary}[section]
\theoremstyle{definition}
\newtheorem{definition}{Definition}
\newtheorem{remark}{Remark}[section]

\numberwithin{equation}{section}

\newcommand{\beq}{\begin{equation}}
\newcommand{\bea}[1]{\begin{array}{#1} }
\newcommand{\eeq}{ \end{equation}}
\newcommand{\ea}{ \end{array}}

\newcommand{\om}{\omega}

\def \rnn {{\mathbb {R}}^{N+1}}

\def \d {{\delta}}

\def\mean#1{\mathchoice%
          {\mathop{\kern 0.2em\vrule width 0.6em height 0.69678ex depth -0.58065ex
                  \kern -0.8em \intop}\nolimits_{\kern -0.4em#1}}%
          {\mathop{\kern 0.1em\vrule width 0.5em height 0.69678ex depth -0.60387ex
                  \kern -0.6em \intop}\nolimits_{#1}}%
          {\mathop{\kern 0.1em\vrule width 0.5em height 0.69678ex
              depth -0.60387ex
                  \kern -0.6em \intop}\nolimits_{#1}}%
          {\mathop{\kern 0.1em\vrule width 0.5em height 0.69678ex depth -0.60387ex
                  \kern -0.6em \intop}\nolimits_{#1}}}

\def\vintslides_#1{\mathchoice%
          {\mathop{\kern 0.1em\vrule width 0.5em height 0.697ex depth -0.581ex
                  \kern -0.6em \intop}\nolimits_{\kern -0.4em#1}}%
          {\mathop{\kern 0.1em\vrule width 0.3em height 0.697ex depth -0.604ex
                  \kern -0.4em \intop}\nolimits_{#1}}%
          {\mathop{\kern 0.1em\vrule width 0.3em height 0.697ex depth -0.604ex
                  \kern -0.4em \intop}\nolimits_{#1}}%
          {\mathop{\kern 0.1em\vrule width 0.3em height 0.697ex depth -0.604ex
                  \kern -0.4em \intop}\nolimits_{#1}}}

\newcommand{\aveint}[2]{\mathchoice%
          {\mathop{\kern 0.2em\vrule width 0.6em height 0.69678ex depth -0.58065ex
                  \kern -0.8em \intop}\nolimits_{\kern -0.45em#1}^{#2}}%
          {\mathop{\kern 0.1em\vrule width 0.5em height 0.69678ex depth -0.60387ex
                  \kern -0.6em \intop}\nolimits_{#1}^{#2}}%
          {\mathop{\kern 0.1em\vrule width 0.5em height 0.69678ex depth -0.60387ex
                  \kern -0.6em \intop}\nolimits_{#1}^{#2}}%
          {\mathop{\kern 0.1em\vrule width 0.5em height 0.69678ex depth -0.60387ex
                  \kern -0.6em \intop}\nolimits_{#1}^{#2}}}

\def\eqn#1$$#2$${\begin{equation}\label#1#2\end{equation}}
\def\charfn_#1{{\raise1.2pt\hbox{$\chi
_{\kern-1pt\lower3pt\hbox{{$\scriptstyle#1$}}}$}}}

\def\qq1{q_*}
\def\q2{q_{**}}
\def\dist{\operatorname{dist}}

\newdimen\vintbar
\def\vint{-\kern-\vintbar\int}

\def\I{\mathcal I}

\def\K{\mathcal K}

\def\Z{\mathcal Z}
\def\0{\boldsymbol 0}

\newcommand{\R}{\mathbb R}
\renewcommand{\Z}{\mathbb Z}

\newcommand{\N}{\mathbb N}

\newcommand{\eps}{\epsilon}

\newcommand{\ol}{\overline}
\newcommand{\tr}{\operatorname{tr}}

\def\Xint#1{\mathchoice
{\XXint\displaystyle\textstyle{#1}}%
{\XXint\textstyle\scriptstyle{#1}}%
{\XXint\scriptstyle\scriptscriptstyle{#1}}%
{\XXint\scriptscriptstyle%
\scriptscriptstyle{#1}}%
\!\int}
\def\XXint#1#2#3{{\setbox0=\hbox{$#1{#2#3}{%
\int}$ }
\vcenter{\hbox{$#2#3$ }}\kern-.6\wd0}}
\def\barint{\,\Xint -} 
\def\bariint{\barint_{} \kern-.4em \barint}
\def\bariiint{\bariint_{} \kern-.4em \barint}
\renewcommand{\iint}{\int_{}\kern-.34em \int} 
\renewcommand{\iiint}{\iint_{}\kern-.34em \int} 
\newtoks\by
\newtoks\paper
\newtoks\book
\newtoks\jour
\newtoks\yr
\newtoks\pages
\newtoks\vol
\newtoks\publ

\def\name[#1, #2]{#1 #2}
\def\ota{{\hbox{\bf ???}}}
\def\cLear{\by=\ota\paper=\ota\book=\ota\jour=\ota\yr=\ota
\pages=\ota\vol=\ota\publ=\ota}
\def\endpaper{\the\by, \textit{\the\paper},
{\the\jour} \textbf{\the\vol} (\the\yr), \the\pages.\cLear}
\def\endbook{\the\by, \textit{\the\book},
\the\publ, \the\yr.\cLear}
\def\endpap{\the\by, \textit{\the\paper}, \the\jour.\cLear}
\def\endproc{\the\by, \textit{\the\paper}, \the\book, \the\publ,
\the\yr, \the\pages.\cLear}

\renewcommand{\d}{\, \mathrm{d}}

\newcommand{{\II}}{\textrm{II}}

\begin{document}
\title[Tug-of-war with Kolmogorov]{Tug-of-war with Kolmogorov}

\address{Carmina Fjellstr\"{o}m\\Department of Mathematics, Uppsala University\\
S-751 06 Uppsala, Sweden}
\email{carmina.fjellstrom@math.uu.se}

\address{Kaj Nystr\"{o}m\\Department of Mathematics, Uppsala University\\
S-751 06 Uppsala, Sweden}
\email{kaj.nystrom@math.uu.se}

\address{Matias Vestberg\\Department of Mathematics, Uppsala University\\
S-751 06 Uppsala, Sweden}
\email{matias.vestberg@math.uu.se}


\author{Carmina Fjellstr\"{o}m, Kaj Nystr{\"o}m and Matias Vestberg}


\begin{abstract}
\noindent \medskip
We introduce a new class of strongly degenerate nonlinear parabolic PDEs
$$((p-2)\Delta_{\infty,X}^N+\Delta_X)u(X,Y,t)+(m+p)(X\cdot\nabla_Yu(X,Y,t)-\partial_tu(X,Y,t))=0,$$
$(X,Y,t)\in\mathbb R^m\times \mathbb R^m\times \mathbb R$, $p\in (1,\infty)$, combining the classical PDE of Kolmogorov and the normalized $p$-Laplace operator.
We characterize solutions in terms of an asymptotic mean value property and the results
are connected to the analysis of certain tug-of-war games with noise.  The value functions for the games introduced approximate solutions to the stated PDE  when the
parameter that controls the size of the possible steps goes to zero.  Existence and uniqueness of viscosity solutions to the Dirichlet problem is established. The asymptotic mean value property, the associated games and the geometry underlying the Dirichlet problem, all reflect the family of dilation and  the Lie group underlying operators of Kolmogorov type, and this makes our setting different from the context of standard parabolic
dilations and Euclidean translations applicable in the context of the heat operator and the normalized parabolic infinity Laplace operator. \\

\noindent
2000  {\em Mathematics Subject Classification.} 35K51, 35K65, 35K70, 35K92, 35H20, 35R03, 35Q91, 91A80, 91A05.
\noindent

\medskip

\noindent
{\it Keywords and phrases: Dirichlet boundary conditions, dynamic programming principle,  p-Laplacian, infinity Laplacian, Kolmogorov equation, mean value property, stochastic games, tug-of-war games, viscosity solutions, ultraparabolic, hypoelliptic.}
\end{abstract}

\maketitle



\setcounter{equation}{0} \setcounter{theorem}{0}
\section{Introduction}
In recent years, there has been a surge in the study of tug-of-war games, mean-value properties, and boundary value problems for degenerate elliptic and parabolic equations modeled on the infinity Laplace operator and the $p$-Laplace operator. The impetus for these developments has been the seminal papers on  tug-of-war games of Peres, Schramm, Sheffield, and Wilson \cite{PSSW,PS}. They showed that these two-player zero-sum games have connections to homogeneous and inhomogeneous normalized PDEs in nondivergence form via the dynamic programming principle (DPP for short). Connections to nonlinear mean value formulas were developed in \cite{manfredipr10c} and \cite{manfredipr12} and, concerning related boundary value problems, we mention \cite{juutinenk06,kawohlmp12}.

In \cite{manfredipr10c},  the authors contribute to the dynamic and parabolic part of the theory by establishing  mean value formulas for certain nonlinear and degenerate parabolic equations, and by relating these formulas to the dynamic programming principle satisfied by the value functions of parabolic tug-of-war games with noise.  A starting point in \cite{manfredipr10c} is the observation that a function $u=u(X,t):\mathbb R^m\times \mathbb R\to \mathbb R$ solves the heat equation
\begin{eqnarray*}
\mathcal{H}u(X,t):=\Delta_{X}u(X,t)-\partial_tu(X,t)=0,
\end{eqnarray*}
if and only if
\begin{equation*}
u(X,t)=\barint_{B_\epsilon(X)} \barint_{t-\epsilon^2/(m+2)}^t u(\tilde X,\tilde t) \d \tilde X \d \tilde t+{o}(\epsilon^2),\mbox{ as }\epsilon\to 0,
\end{equation*}
where  $B_\epsilon(X)$ denotes the standard Euclidean ball in $\mathbb R^m$ of radius $\epsilon$ and centered at $X\in \mathbb R^m$.

An important contribution in \cite{manfredipr10c} is a non-linear version of the stated characterization of solutions to $\mathcal{H}u(X,t)=0$,  stating, for $p$, $1<p<\infty$, that $u=u(X,t):\mathbb R^m\times \mathbb R\to \mathbb R$ is a viscosity solution to the equation
\begin{equation}\label{meanvalue1heatg}
((p-2)\Delta_{\infty,X}^N+\Delta_X)u(X,t)-(m+p)\partial_tu(X,t)=0,
\end{equation}
if and only if
\begin{align*}
u(X,t)&=\frac \alpha 2\barint_{t-\epsilon^2}^t\biggl ( \biggl\{\sup_{\tilde X\in{{B_\epsilon(X)}}}u(\tilde X,\tilde t)\biggr \}+\biggl\{\inf_{\tilde X\in{{B_\epsilon(X)}}}u(\tilde X,\tilde t)\biggr \}\biggr )\d \tilde t\notag\\
&+\beta \barint_{B_\epsilon(X)} \barint_{t-\epsilon^2/(m+2)}^t u(\tilde X,\tilde t) \d \tilde X \d \tilde t+{o}(\epsilon^2),\mbox{ as }\epsilon\to 0,
\end{align*}
in the viscosity sense. Here,
\begin{align*}
 \alpha:= \frac{p-2}{m+p},\hspace{7mm} \beta:= \frac{m+2}{m+p}.
\end{align*}
Note that, formally,
\begin{align*}
((p-2)\Delta_{\infty,X}^N+\Delta_X)u(X,t)&=|\nabla u(X,t)|^{2-p}\Delta_{p,X}u(X,t)\notag\\
&:=|\nabla_X u(X,t)|^{2-p}\nabla_X\cdot(|\nabla_X u(X,t)|^{p-2}\nabla_X u(X,t)),
\end{align*}
showing the connection between the $p$-Laplace operator ($\Delta_{p,X}$), the (normalized) infinity Laplace operator ($\Delta_{\infty,X}^N$) and the Laplace operator ($\Delta_X:=\Delta_{2,X}$). Furthermore, dividing through in \eqref{meanvalue1heatg} with the factor $(m+p)$, and letting $p\to\infty$, we formally also deduce that
$u=u(X,t):\mathbb R^m\times \mathbb R\to \mathbb R$ is a viscosity solution to the equation
\begin{equation*}
\Delta_{\infty,X}^Nu(X,t)-\partial_tu(X,t)=0,
\end{equation*}
if and only if
\begin{align*}
u(X,t)&=\frac 1 2\barint_{t-\epsilon^2}^t\biggl ( \biggl\{\sup_{\tilde X\in{{B_\epsilon(X)}}}u(\tilde X,\tilde t)\biggr \}+\biggl\{\inf_{\tilde X\in{{B_\epsilon(X)}}}u(\tilde X,\tilde t)\biggr \}\biggr )\d \tilde t+{o}(\epsilon^2),\mbox{ as }\epsilon\to 0,
\end{align*}
in the viscosity sense. I.e., the equivalence between solutions and mean value properties can be seen to hold for all $p$, $1<p\leq\infty$.

In \cite{manfredipr10c}, it is also proved that these mean value formulas are related to the DPP satisfied by the value functions of parabolic tug-of-war games with noise. The
DPP is exactly the mean value formula without the correction term ${o}(\epsilon^2)$. In \cite{manfredipr10c}, functions that satisfy the DPP  are called $(p,\epsilon)$-parabolic. As shown in
\cite{manfredipr10c}, $(p,\epsilon)$-parabolic equations have interesting
properties making them interesting on their own, but, in addition, they approximate
solutions to the corresponding parabolic equation, and  $(p,\epsilon)$-parabolic functions converge in the limit as $\epsilon\to 0$ to viscosity solutions of the Dirichlet problem.

In this paper, we initiate a program similar to \cite{manfredipr10c}, but in a new and different situation. Instead of the heat operator $\mathcal{H}$, our starting point is the operator
\begin{eqnarray}\label{e-kolm-nd}
   \K:=\sum_{i=1}^{m}\partial_{x_i x_i}+\sum_{i=1}^m x_i\partial_{y_{i}}-\partial_t,
    \end{eqnarray}
   acting on functions in $\mathbb R^{M+1}$, $M:=2m$, $m\geq 1$, equipped with coordinates $$(X,Y,t):=(x_1,...,x_{m},y_1,...,y_{m},t)\in \mathbb R^{m}\times\mathbb R^{m}\times\mathbb R.$$ The operator in \eqref{e-kolm-nd} was originally introduced and studied in 1934 by Kolmogorov \cite{K} as an example of a degenerate parabolic
operator having strong regularity properties.  Kolmogorov proved that $\K$ has a fundamental solution
$\Gamma = \Gamma(X,Y,t, \tilde X, \tilde Y, \tilde t)$ which is smooth on the set $\big\{(X,Y,t) \ne (\tilde X, \tilde
Y, \tilde t)\big\}$. As a consequence,
\begin{eqnarray}\label{uu3}
 \mbox{$u$ is a distributional solution to } \K u = f \in C^\infty \quad \Rightarrow \quad u \in C^\infty.
\end{eqnarray}
The property in \eqref{uu3} can be restated as
\begin{eqnarray*}
\mbox{$\K$ is hypoelliptic}.
\end{eqnarray*}
As can be read  in the introduction of H{\"o}rmander's famous paper \cite{H}, the work of Kolmogorov strongly influenced H{\"o}rmander when he developed his theory of hypoelliptic operators.

Kolmogorov was originally motivated by statistical physics and he studied $\K$ in the context of stochastic processes. The fundamental
solution $\Gamma(\cdot, \cdot, \cdot, \tilde X, \tilde Y, \tilde t)$ defines the density of the stochastic process
$(X_t,Y_t)$ which solves the Langevin system
\begin{equation}\label{e-langevin}
\left \{
\begin{aligned}
   & d X_t = \sqrt{2}d W_t, \quad X_{\tilde t} = \tilde X,\\
   & d Y_t = X_t d t, \quad\quad Y_{\tilde t} = \tilde Y,
\end{aligned}
\right.
\end{equation}
where $W_t$ is a standard $m$-dimensional Wiener process. The system in \eqref{e-langevin} is a system
with $2m$ degrees of freedom, and $(X,Y)\in\mathbb R^{2m}$, $X=(x_1,...,x_m)$ and $Y=(y_{1},...,y_{m})$, are the velocity and the position of the system, respectively. The model in \eqref{e-langevin} and the equation in \eqref{e-kolm-nd} are of fundamental
importance in kinetic theory as they form the basis for Langevin type models for particle dispersion.

The natural family of dilations
for $\K$, $(\delta_r)_{r>0}$, on $\R^{M+1}$, and the Lie group on $\R^{M+1}$ preserving $\K u=0$ are different from standard parabolic
dilations and Euclidean translations applicable in the context of the heat operator.  The operator $\K$ can be expressed as
$$ \K=\sum_{i=1}^m X_i^2 +X_0,$$
where
\begin{equation*}
 X_i:= \partial_{x_i},\quad i=1,\ldots,m, \qquad X_0:=\sum_{i=1}^mx_i\partial_{y_i}-\partial_t.
\end{equation*}
The vector fields  $X_1, \dots, X_m$ and $X_0$ are left-invariant with respect to the group law \begin{equation}\label{e70}
 (\tilde X,\tilde Y,\tilde t)\circ (X, Y,t)=(\tilde X+X,\tilde Y+Y-t\tilde X,\tilde t+t),
\end{equation}
 in the sense that
\begin{eqnarray*}
	X_i\bigl( u ((\tilde X,\tilde Y,\tilde t) \circ \, \cdot \, ) \bigr)= \left( X_i u \right) ((\tilde X,\tilde Y,\tilde t) \circ \, \cdot \, ), \quad
i=0, \dots, m,
\end{eqnarray*}
for every $(\tilde X,\tilde Y,\tilde t) \in \rnn$. Consequently, $$\K \bigl( u ((\tilde X,\tilde Y,\tilde t) \circ \, \cdot \, ) \bigr) =\bigl( \K u \bigr)
((\tilde X,\tilde Y,\tilde t) \circ \,
\cdot \, ).$$
The natural family of dilations for $\K$, $(\delta_r)_{r>0}$, on $\R^{M+1}$,
is defined by
\begin{equation*}
 \delta_r (X,Y,t) =(r X, r^3 Y,r^2 t),
\end{equation*}
for $(X,Y,t) \in \R^{M +1}$,  $r>0$. In particular, the operator $\K$ is $\delta_r$-homogeneous
of degree two, i.e., $\K\circ\delta_r=r^2 (\delta_r\!\circ \K)$, for all $r>0$. Furthermore, note that
\begin{equation*}
(X,Y,t)^{-1}=(-X,-Y-tX,-t),
\end{equation*}
and hence,
\begin{equation}\label{e70++}
 (\tilde  X,\tilde  Y,\tilde  t)^{-1}\circ (X,Y,t)=(X-\tilde  X,Y-\tilde  Y-(\tilde  t-t)\tilde  X,t-\tilde  t).
\end{equation}

The starting point for our analysis consists of a few observations rigorously discussed in the bulk of the paper concerning mean value-like formulas reflecting the family of dilations and translations underlying the operator $\K$. Let, in the following, $\Omega\subset\mathbb R^{M+1}$ be a domain, i.e., a connected open set, and assume that $u$ is a smooth function in $\Omega$.

The first observation is that the asymptotic mean value formula
\begin{equation}\label{meanvalue1bb}
u(X,Y,t)=\barint_{B_\epsilon(X)} \barint_{B_{\epsilon^3}(Y)} \barint_{t-\epsilon^2/(m+2)}^t u(\tilde X,\tilde Y-(\tilde t-t) X,\tilde t) \d \tilde X \d \tilde Y\d \tilde t+{o}(\epsilon^2),\mbox{ as }\epsilon\to 0,
\end{equation}
holds for all $(X,Y,t)\in\Omega$ if and only if
\begin{eqnarray}\label{solv1bb}
  \K u =0\mbox{ in }\Omega.
\end{eqnarray}
The proof of this fact is analogous to that of Lemma \ref{lem:mean_value_p=2} below, where we treat the closely related operator $\mathcal{K}_2$. The only difference between the results in the cases $\mathcal{K}$ and $\mathcal{K}_2$ is the length of the time interval over which the average is taken. Note that the coordinate $(\tilde X,\tilde Y-(\tilde t-t)X,\tilde t)$ in \eqref{meanvalue1bb} is dictated by the group law, see \eqref{e70} and \eqref{e70++}. As it turns out, the equivalence in \eqref{meanvalue1bb}-\eqref{solv1bb} is still true if the statement in \eqref{meanvalue1bb} is replaced by
\begin{equation*}
u(X,Y,t)=\barint_{B_\epsilon(X)} u\bigl(\tilde X,Y+\frac {\epsilon^2}{2(m+2)} \tilde X,t-\frac {\epsilon^2}{2(m+2)}\bigr ) \d \tilde X+{o}(\epsilon^2),\mbox{ as }\epsilon\to 0.
\end{equation*}
This is the content of Theorem \ref{meanvalue1thmgaian} in the case $p=2$. 

The second observation is that,  if $|\nabla_Xu(X,Y,t)|\neq 0$ whenever $(X,Y,t)\in\Omega$, and if the
asymptotic sup-inf (max-min) mean value formula
\begin{align}\label{meanvalue3}
u(X,Y,t)&=\frac 1 2\barint_{B_{\epsilon^3}(Y)} \barint_{t-\epsilon^2}^t \biggl\{\sup_{\tilde X\in{{B_\epsilon(X)}}}u(\tilde X,\tilde Y-(\tilde t-t)X,\tilde t)\biggr \}\d \tilde Y\d \tilde t\notag\\
&+\frac 1 2\barint_{B_{\epsilon^3}(Y)} \barint_{t-\epsilon^2}^t
\biggl\{\inf_{\tilde X\in{{B_\epsilon(X)}}}u(\tilde X,\tilde Y-(\tilde t-t)X,\tilde t)\biggr \}\d \tilde Y\d \tilde t\notag\\
&+{o}(\epsilon^2),\mbox{ as }\epsilon\to 0,
\end{align}
holds for all $(X,Y,t)\in\Omega$, then $u$ solves the partial differential equation
\begin{eqnarray}\label{solv1}
  \K_\infty u(X,Y,t):=\Delta_{\infty,X}^Nu(X,Y,t)+X\cdot\nabla_Yu(X,Y,t)-\partial_tu(X,Y,t)=0\mbox{ in }\Omega,
\end{eqnarray}
in the appropriate viscosity sense. As above, $\Delta_{\infty,X}^N$ is the so called (normalized) infinity Laplace operator in the $X$ variables only,
\begin{eqnarray*}
\Delta_{\infty,X}^Nu:=|\nabla_Xu|^{-2}\langle D_X^2u \nabla_X u,\nabla_X u\rangle=|\nabla_Xu|^{-2}\sum_{i,j=1}^m\partial_{x_ix_j}u\partial_{x_i}u\partial_{x_j}u.
\end{eqnarray*}
The same conclusion is true if \eqref{meanvalue3} is replaced by
\begin{align}\label{meanvalue4}
u(X,Y,t)&=\frac 1 2 \biggl\{\sup_{\tilde X\in{{B_\epsilon(X)}}}u(\tilde X,Y+{\epsilon^2}\tilde X/2,t-{\epsilon^2}/2)+\inf_{\tilde X\in{{B_\epsilon(X)}}}u(\tilde X,Y+{\epsilon^2}\tilde X/2,t-{\epsilon^2}/2)\biggr \}\notag\\
&+{o}(\epsilon^2),\mbox{ as }\epsilon\to 0.
\end{align}
This is the conclusion of Theorem \ref{meanvalue1thmgaian} in the case $p=\infty$. Note that \eqref{meanvalue3} and \eqref{meanvalue4} remain valid, as viscosity solutions are by definition continuous, with $\sup_{\tilde X\in{{B_\epsilon(X)}}}$ and $\inf_{\tilde X\in{{B_\epsilon(X)}}}$ replaced by
$\max_{\tilde X\in\overline{B_\epsilon(X)}}$ and $\min_{\tilde X\in\overline{B_\epsilon(X)}}$, respectively.

As we will see, one can  give a probabilistic interpretation of \eqref{meanvalue4} and the PDE in \eqref{solv1} in the context of a tug-of-war game which loosely can be defined as follows. Assume that the game starts at $(X,Y,t)\in\Omega$ and that $\epsilon>0$ is small. At each step of the game the two opponents flip a fair coin and the winner is allowed to pick a velocity direction $\eta$. The token then gets transported according to $$(X,Y)\to (\hat X,\hat Y),\  \hat X:=X+\epsilon\eta,\ \hat Y:= Y+{\epsilon^2}\hat X/2.$$

In addition, the two cases discussed, $\K u=0$ and  $\K_\infty u=0$, can be combined into tug-of-war games with noise and as a consequence, we are led, for $p\in (1,\infty)$ given, to consider the  Kolmogorov $p$-Laplacian type equation
\begin{align}\label{solv1+-}
  \K_p u(X,Y,t)&:=((p-2)\Delta_{\infty,X}^N+\Delta_X)u(X,Y,t)\notag\\
  &+(m+p)(X\cdot\nabla_Yu(X,Y,t)-\partial_tu(X,Y,t))=0.
\end{align}

To our knowledge the PDEs  $\K_\infty$ and $\K_p$ have previously not been discussed in the literature and therefore all classical questions concerning existence, uniqueness and regularity of solutions seem to be open problems. In particular, to develop the theory of these games and operators, the existence and uniqueness of viscosity solutions to the Dirichlet problem for the operator $\K_p$ $(\K_\infty)$ with continuous boundary data, and in potentially  velocity $(X)$, position $(Y)$ as well as time $(t)$-dependent domains $\Omega\subset\mathbb R^{M+1}$, are of fundamental importance. It is also important to study the limit of the (fair) value function of the  game as $\epsilon\to 0$ and its relation to the Dirichlet problem. Naturally, the Dirichlet problem is of independent interest, but, in this paper, we are particularly interested in this problem in the context of the tug-of-war game and, in the following, we will briefly discuss the additional complexity we encounter in our context in comparison to the corresponding parabolic problems studied in \cite{manfredipr10c}.  The additional complexity essentially stems from two facts.

First, in the setup outlined, the players can only modify  the velocity coordinate $(X)$ of the game process directly, while the position coordinate $(Y)$ of the game process is updated according to
$Y\to Y+{\epsilon^2}X/2$. In this sense, the position coordinate $(Y)$ is determined by velocity and time, and hence the players can only influence
the position coordinate indirectly.

Second,  already in the case of $\K$,  the analysis of the Dirichlet problem is  complicated  by the presence of characteristic points for the operator $\K$ on $\partial\Omega$. Indeed, let
$U_X\subset\mathbb R^m$ and $U_Y\subset\mathbb R^m$ be  bounded domains with say $C^2$-smooth boundaries. Given $T$, $0<T<\infty$,  let $I:=(0,T)\subset \mathbb R$. Considering product domains $\Omega=U_X\times U_Y\times I\subset \mathbb R^{M+1}$ we introduce
\begin{align*}
\partial_\K(U_X\times U_Y\times I):=\partial_1\cup\partial_2\cup\partial_3,
\end{align*}
where
\begin{align*}
\partial_1&:=\partial U_X\times \overline{U_Y}\times[0,T),\notag\\
\partial_2&:=\{(X,Y)\in \overline{U_X}\times\partial U_Y:\ X\cdot N_Y>0\}\times[0,T),\notag\\
\partial_3&:=(U_X\times U_Y)\times\{0\},
\end{align*}
and where $N_{Y}$ denotes the outer unit normal to $U_{Y}$ at $Y\in \partial U_Y$. $\partial_\K(U_X\times U_Y\times I)$ is sometimes referred to as the Kolmogorov boundary of $U_X\times U_Y\times I$, and the Kolmogorov boundary serves, already in the context of the operator $\K$, as the natural substitute for the parabolic boundary used in the context of the Cauchy-Dirichlet problem for uniformly parabolic equations. Given $F\in C(\mathbb R^{M+1})$, the  Dirichlet problem to study, see
\cite{Manfredini,NP} for instance, for $\K$ is the boundary value problem
\begin{align*}
\begin{cases}
\K u(X,Y,t)=0,\quad  &\textrm{for}\quad (X,Y,t)\in U_X\times U_Y\times I,\\
u(X,Y,t)=F(X,Y,t)
,\quad  &\textrm{for}\quad  (X,Y,t)\in \partial_\K(U_X\times U_Y\times I).
\end{cases}
\end{align*}
In particular, this means that no boundary data is imposed on the part of the topological boundary $\partial(U_X\times U_Y\times I)$ defined
by \begin{align}\label{state}
\partial_2^\ast:=\{(X,Y)\in U_X \times\partial U_Y:\ X\cdot N_Y\leq 0\}\times[0,T).
\end{align}

Put together, this implies that if we want the (fair) value function of the  game $(u_\epsilon)$ to converge to a (unique) viscosity solution to the Dirichlet problem for the operator $\K_p$ $(\K_\infty)$, with given boundary data, then the rules of the game have to take into account the fact that the state of the game may be at a point on the part of the boundary defined by $\partial_2^\ast$ and introduced in \eqref{state}. In particular, this means that we in some sense must restrict the directions, as we impose no boundary data on $\partial_2^\ast$, in which the players are allowed to modify the velocity coordinate of the game to ensure that the game process is pushed into $\Omega$ so that the game can be continued. This argument assumes that the game  can only end if the game exits
$\Omega=U_X\times U_Y\times I$ through $\partial_\K(U_X\times U_Y\times I)$. This complicates matters considerably as rules implying that, at instances, the players will only be allowed, when modifying $X$, to choose directions in a cone depending on $X$, $Y$, and $N_Y$, have to be introduced.

To complete the Tug-of-war with Kolmogorov program in all detail, we will, in the second part of the paper, assume $U_Y=\mathbb R^m$. I.e., we will impose no restriction on the position coordinate $Y$ while the pay-off function will depend on all variables. The Dirichlet problem for the operator $\K_p$ $(\K_\infty)$, and the modified tug-of-war games, in more general domains $\Omega=U_X\times U_Y\times I$ are targets for future research, see Section \ref{OP}. Moreover, we note that the probabilistic interpretation  of the PDE corresponding to $\K_p$ makes sense only when $\alpha$ and $\beta$ are nonnegative, that is, when $p\geq 2$.

The rest of the paper is organized as follows. Section \ref{sec1} is of preliminary nature and we here introduce notation and the correct viscosity formalism. Section \ref{sec2} is devoted to the proof of Theorem \ref{meanvalue1thm}, stating the connection between asymptotic mean value properties and solutions to $\K_p u=0$.
Motivated by the asymptotic mean value theorem (Theorem \ref{meanvalue1thm}), in Section \ref{sec3} we study functions satisfying the mean value property without the correction term ${o}(\epsilon^2)$.
To distinguish between our context and the notion of $(p,\epsilon)$-parabolic functions
introduced in \cite{manfredipr10c}, we call these
functions $(p,\epsilon)$-Kolmogorov functions.  In analogy with $(p,\epsilon)$-parabolic functions, $(p,\epsilon)$-Kolmogorov functions have interesting properties
to be studied in their own right. In Section \ref{sec3}, we prove that $(p,\epsilon)$-Kolmogorov functions are value functions of certain tug-of-war games with noise briefly discussed above. In Section \ref{sec4}, we let $\epsilon\to 0$ and we prove, in domains of the form $\Omega:=U_X\times U_Y\times I$, $U_Y=\mathbb R^m$, that the limiting function is the unique viscosity solution to the Dirichlet problem for the PDE introduced. In particular, in Section \ref{sec4}, existence and uniqueness of viscosity solutions to the Dirichlet problem for $\K_p$ is established in certain subsets of $\mathbb R^{M+1}$.  The analysis in Section \ref{sec3} and  Section \ref{sec4} is, as discussed and compared to \cite{manfredipr10c}, complicated by the underlying non-Euclidean Lie group connected to operators of Kolmogorov type, and by the fact that the very notion of parabolic boundary is already more complicated compared to the heat operator $\mathcal{H}$. Finally, in Section \ref{OP}, we state some open problems.

\section{Preliminaries}\label{sec1}

Given $(X,Y,t)\in \R^{M+1}$, we let \begin{equation*}
  \|(X,Y, t)\|:=|(X,Y)|\!+|t|^{\frac{1}{2}},\ |(X,Y)|= |X| + |Y|^\frac{1}{3}.
\end{equation*}
We recall the following pseudo-triangular
inequality: there exists a positive constant ${c}$ such that
\begin{eqnarray}\label{e-ps.tr.in}
 \quad\|(X,Y,t)^{-1}\|\le {c}  \| (X,Y,t) \|,\quad \|(X,Y,t)\circ (\tilde X,\tilde Y,\tilde t)\| \le  {c}  (\| (X,Y,t) \| + \| (\tilde X,\tilde Y,\tilde t)
\|),
\end{eqnarray}
whenever $(X,Y,t),(\tilde X,\tilde Y,\tilde t)\in \R^{M+1}$. Using \eqref{e-ps.tr.in}, it  follows directly that
\begin{equation} \label{e-triangularap}
    \|(\tilde X,\tilde Y,\tilde t)^{-1}\circ (X,Y,t)\|\le c \, \|(X,Y,t)^{-1}\circ (\tilde X,\tilde Y,\tilde t)\|,
\end{equation}
whenever $(X,Y,t),(\tilde X,\tilde Y,\tilde t)\in \R^{M+1}$. Let
\begin{equation}\label{e-ps.distint}
    d_\K((X,Y,t),(\tilde X,\tilde Y,\tilde t)):=\frac 1 2\bigl( \|(\tilde X,\tilde Y,\tilde t)^{-1}\circ (X,Y,t)\|+\|(X,Y,t)^{-1}\circ (\tilde X,\tilde Y,\tilde t)\|).
\end{equation}
Using \eqref{e-triangularap}, it follows that
\begin{equation}\label{e-ps.dist}
   \|(\tilde X,\tilde Y,\tilde t)^{-1}\circ (X,Y,t)\|\approx d_\K((X,Y,t),(\tilde X,\tilde Y,\tilde t))\approx \|(X,Y,t)^{-1}\circ (\tilde X,\tilde Y,\tilde t)\|
\end{equation}
for all $(X,Y,t),(\tilde X,\tilde Y,\tilde t)\in \R^{M+1}$ and with uniform constants. Again, using \eqref{e-ps.tr.in}, we also see that
\begin{equation*} 
    d_\K((X,Y,t),(\tilde X,\tilde Y,\tilde t))\le {c} \bigl(d_\K((X,Y,t),(\hat X,\hat Y,\hat t))+d_\K((\hat X,\hat Y,\hat
t),(\tilde X,\tilde Y,\tilde t))\bigr ),
\end{equation*}
whenever $(X,Y,t),(\hat X,\hat Y,\hat t),(\tilde X,\tilde Y,\tilde t)\in \R^{M+1}$, and hence $d_\K$ is a symmetric quasi-distance. Based on $d_\K$, we introduce the balls
\begin{equation*}
    \mathcal{B}_r(X,Y,t):= \{ (\tilde X,\tilde Y,\tilde t) \in\mathbb R^{M+1} \mid d_\K((\tilde X,\tilde Y,\tilde t),(X,Y,t)) <
r\},
\end{equation*}
for $(X,Y,t)\in \R^{M+1}$ and $r>0$.  The measure of the ball $\mathcal{B}_r(X,Y,t)$ is  $|\mathcal{B}_r(X,Y,t)|\approx r^{{\bf q}}$, independent of $(X,Y,t)$, where
$${\bf q}:=4m+2.$$

Let $D\subset\mathbb R^{M+1}$ be an open set. We denote by $\mbox{LSC}(D)$ the set of lower semicontinuous functions on $D$, i.e., all functions $f: D \to \R$ such that for all points $(\hat X,\hat Y,\hat t)\in D$ and for any sequence $\{(X_n, Y_n, t_n)\}_n$, $(X_n, Y_n, t_n)\in D$,   $(X_n,Y_n,t_n) \to (\hat X,\hat Y,\hat t)$ as $n \to \infty$ in $D$, we have
\begin{equation*}
\liminf _{n \to \infty} f(X_n,Y_n,t_n) \geq f(\hat X,\hat Y,\hat t).
\end{equation*}
We denote by
$\mbox{USC}(D)$ the set of upper semicontinuous functions on $D$, i.e., all functions $f: D \to \R$ such that for all points $(\hat X,\hat Y,\hat t)\in D$ and for any sequence $\{(X_n, Y_n, t_n)\}_n$, $(X_n, Y_n, t_n)\in D$,   $(X_n,Y_n,t_n) \to (\hat X,\hat Y,\hat t)$ as $n \to \infty$ in $D$, we have
\begin{equation*}
\limsup _{n \to \infty}f(X_n,Y_n,t_n) \leq f(\hat X,\hat Y,\hat t).
\end{equation*}
Note that a function $f\in \mbox{USC}(D)$  if and only if $-f\in \mbox{LSC}(D)$. Also, $f$ is  continuous on $D$, $f\in C(D)$, if and only if
$f\in \mbox{USC}(D)\cap \mbox{LSC}(D)$.

We will frequently use the elementary fact that if $D\subset\mathbb R^{M+1}$ is an open set, $\overline{B_\epsilon(X)}\times B_{\epsilon^3}(Y)\times
(t-\epsilon^2,t)\subset D$, and if $u\in C(D)$, then
\begin{align}\label{elementary}
\sup_{\tilde X\in{{B_\epsilon(X)}}}u(\tilde X,\tilde Y-(\tilde t-t)X,\tilde t)&=\max_{\tilde X\in\overline{B_\epsilon(X)}}u(\tilde X,\tilde Y-(\tilde t-t)X,\tilde t),\notag\\
\inf_{\tilde X\in{{B_\epsilon(X)}}}u(\tilde X,\tilde Y-(\tilde t-t)X,\tilde t)&=\min_{\tilde X\in\overline{B_\epsilon(X)}}u(\tilde X,\tilde Y-(\tilde t-t)X,\tilde t),
\end{align}
for every $(\tilde Y,\tilde t)\in  B_{\epsilon^3}(Y)\times
(t-\epsilon^2,t)$ fixed.

For a symmetric $m\times m$-matrix $A$, we denote its largest and smallest eigenvalue
by $\Lambda(A)$ and $\lambda(A)$, respectively, i.e.,
$$\Lambda(A)=\max_{|\eta|=1}(A\eta)\cdot\eta\mbox{ and }\lambda(A)=\min_{|\eta|=1}(A\eta)\cdot\eta.$$

We now give the definition of the super- and subjets used in the proof of Lemma \ref{Vissolsimp} below. We only state the definitions for interior points of the domain, as this is the concept we need.
\begin{definition} \label{def:superjet} Let $D\subset\mathbb R^n$ be an open set. Let $v\in \mbox{USC}(D)$ and $u\in \mbox{LSC}(D)$, $\hat x \in D$, and let $S_n$ be the set of all $n\times n$-dimensional symmetric matrices. The superjet ${J}^{2,+}v(\hat x)$ of $v$ at $\hat x$ is the set of all pairs $(p,A) \in \mathbb R^n\times S_n$ such that
\begin{align*}
v(x) &\leq v(\hat x) + \langle p, x - \hat x\rangle + \frac 1  2\langle A(x - \hat x), x - \hat x\rangle + o(|x - \hat x|^2).
\end{align*}
The subjet ${J}^{2,-}u(\hat x)$ of $u$ at $\hat x$ is the set of all pairs $(p,A) \in \mathbb R^n\times S_n$ such that
\begin{align*}
v(x) &\geq v(\hat x) + \langle p, x - \hat x\rangle + \frac 1  2\langle A(x - \hat x), x - \hat x\rangle + o(|x - \hat x|^2).
\end{align*}
We say that the pair $(p,A) \in \mathbb R^n\times S_n$ belongs to $\ol{J}^{2,+}v(\hat x)$ if there is a sequence $x_j \to \hat x$ such that $v(x_j) \to v(\hat x)$ and there are $(p_j, A_j) \in {J}^{2,+}v(\hat x_j)$ such that $(p_j, A_j) \to (p, A)$. The set $\ol{J}^{2,-}u(\hat x)$ is defined in a corresponding manner.
\end{definition}

\subsection{Viscosity solutions}
In the following, we introduce the notion of viscosity solutions used throughout the paper. To be clear and explicit,
we state the definitions for the case $p=\infty$ and the equation in \eqref{solv1}, and for the case $1<p<\infty$ and the equation in \eqref{solv1+-}, separately. Given $\phi\in C^2(D)$ and $(X, Y,t)\in D$, we let $\nabla_X^2 \phi(X, Y,t)$ denote the Hessian of the map $X \mapsto \phi(X,Y,t)$.

\begin{definition}\label{Vissolagain} Let $D\subset\mathbb R^{M+1}$ be an open set. A function $u\in \mbox{LSC}(D)$ is
a viscosity supersolution to the equation in \eqref{solv1} in $D$ if, whenever $(\hat X,\hat Y,\hat t)\in D$ and $\phi\in C^2(D)$ are such that
\begin{eqnarray*}
(1)&& u(\hat X,\hat Y,\hat t)=\phi(\hat X,\hat Y,\hat t),\notag\\
(2)&& u(X,Y,t)>\phi(X,Y,t)\mbox{ for all }(X,Y,t)\in D,\ (X,Y,t)\neq (\hat X,\hat Y,\hat t),
\end{eqnarray*}
then, at $(\hat X,\hat Y,\hat t)$,
\begin{eqnarray*}
(i)&& (\partial_t\phi-\hat X\cdot\nabla_Y\phi))\geq \Delta_{\infty,X}^N\phi,\mbox{ if $\nabla_X\phi(\hat X,\hat Y,\hat t)\neq 0$},\notag\\
(ii)&& (\partial_t\phi-\hat X\cdot\nabla_Y\phi))\geq \lambda(\nabla_X^2\phi),\mbox{ if $\nabla_X \phi(\hat X,\hat Y,\hat t)=0$}.
\end{eqnarray*}
A function $u\in \mbox{USC}(D)$ is
a viscosity subsolution to the equation in \eqref{solv1} in $D$ if, whenever \\
$(\hat X,\hat Y,\hat t)\in D$ and $\phi\in C^2(D)$ are such that
\begin{eqnarray*}
(1)&& u(\hat X,\hat Y,\hat t)=\phi(\hat X,\hat Y,\hat t),\notag\\
(2)&& u(X,Y,t)<\phi(X,Y,t)\mbox{ for all }(X,Y,t)\in D,\ (X,Y,t)\neq (\hat X,\hat Y,\hat t),
\end{eqnarray*}
then,  at $(\hat X,\hat Y,\hat t)$,
\begin{eqnarray*}
(i)&& (\partial_t\phi-\hat X\cdot\nabla_Y\phi))\leq \Delta_{\infty,X}^N\phi,\mbox{ if $\nabla_X \phi(\hat X,\hat Y,\hat t)\neq 0$},\notag\\
(ii)&& (\partial_t\phi-\hat X\cdot\nabla_Y\phi))\leq \Lambda(\nabla^2_X\phi),\mbox{ if $\nabla_X\phi(\hat X,\hat Y,\hat t)=0$}.
\end{eqnarray*}
A function $u\in C(D)$ is said to be a viscosity solution to
\eqref{solv1} in $D$ if it is both a viscosity supersolution and a viscosity subsolution in $D$.
\end{definition}

\begin{definition}\label{Vissol} Let $D\subset\mathbb R^{M+1}$ be an open set and consider $p$, $1<p<\infty$. A function $u\in \mbox{LSC}(D)$ is
a viscosity supersolution to the equation in \eqref{solv1+-} in $D$ if, whenever $(\hat X,\hat Y,\hat t)\in D$ and $\phi\in C^2(D)$ are such that
\begin{eqnarray*}
(1)&& u(\hat X,\hat Y,\hat t)=\phi(\hat X,\hat Y,\hat t),\notag\\
(2)&& u(X,Y,t)>\phi(X,Y,t)\mbox{ for all }(X,Y,t)\in D,\ (X,Y,t)\neq (\hat X,\hat Y,\hat t),
\end{eqnarray*}
then, at $(\hat X,\hat Y,\hat t)$,
\begin{eqnarray*}
(i)&& (m+p)(\partial_t\phi-\hat X\cdot\nabla_Y\phi))\geq ((p-2)\Delta_{\infty,X}^N+\Delta_X)\phi,\mbox{ if $\nabla_X\phi(\hat X,\hat Y,\hat t)\neq 0$},\notag\\
(ii)&& (m+p)(\partial_t\phi-\hat X\cdot\nabla_Y\phi))\geq \lambda((p-2)\nabla_X^2\phi)+\Delta_X\phi,\mbox{ if $\nabla_X \phi(\hat X,\hat Y,\hat t)=0$}.
\end{eqnarray*}
A function $u\in \mbox{USC}(D)$ is
a viscosity subsolution to the equation in \eqref{solv1+-} in $D$ if, whenever $(\hat X,\hat Y,\hat t)\in D$ and $\phi\in C^2(D)$ are such that
\begin{eqnarray*}
(1)&& u(\hat X,\hat Y,\hat t)=\phi(\hat X,\hat Y,\hat t),\notag\\
(2)&& u(X,Y,t)<\phi(X,Y,t)\mbox{ for all }(X,Y,t)\in D,\ (X,Y,t)\neq (\hat X,\hat Y,\hat t),
\end{eqnarray*}
then,  at $(\hat X,\hat Y,\hat t)$,
\begin{eqnarray*}
(i)&& (m+p)(\partial_t\phi-\hat X\cdot\nabla_Y\phi))\leq ((p-2)\Delta_{\infty,X}^N+\Delta_X)\phi,\mbox{ if $\nabla_X \phi(\hat X,\hat Y,\hat t)\neq 0$},\notag\\
(ii)&& (m+p)(\partial_t\phi-\hat X\cdot\nabla_Y\phi))\leq \Lambda((p-2)\nabla^2_X\phi)+\Delta_X\phi,\mbox{ if $\nabla_X\phi(\hat X,\hat Y,\hat t)=0$}.
\end{eqnarray*}
A function $u\in C(D)$ is said to be a viscosity solution to
\eqref{solv1+-} in $D$ if it is both a viscosity supersolution and a viscosity subsolution in $D$.
\end{definition}

We will, at instances, find it convenient to use the following lemma which states that we can further reduce the number of test in the definition of viscosity solutions.

\begin{lemma}\label{Vissolsimp} Let $D\subset\mathbb R^{M+1}$ be an open set and consider $p$, $1 < p < \infty$. A function $u\in C(D)$ is
a viscosity solution to \eqref{solv1+-} in $D$ if the following two conditions hold.

If $(\hat X,\hat Y,\hat t)\in D$ and $\phi\in C^2(D)$ are such that
\begin{eqnarray*}
(1)&& u(\hat X,\hat Y,\hat t)=\phi(\hat X,\hat Y,\hat t),\notag\\
(2)&& u(X,Y,t)>\phi(X,Y,t)\mbox{ for all }(X,Y,t)\in D,\ (X,Y,t)\neq (\hat X,\hat Y,\hat t),
\end{eqnarray*}
then, at $(\hat X,\hat Y,\hat t)$
\begin{eqnarray*}
(i)&& (m+p)(\partial_t\phi-\hat X\cdot\nabla_Y\phi))\geq ((p-2)\Delta_{\infty,X}^N+\Delta_X)\phi,\mbox{ if $\nabla_X\phi(\hat X,\hat Y,\hat t)\neq 0$},\notag\\
(ii)&& (m+p)(\partial_t\phi-\hat X\cdot\nabla_Y\phi))\geq 0,\mbox{ if $\nabla_X \phi(\hat X,\hat Y,\hat t)=0$ and $\nabla_X^2 \phi(\hat X,\hat Y,\hat t)=0$}.
\end{eqnarray*}

If $(\hat X,\hat Y, \hat t)\in D$ and $\phi\in C^2(D)$ are such that
\begin{eqnarray*}
(1)&& u(\hat X,\hat Y,\hat t)=\phi(\hat X,\hat Y,\hat t),\notag\\
(2)&& u(X,Y,t)<\phi(X,Y,t)\mbox{ for all }(X,Y,t)\in D,\ (X,Y,t)\neq (\hat X,\hat Y,\hat t),
\end{eqnarray*}
then,  at $(\hat X,\hat Y,\hat t)$,
\begin{eqnarray*}
(i)&& (m+p)(\partial_t\phi-\hat X\cdot\nabla_Y\phi))\leq ((p-2)\Delta_{\infty,X}^N+\Delta_X)\phi,\mbox{ if $\nabla_X \phi(\hat X,\hat Y,\hat t)\neq 0$},\notag\\
(ii)&& (m+p)(\partial_t\phi-\hat X\cdot\nabla_Y\phi))\leq 0,\mbox{ if $\nabla_X\phi(\hat X,\hat Y,\hat t)=0$ and $\nabla_X^2 \phi(\hat X,\hat Y,\hat t)=0$}.
\end{eqnarray*}
The analogous conclusions are valid in the case $p=\infty$.
\end{lemma}
\begin{proof} We only supply the proof of the lemma in the case $1<p<\infty$, as the proof in the case $p=\infty$ is analogous. First, we focus on the case $p\geq 2$ and, at the end of the proof, we explain  how the argument can be modified to work also in the case $p<2$. Assume that $u\in C(D)$ is such that the conditions stated in the lemma are true but that $u$ is not a viscosity solution to \eqref{solv1+-} in $D$ in the sense of Definition \ref{Vissol}.  Based on this assumption, we want to derive a contradiction.  Note that the only difference between the conditions in Lemma \ref{Vissolsimp} and the conditions in Definition \ref{Vissol} appears in $(ii)$ of Definition \ref{Vissol}.  As the conditions stated in the lemma are symmetric with respect to the test function touching from above and below, we can in the following assume, without loss of generality, that there exists $(\hat X,\hat Y,\hat t)\in D$, $\phi\in C^2(D)$, and $\eta>0$, such that
\begin{eqnarray*}
(1)&& u(\hat X,\hat Y,\hat t)=\phi(\hat X,\hat Y,\hat t),\notag\\
(2)&& u(X,Y,t)>\phi(X,Y,t)\mbox{ for all }(X,Y,t)\in D,\ (X,Y,t)\neq (\hat X,\hat Y,\hat t),
\end{eqnarray*}
 such that $\nabla_X\phi(\hat X,\hat Y,\hat t)=0$, $\nabla_X^2 \phi(\hat X,\hat Y,\hat t)\neq 0$, and
\begin{eqnarray}\label{contra1}
\quad (m+p)(\partial_t-\hat X\cdot\nabla_Y\phi)(\hat X,\hat Y,\hat t)< \lambda((p-2) \nabla_X^2\phi(\hat X,\hat Y,\hat t))+\Delta_X\phi(\hat X,\hat Y,\hat t)-\eta.
\end{eqnarray}
This is a consequence of Definition \ref{Vissol} and the assumption that $u$ is not a viscosity solution to \eqref{solv1+-} in $D$. We want to prove that this is impossible by deriving a contradiction based on this assumption.

Let $U$ be an open set containing $(\hat X, \hat Y, \hat t)$ so that $\bar U$ is compact and contained in $D$. Given
$(X,Y,t)$ and $(\tilde X,\tilde Y,\tilde t)$, we introduce the function $w_j :\bar U \times \bar U \to \R$,
\begin{align*}
w_j(X,Y,t,\tilde X,\tilde Y,\tilde t)&:=u(X,Y,t)-\phi(\tilde X,\tilde Y,\tilde t)\notag\\
&+\bigl(\frac {j^4}{4}|X-\tilde X|^4+\frac {j^4}{4}|Y-\tilde Y|^4+\frac j2|t-\tilde t|^2\bigr ),
\end{align*}
and we pick a point $(X_j,Y_j,t_j,\tilde X_j,\tilde Y_j,\tilde t_j)$ in $\bar U\times\bar U$ at which the minimum of $w_j$ is realized. As in Proposition 3.7 in \cite{crandallil92}, we then see that
\begin{eqnarray}\label{contra4-}
\frac {j^4}{4}|X_j-\tilde X_j|^4\rightarrow 0,\ \frac {j^4}{4}|Y_j-\tilde Y_j|^4\rightarrow 0,\mbox{ and }\frac j2|t_j-\tilde t_j|^2\rightarrow 0.
\end{eqnarray}
We claim that both $(X_j,Y_j,t_j)$ and $(\tilde X_j, \tilde Y_j, \tilde t_j)$ converge to $(\hat X, \hat Y, \hat t)$. To see this, suppose on the contrary that there is $r>0$ such that for example
\begin{align}\label{point_outside_ball}
 (X_j, Y_j, t_j) \notin B_r(\hat X, \hat Y, \hat t)
\end{align}
for arbitrarily large $j$. Here, $B_r(\hat X, \hat Y, \hat t)$ is the (standard) Euclidean ball of radius $r$ centered at $(\hat X, \hat Y, \hat t)$. Since $u-\phi$ vanishes at $(\hat X, \hat Y, \hat t)$ and since this is a strict minimum, there is $\varepsilon > 0$ such that $u-\phi > 2\varepsilon$ on $\complement B_r(\hat X, \hat Y, \hat t)$. By \eqref{contra4-}, the distance between $(X_j, Y_j, t_j)$ and $(\tilde X_j, \tilde Y_j, \tilde t_j)$ vanishes in the limit $j\to \infty$, so by the uniform continuity of $\phi$ in $\bar U$, we have
\begin{align}\label{phi_difference}
 |\phi(X_j,Y_j,t_j)-\phi(\tilde X_j, \tilde Y_j, \tilde t_j)| < \varepsilon,
\end{align}
for all large $j$. Thus, for arbitrarily large $j$ satisfying \eqref{point_outside_ball} and \eqref{phi_difference}, we have
\begin{align*}
 w_j(X_j,Y_j,t_j,\tilde X_j, \tilde Y_j, \tilde t_j) &\geq u(X_j, Y_j, t_j) -\phi(\tilde X_j, \tilde Y_j,\tilde t_j)
 \\
 &= u(X_j, Y_j, t_j) - \phi(X_j, Y_j, t_j) + \phi(X_j, Y_j, t_j) - \phi(\tilde X_j, \tilde Y_j,\tilde t_j)
 \\
 &> 2\varepsilon - \varepsilon
 \\
 &= \varepsilon.
\end{align*}
But this contradicts the definition of $(X_j,Y_j,t_j,\tilde X_j,\tilde Y_j,\tilde t_j)$ as the infimum of $w_j$ on $\bar U \times \bar U$: the infimum cannot be positive since $w_j$ vanishes at $(\hat X, \hat Y, \hat t, \hat X, \hat Y, \hat t)$. Similarly, one can treat the sequence $(\tilde X_j, \tilde Y_j, \tilde t_j)$, and we have
\begin{eqnarray*}
(X_j,Y_j,t_j,\tilde X_j,\tilde Y_j,\tilde t_j)\to (\hat X,\hat Y,\hat t,\hat X,\hat Y,\hat t)\mbox{ as $j\to\infty$}.
\end{eqnarray*}

Assume that $X_{j_l}=\tilde X_{j_l}$ for an infinite sequence $\{j_l\}_l$ with $j_l\geq j_0$. Let
\begin{eqnarray*}
\quad\varphi_{j_l}(\tilde X,\tilde Y,\tilde t):=\frac {{j_l^4}}{4}|X_{j_l}-\tilde X|^4+\frac {{j_l^4}}{4}|Y_{j_l}-\tilde Y|^4+\frac {j_l}2|t_{j_l}-\tilde t|^2.
\end{eqnarray*}
As
\begin{align*}
w_{j_l}(X_{j_l},Y_{j_l},t_{j_l},\tilde X,\tilde Y,\tilde t)&=u(X_{j_l},Y_{j_l},t_{j_l})-(\phi(\tilde X,\tilde Y,\tilde t)-\varphi_{j_l}(\tilde X,\tilde Y,\tilde t)),
\end{align*}
it follows by the definition of $(X_{j_l},Y_{j_l},t_{j_l},\tilde X_{j_l},\tilde Y_{j_l},\tilde t_{j_l})$ that
\begin{eqnarray*}
\phi(\tilde X,\tilde Y,\tilde t)-\varphi_{j_l}(\tilde X,\tilde Y,\tilde t)
\end{eqnarray*}
has a local maximum at $(\tilde X_{j_l},\tilde Y_{j_l},\tilde t_{j_l})$. Using \eqref{contra1} and continuity of the map
\begin{eqnarray*}
(\tilde X,\tilde Y,\tilde t)\to \lambda((p-2) \nabla_X^2\phi(\tilde X,\tilde Y,\tilde t))+\Delta_X\phi(\tilde X,\tilde Y,\tilde t),
\end{eqnarray*}
we deduce that
\begin{eqnarray}\label{contra7}
&&(m+p)(\partial_t-\tilde X_{j_l}\cdot\nabla_Y)\phi(\tilde X_{j_l},\tilde Y_{j_l},\tilde t_{j_l})\notag\\
&&< \lambda((p-2) \nabla_X^2 \phi(\tilde X_{j_l},\tilde Y_{j_l},\tilde t_{j_l}))+\Delta_X\phi(\tilde X_{j_l},\tilde Y_{j_l},\tilde t_{j_l})-\eta,
\end{eqnarray}
for all ${j_l}\geq {j_l}_0$. Using the definition of $(\tilde X_{j_l},\tilde Y_{j_l},\tilde t_{j_l})$ as a local maximum of $\phi-\varphi_{j_l}$, we have that
\begin{eqnarray*}
(\partial_t-\tilde X_{j_l}\cdot\nabla_Y)\phi(\tilde X_{j_l},\tilde Y_{j_l},\tilde t_{j_l})=(\partial_t-\tilde X_{j_l}\cdot\nabla_Y)\varphi_{j_l}(\tilde X_{j_l},\tilde Y_{j_l},\tilde t_{j_l})
\end{eqnarray*}
and that $\nabla_X^2\phi(\tilde X_{j_l},\tilde Y_{j_l},\tilde t_{j_l})\leq \nabla_X^2\varphi_{j_l}(\tilde X_{j_l},\tilde Y_{j_l},\tilde t_{j_l})$ if ${j_l}\geq {j_l}_0$. Using these observations and \eqref{contra7} and recalling that we consider the case $p\geq 2$, we deduce
\begin{align}\label{contra9}
\eta &< -(m+p)(\partial_t - \tilde X_{j_l} \cdot \nabla_Y) \varphi_{j_l}(\tilde X_{j_l},\tilde Y_{j_l},\tilde t_{j_l}) \notag
\\
&+ (p-2)\lambda( \nabla_X^2 \varphi_{j_l}(\tilde X_{j_l},\tilde Y_{j_l},\tilde t_{j_l}))+\Delta_X\varphi_{j_l}(\tilde X_{j_l},\tilde Y_{j_l},\tilde t_{j_l})\notag
\\
=&-(m+p)(\partial_t-\tilde X_{j_l}\cdot\nabla_Y)\varphi_{j_l}(\tilde X_{j_l},\tilde Y_{j_l},\tilde t_{j_l}),
\end{align}
where we have also used that $X_{j_l}=\tilde X_{j_l}$ and hence that $\nabla_X^2\varphi_{j_l}(\tilde X_{j_l},\tilde Y_{j_l},\tilde t_{j_l})=0$ by construction. Next, we let
\begin{eqnarray*}
\quad\psi_{j_l}(X,Y,t):=-\frac {{j_l^4}}{4}|X-\tilde X_{j_l}|^4-\frac {{j_l^4}}{4}|Y-\tilde Y_{j_l}|^4-\frac {j_l}2|t-\tilde t_{j_l}|^2.
\end{eqnarray*}
Then,
\begin{eqnarray*}
u(X,Y,t)-\psi_{j_l}(X,Y,t)
\end{eqnarray*}
has a local minimum at $(X_{j_l},Y_{j_l}, t_{j_l})$. In this case, we deduce, using that $\nabla_X^2\psi_{j_l}(X_{j_l},Y_{j_l},t_{j_l})=0$ and condition $(ii)$ of the lemma in the case of touching from below,
\begin{align}\label{contra9+}
0\leq (m+p)(\partial_t-X_{j_l}\cdot\nabla_Y)\psi_{j_l}(X_{j_l},Y_{j_l},t_{j_l}).
\end{align}
Therefore, summing \eqref{contra9} and \eqref{contra9+},
\begin{align}\label{contra10}
\eta&< -(m+p)(\partial_t-\tilde X_{j_l}\cdot\nabla_Y)\varphi_{j_l}(\tilde X_{j_l},\tilde Y_{j_l},\tilde t_{j_l})+(m+p)(\partial_t-X_{j_l}\cdot\nabla_Y)\psi_{j_l}(X_{j_l},Y_{j_l},t_{j_l})\notag\\
&={{j_l^4}}(\tilde X_{j_l}-X_{j_l})\cdot (\tilde Y_{j_l}-Y_{j_l})|\tilde Y_{j_l}-Y_{j_l}|^2
\\
\notag &=0,
\end{align}
since $\tilde X_{j_l} = X_{j_l}$.
Thus, \eqref{contra10} produces a contradiction and therefore either our original assumption must be incorrect, and then we are done, or $X_j\neq \tilde X_j$ for all $j\geq j_0$ and for some $j_0\gg 1$.

Based on the previous argument, we from now on assume that $X_j\neq \tilde X_j$ for all $j\geq j_0$. Denote
\begin{align*}
 \Psi(X,Y,t,\tilde X, \tilde Y, \tilde t) = \frac{j^4}{4}|X-\tilde X|^4 + \frac{j^4}{4}|Y - \tilde Y|^4 + \frac{j}{2}|t-\tilde t|^2,
\end{align*}
and recall that
\begin{align*}
 w_j(X,Y,t,\tilde X, \tilde Y, \tilde t) = u(X,Y,t) - \phi(\tilde X, \tilde Y, \tilde t) + \Psi(X,Y,t,\tilde X, \tilde Y, \tilde t)
\end{align*}
has a local minimum at $P_j := (X_j,Y_j,t_j,\tilde X_j,\tilde Y_j, \tilde t_j)$.  Since the map
\begin{align*}
 (\tilde X, \tilde Y, \tilde t) \mapsto w_j(X_j, Y_j, t_j, \tilde X, \tilde Y, \tilde t)
\end{align*}
has a minimum at $(\tilde X_j, \tilde Y_j, \tilde t_j)$, we obtain some useful relations between the derivatives of $\phi$ and $\Psi$ at the point $P_j$:
\begin{align*}
 (\partial_t \phi)(\tilde X_j, \tilde Y_j, \tilde t_j) &= \partial_{\tilde t} \Psi(P_j) = - \partial_t \Psi(P_j),
 \\
  (\nabla_X \phi)(\tilde X_j, \tilde Y_j, \tilde t_j) &= \nabla_{\tilde X} \Psi(P_j) = - \nabla_X\Psi(P_j),
 \\
 (\nabla_Y \phi)(\tilde X_j, \tilde Y_j, \tilde t_j) &= \nabla_{\tilde Y} \Psi(P_j) = - \nabla_Y\Psi(P_j).
\end{align*}

We now apply Theorem 3.2 in \cite{crandallil92} with the choices $w=-w_j$, $k=2$, $u_1 = -u$, $u_2 = \phi$, $\hat x = P_j$. Our function $\Psi$ corresponds to the function $\phi$ in Theorem 3.2 in \cite{crandallil92}. The conclusion is that for any $\varepsilon > 0$ we can find symmetric $(M+1)\times(M+1)$ matrices $E, H$ such that
\begin{align*}
  (-\nabla_{\tilde X,\tilde Y,\tilde t}\Psi(P_j), H) = (\nabla_{X,Y,t}\Psi(P_j), H) &\in \overline{J}^{2,+}(-u)(X_j,Y_j,t_j) = -\overline{J}^{2,-}u(X_j,Y_j,t_j),
 \\
 (\nabla_{\tilde X, \tilde Y, \tilde t}\Psi(P_j), E) &\in \overline{J}^{2,+} \phi(\tilde X_j, \tilde Y_j, \tilde t_j),
\end{align*}
and
\begin{align}\label{matrix_estimate_thm_of_sums}
 \begin{pmatrix}
H & 0\\
0 & E
\end{pmatrix}
\leq A + \varepsilon A^2,
\end{align}
where $A = \nabla^2\Psi(P_j)=\nabla_{X,Y,t}^2\Psi(P_j)$. As we shall see, the choice of $\varepsilon$ in our case is not important. Denoting $M_1 = \nabla_X^2\Psi(P_j)$ and $M_2=\nabla_Y^2\Psi(P_j)$, a direct calculation shows that
\begin{align*}
 \nabla^2\Psi(P_j) =  \begin{pmatrix}
M_1 & 0 & 0 & -M_1 & 0 & 0 \\
0 & M_2 & 0 & 0 & -M_2 & 0 \\
0 & 0 & j & 0 & 0 & -j \\
-M_1 & 0 & 0 & M_1 & 0 & 0 \\
0 & -M_2 & 0 & 0 & M_2 & 0 \\
0 & 0 & -j & 0 & 0 & j
\end{pmatrix},
\end{align*}
and
\begin{align*}
 (\nabla^2\Psi(P_j))^2 =  2\begin{pmatrix}
M_1^2 & 0 & 0 & -M_1^2 & 0 & 0 \\
0 & M_2^2 & 0 & 0 & -M_2^2 & 0 \\
0 & 0 & j^2 & 0 & 0 & -j^2 \\
-M_1^2 & 0 & 0 & M_1^2 & 0 & 0 \\
0 & -M_2^2 & 0 & 0 & M_2^2 & 0 \\
0 & 0 & -j^2 & 0 & 0 & j^2
\end{pmatrix}.
\end{align*}
Both $A$ and $A^2$ map any vector of the form $(v, v)$ where $v\in \R^{M+1}$ to zero, and thus \eqref{matrix_estimate_thm_of_sums} implies that
\begin{align*}
 v^T(H + E)v \leq 0,
\end{align*}
for all $v\in \R^{M+1}$. Setting suitable components of $v$ to zero, one sees that any principal subminor of $H+E$ satisfies the same type of condition. By the definition of the sets $\bar{J}^{2,+}$, we know that there exists a sequence $(\tilde X_j^k, \tilde Y_j^k, \tilde t_j^k)$ converging to $(\tilde X_j, \tilde Y_j, \tilde t_j)$ and elements $(\xi^k_j, E_k) \in J^{2,+}\phi(\tilde X_j^k, \tilde Y_j^k, \tilde t_j^k)$ such that $(\xi^k_j, E_k) \to (\nabla_{\tilde X,\tilde Y, \tilde t}\Psi(P_j), E)$. Since $\phi$ is smooth, we know by the basic properties of superjets that
\begin{align*}
 \nabla_X^2\phi(\tilde X_j^k, \tilde Y_j^k, \tilde t_j^k) \leq E^X_k,
\end{align*}
where $E^X_k$ refers to the subminor of $E$ corresponding to the $X$-coordinates.

Since $(\tilde X_j,\tilde Y_j, \tilde t_j)$ converges to $(\hat X, \hat Y, \hat t)$ we can deduce as in the proof of (2.16) that
\begin{align}\label{basic_estimate}
 (m+p)(\partial_t - \tilde X_j\cdot \nabla_Y)\phi(\tilde X_j, \tilde Y_j, \tilde t_j) &< \lambda((p-2)\nabla_X^2\phi(\tilde X_j, \tilde Y_j, \tilde t_j) )
 \\
 \notag &\hphantom{<} + \Delta_X \phi(\tilde X_j, \tilde Y_j, \tilde t_j) - \eta,
\end{align}
for sufficiently large $j$. Thus, for sufficiently large $k$ we also have
\begin{align}\label{eta_est_p_geq_2}
 \notag \eta <& -(m+p)(\partial_t - \tilde X_j^k\cdot \nabla_Y)\phi(\tilde X_j^k, \tilde Y_j^k, \tilde t_j^k) + \lambda((p-2)\nabla_X^2\phi(\tilde X_j^k, \tilde Y_j^k, \tilde t_j^k) )\notag\\
 & + \Delta_X \phi(\tilde X_j^k, \tilde Y_j^k, \tilde t_j^k)\notag\\
 \leq & -(m+p)(\partial_t - \tilde X_j^k\cdot \nabla_Y)\phi(\tilde X_j^k, \tilde Y_j^k, \tilde t_j^k) + (p-2)\lambda(E^X_k) + \textrm{tr}(E^X_k),
\end{align}
where we also used the fact that we consider the case $p\geq 2$. Passing to the limit $k\to \infty$ and using the relations between first order derivatives of $\phi$ and $\Psi$ at $P_j$ yields
\begin{align}\label{hehu1}
 \eta &\leq -(m+p)(\partial_t - \tilde X_j\cdot \nabla_Y)\phi(\tilde X_j, \tilde Y_j, \tilde t_j) + (p-2)\lambda(E^X) + \textrm{tr}(E^X)
 \\
 \notag &= -(m+p)(\partial_{\tilde t} - \tilde X_j\cdot \nabla_{\tilde Y})\Psi(P_j) + (p-2)\lambda(E^X) + \textrm{tr}(E^X).
\end{align}
Similarly, we find a sequence $(X^k_j, Y^k_j, t^k_j)$ converging to $(X_j, Y_j, t_j)$ and elements $(q_k, -H_k)$ belonging to $J^{2,-} u(X^k_j, Y^k_j, t^k_j)$ such that $(q_k, -H_k) \to (\nabla_{\tilde X, \tilde Y, \tilde t}\Psi(P_j), -H)$.
Since $u$ is not necessarily smooth, utilizing this fact to get some estimate involving $H$ is not as straight-forward as in the previous case. Here we use an observation made in \cite{crandallil92}. Namely, one can always find a $C^2$-function $\zeta^k_j$ touching $u$ from below such that
\begin{align*}
(\nabla \zeta^k_j(X^k_j, Y^k_j, t^k_j), \nabla^2 \zeta^k_j(X^k_j, Y^k_j, t^k_j)) = (q_k, -H_k).
\end{align*}
Since $X_j \neq \tilde X_j$ and since $q_k \to \nabla_{\tilde X, \tilde Y, \tilde t}\Psi(P_j)$, we see that for $k$ sufficiently large, $\nabla_X \zeta^k_j(X^k_j, Y^k_j, t^k_j) \neq 0$ and thus, by property (i) in the statement of the lemma we have
\begin{align}\label{zeta_test_funct_est}
 \notag (m+p)(\partial_t - X^k_j\cdot \nabla_Y)\zeta^k_j(X^k_j,Y^k_j,t^k_j) &\geq (p-2)\Delta^N_{\infty,X} \zeta^k_j(X^k_j,Y^k_j,t^k_j) + \Delta_X \zeta^k_j(X^k_j,Y^k_j,t^k_j)
 \\
 \notag &\geq (p-2) \lambda( \nabla^2_X \zeta^k_j(X^k_j,Y^k_j,t^k_j) ) + \Delta_X \zeta^k_j (X^k_j,Y^k_j,t^k_j)
 \\
  &= (p-2)\lambda (-H_k^X) - \textrm{tr}(H_k^X).
\end{align}
Passing to the limit $k\to \infty$, we end up with
\begin{align}\label{hehu2}
 0 \leq (m+p)(\partial_{\tilde t} - X_j\cdot \nabla_{\tilde Y})\Psi(P_j) - (p-2)\lambda(-H^X) + \textrm{tr}(H^X).
\end{align}
Adding \eqref{hehu1} and \eqref{hehu2}, using the fact that $E^X + H^X$ is negative semidefinite and applying Young's inequality, we obtain
\begin{align*}
 \eta &\leq (m+p)(\tilde X_j - X_j)\cdot \nabla_{\tilde Y}\Psi(P_j) + (p-2)(\lambda(E^X) - \lambda(-H^X)) + \textrm{tr}(E^X + H^X)
 \\
 &\leq (m+p)j^4(\tilde X_j - X_j) \cdot (\tilde Y_j - Y_j)|\tilde Y_j - Y_j|^2
 \\
 &\leq c(j^4|\tilde X_j - X_j|^4 + j^4|\tilde Y_j - Y_j|^4).
\end{align*}
By \eqref{contra4-}, the right-hand side converges to zero as $j\to 0$. Since $\eta >0$, this is a contradiction.

In the case $p<2$, some modifications are needed as \eqref{contra9}, \eqref{eta_est_p_geq_2} and \eqref{zeta_test_funct_est} are only valid in their present form if $p\geq 2$. First, note that for any symmetric $(m \times m)$ matrix $B$ with ordered eigenvalues $\lambda_i(B)$, with $\lambda_m(B)$ being the largest, we have
\begin{align*}
 \lambda((p-2)B) + \textrm{tr} B = (p-2) \lambda_m(B) + \sum^m_{i=1}\lambda_i(B) = (p-1) \lambda_m(B) + \sum^{m-1}_{i=1} \lambda_i(B).
\end{align*}
Recalling that the Loewner order of symmetric matrices implies the same order for the eigenvalues, we replace the estimate of terms in \eqref{eta_est_p_geq_2} by
\begin{align*}
\lambda((p-2)\nabla_X^2\phi(\tilde X_j^k, \tilde Y_j^k, \tilde t_j^k) ) + \Delta_X \phi(\tilde X_j^k, \tilde Y_j^k, \tilde t_j^k) \leq (p-1)\Lambda(E^X_k) + \sum^{m-1}_{i=1} \lambda_i(E^X_k).
\end{align*}
A similar reasoning can be utilized in the case of \eqref{contra9}. In the estimate \eqref{zeta_test_funct_est} we instead must proceed as follows:
\begin{align*}
 (p-2)&\Delta^N_{\infty,X} \zeta^k_j(X^k_j,Y^k_j,t^k_j) + \Delta_X \zeta^k_j(X^k_j,Y^k_j,t^k_j)
 \\
 &\geq (p-2)\Lambda(\nabla_X^2 \zeta^k_j(X^k_j,Y^k_j,t^k_j)) + \Delta_X \zeta^k_j(X^k_j,Y^k_j,t^k_j)
 \\
 &= (p-2)\Lambda(-H^X_k) + \textrm{tr}(-H^X_k)
 \\
 &= (p-1)\Lambda(-H^X_k) + \sum^{m-1}_{i=1}\lambda_i(-H^X_k).
\end{align*}
After passing to the limit $k\to \infty$, we can then utilize the fact that $\lambda_i(E^X) \leq \lambda_i(-H^X)$ for all $i \in \{1,\dots,m\}$ to arrive at the same contradiction as in the previous case.
\end{proof}

\subsection{Asymptotic mean value formulas in the viscosity sense}

Let $p$, $1<p\leq\infty$. In the following, we let
\begin{align}\label{Conv}
\alpha=1, \beta=0, \textrm{ if } p=\infty, \textrm{ and } \alpha=\frac{p-2}{m+p}, \beta=\frac{m+2}{m+p}, \textrm{ if } p<\infty.
\end{align}
Then $\alpha+\beta=1$ for all $p$, $1<p\leq\infty$. Throughout the paper we will use the convention that $\alpha$ and $\beta$ are defined according to \eqref{Conv}.

\begin{definition}\label{meanvalueviscosity} Let $D\subset\mathbb R^{M+1}$ be an open set and consider $p$, $1<p\leq\infty$. Let $\alpha$ and $\beta$ be defined as in \eqref{Conv}.  Let $u\in C(D)$. We say that $u$ satisfies the asymptotic mean value formula
\begin{align*}
u(X,Y,t)&=\frac \alpha 2\barint_{B_{\epsilon^3}(Y)} \barint_{t-\epsilon^2}^t \biggl\{\sup_{\tilde X\in{{B_\epsilon(X)}}}u(\tilde X,\tilde Y-(\tilde t-t)X,\tilde t)\biggr \}\d \tilde Y\d \tilde t\notag\\
&+\frac \alpha 2\barint_{B_{\epsilon^3}(Y)} \barint_{t-\epsilon^2}^t
\biggl\{\inf_{\tilde X\in{{B_\epsilon(X)}}}u(\tilde X,\tilde Y-(\tilde t-t)X,\tilde t)\biggr \}\d \tilde Y\d \tilde t\notag\\
&+\beta \barint_{B_\epsilon(X)} \barint_{B_{\epsilon^3}(Y)} \barint_{t-\epsilon^2}^t u(\tilde X,\tilde Y-(\tilde t-t)X,\tilde t) \d \tilde X \d \tilde Y\d \tilde t\notag\\
&+{o}(\epsilon^2),\mbox{ as }\epsilon\to 0,
\end{align*}
in the viscosity sense at $(X,Y,t)\in D$, if for every $\phi$ as in Lemma \ref{Vissolsimp} and touching $u$ from below, we have
\begin{align}\label{meanvalue3++}
\phi(X,Y,t)&\geq\frac \alpha 2\barint_{B_{\epsilon^3}(Y)} \barint_{t-\epsilon^2}^t \biggl\{\sup_{\tilde X\in{{B_\epsilon(X)}}}\phi(\tilde X,\tilde Y-(\tilde t-t)X,\tilde t)\biggr \}\d \tilde Y\d \tilde t\notag\\
&+\frac \alpha 2\barint_{B_{\epsilon^3}(Y)} \barint_{t-\epsilon^2}^t
\biggl\{\inf_{\tilde X\in{{B_\epsilon(X)}}}\phi(\tilde X,\tilde Y-(\tilde t-t)X,\tilde t)\biggr \}\d \tilde Y\d \tilde t\notag\\
&+\beta \barint_{B_\epsilon(X)} \barint_{B_{\epsilon^3}(Y)} \barint_{t-\epsilon^2}^t \phi(\tilde X,\tilde Y-(\tilde t-t)X,\tilde t) \d \tilde X \d \tilde Y\d \tilde t\notag\\
&+{o}(\epsilon^2),\mbox{ as }\epsilon\to 0,
\end{align}
and, if for every $\phi$ as in Lemma \ref{Vissolsimp} and touching $u$ from above, we have
\begin{align}\label{meanvalue3++gg}
\phi(X,Y,t)&\leq\frac \alpha 2\barint_{B_{\epsilon^3}(Y)} \barint_{t-\epsilon^2}^t \biggl\{\sup_{\tilde X\in{{B_\epsilon(X)}}}\phi(\tilde X,\tilde Y-(\tilde t-t)X,\tilde t)\biggr \}\d \tilde Y\d \tilde t\notag\\
&+\frac \alpha 2\barint_{B_{\epsilon^3}(Y)} \barint_{t-\epsilon^2}^t
\biggl\{\inf_{\tilde X\in{{B_\epsilon(X)}}}\phi(\tilde X,\tilde Y-(\tilde t-t)X,\tilde t)\biggr \}\d \tilde Y\d \tilde t\notag\\
&+\beta \barint_{B_\epsilon(X)} \barint_{B_{\epsilon^3}(Y)} \barint_{t-\epsilon^2}^t \phi(\tilde X,\tilde Y-(\tilde t-t)X,\tilde t) \d \tilde X \d \tilde Y\d \tilde t\notag\\
&+{o}(\epsilon^2),\mbox{ as }\epsilon\to 0.
\end{align}
\end{definition}
Note that in the above definition, we only consider the type of test functions used in Lemma \ref{Vissolsimp}. That is, at the point $(X,Y,t)$ under consideration, we only require the conditions to hold for test functions with $\nabla_X \phi(X,Y,t) \neq 0$ and the test functions for which both $\nabla_X \phi(X,Y,t) = 0$ and $\nabla_X^2 \phi(X,Y,t)=0$ hold. This is important in order for Theorem \ref{meanvalue1thm} below to be valid.

\section{Asymptotic mean-value properties}\label{sec2}
In this section we show a connection between viscosity solutions and asymptotic mean value formulas. Our starting point is the following result for $C^2$-solutions in the case $p=2$.
\begin{lemma}\label{lem:mean_value_p=2}
 Let $D\subset\mathbb R^{M+1}$. Then a function $u\in C^2(D)$ satisfies the asymptotic mean value formula
 \begin{align*}
  u(X,Y,t) = \barint_{B_\epsilon(X)} \barint_{B_{\epsilon^3}(Y)} \barint_{t-\epsilon^2}^t u(\tilde X,\tilde Y-(\tilde t-t)X,\tilde t) \d \tilde X \d \tilde Y\d \tilde t + {o}(\epsilon^2),\mbox{ as }\epsilon\to 0,
 \end{align*}
in the classical sense if and only if $u$ is a classical solution to
\begin{align*}
 \mathcal{K}_2 u(X,Y,t) = 0 \mbox { in }D.
\end{align*}
\end{lemma}
\begin{proof}
 Consider $(\tilde X,\tilde Y,\tilde t)\in {B_\epsilon(X)}\times{B_{\epsilon^3}(Y)}\times({t-\epsilon^2},t)$. Let $u \in C^2(D)$. Then by Taylor's formula at $(X,Y,t)$, we have
\begin{align*}
u(\tilde X,\tilde Y-(\tilde t-t)X,\tilde t)&= u(X,Y,t)+\nabla_Xu(X,Y,t)\cdot(\tilde X-X)\notag
\\
&+\frac 1 2 \langle \nabla_X^2u(X,Y,t)(\tilde X-X), (\tilde X-X)\rangle\notag
\\
&+\nabla_Y u(X,Y,t)\cdot (\tilde Y-Y-(\tilde t-t)X)+\partial_tu(X,Y,t)(\tilde t-t)+ {o}(\epsilon^2)\notag\\
&=u(X,Y,t)+\nabla_Xu(X,Y,t)\cdot(\tilde X-X)\notag\\
&+\frac 1 2 \langle \nabla_X^2u(X,Y,t)(\tilde X-X), (\tilde X-X)\rangle\notag\\
&-(\tilde t-t)\bigl (X\cdot\nabla_Y-\partial_t\bigr)u(X,Y,t)+ {o}(\epsilon^2), \mbox{ as }\epsilon\to 0.
\end{align*}
We  intend to take the average in the above display with respect to $(\tilde X,\tilde Y,\tilde t)\in {B_\epsilon(X)}\times{B_{\epsilon^3}(Y)}\times({t-\epsilon^2},t)$. Doing so, we see by symmetry that the contribution from the term $\nabla_Xu(X,Y,t)\cdot(\tilde X-X)$ is zero. Furthermore, by the same reason,
\begin{align*}
 &\frac 1 2 \barint_{B_\epsilon(X)} \barint_{B_{\epsilon^3}(Y)} \barint_{t-\epsilon^2}^t \langle \nabla_X^2u(X,Y,t)(\tilde X-X), (\tilde X-X)\rangle\d \tilde X \d \tilde Y\d \tilde t\notag\\
 &=\frac {\epsilon^2}{2(m+2)}\Delta_X u(X,Y,t)+{o}(\epsilon^2).
\end{align*}
In addition,
\begin{eqnarray*}
 &&\barint_{B_\epsilon(X)} \barint_{B_{\epsilon^3}(Y)} \barint_{t-\epsilon^2}^t (\tilde t-t) \d \tilde X \d \tilde Y\d \tilde t=-\frac {\epsilon^2}{2}.
\end{eqnarray*}
Put together, we deduce that
\begin{align}\label{meanvalue3aaaa}
 &\barint_{B_\epsilon(X)} \barint_{B_{\epsilon^3}(Y)} \barint_{t-\epsilon^2}^t u(\tilde X,\tilde Y-(\tilde t-t)X,\tilde t) \d \tilde X \d \tilde Y\d \tilde t\notag\\
 &=u(X,Y,t)+\frac {\epsilon^2}{2(m+2)}\mathcal{K}_2 u(X,Y,t)+{o}(\epsilon^2),\mbox{ as }\epsilon\to 0.
\end{align}
This holds for any $C^2$-function $u$. In particular, assume that $\mathcal{K}_2 u(X,Y,t)=0$. Then the asymptotic mean value formula holds. Assuming instead that the asymptotic mean value formula holds, we see that $\mathcal{K}_2 u(X,Y,t)=0$.
\end{proof}

\begin{theorem}\label{meanvalue1thm} Let $D\subset\mathbb R^{M+1}$ be an open set and consider $p$, $1<p\leq\infty$. Let $\alpha$ and $\beta$ be defined as
in \eqref{Conv}. Let $u\in C(D)$. The asymptotic mean value formula
\begin{align*}
u(X,Y,t) &= \frac \alpha 2\barint_{B_{\epsilon^3}(Y)} \barint_{t-\epsilon^2}^t \biggl\{\sup_{\tilde X\in{{B_\epsilon(X)}}}u(\tilde X,\tilde Y-(\tilde t-t)X,\tilde t)\biggr \}\d \tilde Y\d \tilde t \notag
\\
&+\frac \alpha 2\barint_{B_{\epsilon^3}(Y)} \barint_{t-\epsilon^2}^t
\biggl\{\inf_{\tilde X\in{{B_\epsilon(X)}}}u(\tilde X,\tilde Y-(\tilde t-t)X,\tilde t)\biggr \}\d \tilde Y\d \tilde t\notag
\\
&+\beta \barint_{B_\epsilon(X)} \barint_{B_{\epsilon^3}(Y)} \barint_{t-\epsilon^2}^t u(\tilde X,\tilde Y-(\tilde t-t)X,\tilde t) \d \tilde X \d \tilde Y\d \tilde t\notag
\\
&+{o}(\epsilon^2),\mbox{ as }\epsilon\to 0,
\end{align*}
holds for every $(X,Y,t)\in D$ in the viscosity sense  if  and only if $u$ is a viscosity solution to
\begin{eqnarray}\label{eq:K_pu=0}
  \K_p u(X,Y,t)=0\mbox { in }D.
\end{eqnarray}
\end{theorem}
\begin{proof} We give detailed proofs in the cases $p=2$ and $p=\infty$. In the end we, explain how the same techniques can be utilized also to treat the other cases $1<p<\infty$.

{\bf Proof in the case $\boldsymbol{p=2}$}. Suppose that $u\in C(D)$ is a viscosity solution to \eqref{eq:K_pu=0} in the case $p=2$. Let $\phi$ be a smooth function touching $u$ at $(X,Y,t)$ from below. Then, by Definition \ref{Vissol},
\begin{align*}
 \mathcal{K}_2 \phi (X,Y,t) \leq 0.
\end{align*}
Note that the expression \eqref{meanvalue3aaaa} was proved for any $C^2$-function and that especially holds for $\phi$. Thus,
\begin{align}\label{phi_est}
 \notag \barint_{B_\epsilon(X)} \barint_{B_{\epsilon^3}(Y)} \barint_{t-\epsilon^2}^t \phi(\tilde X,\tilde Y-(\tilde t-t)X,\tilde t) \d \tilde X \d \tilde Y\d \tilde t &=\phi(X,Y,t)+\frac {\epsilon^2}{2(m+2)}\mathcal{K}_2 \phi(X,Y,t)+{o}(\epsilon^2)
 \\
 &\leq \phi(X,Y,t) + {o}(\epsilon^2),
\end{align}
i.e., we have verified \eqref{meanvalue3++}. For a test function $\phi$ touching $u$ from above, we similarly obtain the reverse estimate \eqref{meanvalue3++gg}. Thus, we have verified that $u$ satisfies the asymptotic mean value formula in the viscosity sense. Conversely, suppose that $u$ satisfies the asymptotic mean value formula in the viscosity sense. If $\phi$ is a test function touching $u$ from below, we may use \eqref{meanvalue3aaaa} with $u=\phi$ and the estimate \eqref{meanvalue3++} to conclude that
\begin{align*}
 \frac {\epsilon^2}{2(m+2)}\mathcal{K}_2 \phi(X,Y,t) &= \barint_{B_\epsilon(X)} \barint_{B_{\epsilon^3}(Y)} \barint_{t-\epsilon^2}^t \phi(\tilde X,\tilde Y-(\tilde t-t)X,\tilde t) \d \tilde X \d \tilde Y\d \tilde t - \phi(X,Y,t) +{o}(\epsilon^2)
 \\
 &\leq {o}(\epsilon^2).
\end{align*}
Dividing by $\epsilon^2$ and passing to the limit $\epsilon \to 0$, we end up with $\mathcal{K}_2 \phi (X,Y,t) \leq 0$. An analogous argument shows that $\mathcal{K}_2 \phi (X,Y,t) \geq 0$ holds for test functions $\phi$ touching $u$ from above, so $u$ is a viscosity solution.

{\bf Proof in the case $\boldsymbol{p=\infty}$}. Let $\phi$ be a test function. Consider $(\tilde Y,\tilde t)\in {B_{\epsilon^3}(Y)}\times({t-\epsilon^2},t)$  and let $X_1:=X_1^{\epsilon,\tilde Y,\tilde t} \in \overline{B_\epsilon(X)}$ be a point such that,
\begin{align*}
\phi(X_1, \tilde Y-(\tilde t-t)X,\tilde t)&=\min_{\tilde X\in\overline{{B_\epsilon(X)}}}\phi(\tilde X,\tilde Y-(\tilde t-t)X,\tilde t)=\inf_{\tilde X\in{{B_\epsilon(X)}}}\phi(\tilde X,\tilde Y-(\tilde t-t)X,\tilde t).
\end{align*}
Again, using Taylor's formula at $(X,Y,t)$ we deduce
\begin{align*}
\phi(X_1,\tilde Y-(\tilde t-t)X,\tilde t)
&=\phi(X,Y,t)+\nabla_X\phi(X,Y,t)\cdot(X_1-X)\notag\\
&+\frac 1 2 \langle \nabla_X^2\phi(X,Y,t)(X_1-X), (X_1-X)\rangle\notag\\
&-(\tilde t-t)\bigl (X\cdot\nabla_Y-\partial_t\bigr)\phi(X,Y,t)+ {o}(\epsilon^2),
\end{align*}
as $\epsilon\to 0$. Similarly, evaluating the Taylor expansion also at $\tilde X_1=2X-X_1$ and adding the two expansions, we see that
\begin{align*}
&\phi(\tilde X_1,\tilde Y-(\tilde t-t)X,\tilde t)+\phi(X_1,\tilde Y-(\tilde t-t) X,\tilde t)-2\phi(X,Y,t)\notag\\
&= \langle \nabla_X^2\phi(X,Y,t)(X_1-X), (X_1-X)\rangle-2(X\cdot\nabla_Y-\partial_t)\phi(X,Y,t)(\tilde t-t)+ {o}(\epsilon^2), \mbox{ as }\epsilon\to 0.
\end{align*}
Using that $X_1$ is the point where minimum occurs,
\begin{align*}
&\phi(\tilde X_1,\tilde Y-(\tilde t-t)X,\tilde t)+\phi(X_1,\tilde Y-(\tilde t-t)X,\tilde t)-2\phi(X,Y,t)\notag\\
&\leq \max_{\tilde X\in\overline{{B_\epsilon(X)}}}\phi(\tilde X,\tilde Y-(\tilde t-t)X,\tilde t)+\min_{\tilde X\in\overline{{B_\epsilon(X)}}}\phi(\tilde X,\tilde Y-(\tilde t-t)X,\tilde t)-2\phi(X,Y,t),
\end{align*}
and we deduce
\begin{align*}
& \max_{\tilde X\in\overline{{B_\epsilon(X)}}}\phi(\tilde X,\tilde Y-(\tilde t-t)X,\tilde t)+\min_{\tilde X\in\overline{{B_\epsilon(X)}}}\phi(\tilde X,\tilde Y-(\tilde t-t)X,\tilde t)-2\phi(X,Y,t)\notag\\
&\geq \langle \nabla_X^2\phi(X,Y,t)(X_1-X), (X_1-X)\rangle-2(X\cdot\nabla_Y-\partial_t)\phi(X,Y,t)(\tilde t-t)+{o}(\epsilon^2).
\end{align*}
Hence, taking averages
\begin{align}\label{imp}
&\frac 1 2\barint_{B_{\epsilon^3}(Y)} \barint_{t-\epsilon^2}^t \biggl\{\max_{\tilde X\in\overline{{B_\epsilon(X)}}}\phi(\tilde X,\tilde Y-(\tilde t-t)X,\tilde t)\biggr \}\d \tilde Y\d \tilde t\notag\\
&+\frac 1 2\barint_{B_{\epsilon^3}(Y)} \barint_{t-\epsilon^2}^t
\biggl\{\min_{\tilde X\in\overline{{B_\epsilon(X)}}}\phi(\tilde X,\tilde Y-(\tilde t-t)X,\tilde t)\biggr \}\d \tilde Y\d \tilde t\notag\\
&-\phi(X,Y,t)\notag\\
&\geq \frac {\epsilon^2}2\barint_{B_{\epsilon^3}(Y)} \barint_{t-\epsilon^2}^t \langle \nabla_X^2\phi(X,Y,t)(\frac{X_1^{\epsilon,\tilde Y,\tilde t}-X}{\epsilon}), (\frac{X_1^{\epsilon,\tilde Y,\tilde t}-X}{\epsilon})\rangle\, \d \tilde Y\d \tilde t\notag\\
&+\frac {\epsilon^2}2(X\cdot\nabla_Y-\partial_t)\phi(X,Y,t)+{o}(\epsilon^2).
\end{align}
This inequality holds for any smooth function $\phi$. Note that the reverse inequality can be derived by considering a point where $\phi$ attains its maximum. Assume that $\nabla_X\phi(X,Y,t)\neq 0$.  Then $\nabla_X\phi(X,\tilde Y -(\tilde t - t)X,\tilde t)\neq 0$ for all  $(\tilde Y,\tilde t)\in B_{\epsilon^3}(Y)\in [t-\epsilon^2,t]$ if $\epsilon$ is small enough, which implies that $X_1$ must lie on the boundary of $B_\epsilon(X)$. Furthermore, since $X_1$ is a minimum also on $\partial B_\epsilon(X)$, we can deduce that $\nabla_X\phi(X_1, \tilde Y - (\tilde t - t)X, \tilde t)$ is perpendicular to $\partial B_\epsilon(X)$ at $X_1$ and points inwards. Thus,
$$\lim_{\epsilon\to 0}\frac{X_1^{\epsilon,\tilde Y,\tilde t}-X}{\epsilon} = \lim_{\epsilon\to 0} -\frac {\nabla_X\phi}{|\nabla_X\phi|}(X_1^{\epsilon, \tilde Y, \tilde t},\tilde Y - (\tilde t - t)X, \tilde t) =-\frac {\nabla_X\phi}{|\nabla_X\phi|}(X,Y,t).$$
As a result,
\begin{eqnarray}\label{imp+}
\quad\quad\quad\lim_{\epsilon\to 0} \barint_{B_{\epsilon^3}(Y)} \barint_{t-\epsilon^2}^t \langle \nabla_X^2\phi(X,Y,t)(\frac{X_1^{\epsilon,\tilde Y,\tilde t}-X}{\epsilon}), (\frac{X_1^{\epsilon,\tilde Y,\tilde t}-X}{\epsilon})\rangle\, \d \tilde Y\d \tilde t=\Delta_{\infty,X}^N\phi(X,Y,t).
\end{eqnarray}
Combining \eqref{imp+} with \eqref{imp} and its reverse analogue, we see that
\begin{align}\label{impasdfsdaf}
&\frac 1 2\barint_{B_{\epsilon^3}(Y)} \barint_{t-\epsilon^2}^t \biggl\{\max_{\tilde X\in\overline{{B_\epsilon(X)}}}\phi(\tilde X,\tilde Y-(\tilde t-t)X,\tilde t)\biggr \}\d \tilde Y\d \tilde t\notag\\
&+\frac 1 2\barint_{B_{\epsilon^3}(Y)} \barint_{t-\epsilon^2}^t
\biggl\{\min_{\tilde X\in\overline{{B_\epsilon(X)}}}\phi(\tilde X,\tilde Y-(\tilde t-t)X,\tilde t)\biggr \}\d \tilde Y\d \tilde t\notag\\
&-\phi(X,Y,t)\notag
\\
&= \frac {\epsilon^2}2 \mathcal{K}_\infty \phi(X,Y,t) +{o}(\epsilon^2).
\end{align}
From this, it immediately follows, similarly as in the case $p=2$, that $u$ satisfies the condition for the asymptotic mean value formula at $(X,Y,t)$ if and only if the estimates required by the definition of viscosity solutions hold at $(X,Y,t)$.

It remains to consider the case where $\nabla_X\phi(X,Y,t)=0$ and $\nabla_X^2\phi(X,Y,t)=0$. In this case, the term in \eqref{imp} involving the Hessian vanishes. Recalling that we can prove a similar estimate in the reverse direction where again the term involving the Hessian vanishes, we see that
\begin{align}\label{imp_x_hess=0}
&\frac 1 2\barint_{B_{\epsilon^3}(Y)} \barint_{t-\epsilon^2}^t \biggl\{\max_{\tilde X\in\overline{{B_\epsilon(X)}}}\phi(\tilde X,\tilde Y-(\tilde t-t)X,\tilde t)\biggr \}\d \tilde Y\d \tilde t\notag
\\
&+\frac 1 2\barint_{B_{\epsilon^3}(Y)} \barint_{t-\epsilon^2}^t
\biggl\{\min_{\tilde X\in\overline{{B_\epsilon(X)}}}\phi(\tilde X,\tilde Y-(\tilde t-t)X,\tilde t)\biggr \}\d \tilde Y\d \tilde t\notag
\\
&-\phi(X,Y,t)\notag\\
&= \frac {\epsilon^2}2(X\cdot\nabla_Y-\partial_t)\phi(X,Y,t)+{o}(\epsilon^2).
\end{align}
Suppose now that $\phi$ touches $u$ from below. If $u$ is a viscosity solution, then by Lemma \ref{Vissolsimp}, we have
\begin{align}\label{thnt}
 (X\cdot\nabla_Y-\partial_t)\phi(X,Y,t) \leq 0,
\end{align}
and, combining this with \eqref{imp_x_hess=0}, we see that \eqref{meanvalue3++} holds. Conversely, assuming \eqref{meanvalue3++}, we obtain from \eqref{imp_x_hess=0} that
\begin{align*}
 \frac {\epsilon^2}2(X\cdot\nabla_Y-\partial_t)\phi(X,Y,t)+{o}(\epsilon^2) \leq 0.
\end{align*}
Dividing by $\epsilon^2$ and passing to the limit $\epsilon \to 0$, we end up with \eqref{thnt}. Similar equivalences can be obtained if $\phi$ touches $u$ from above. This confirms that also in the case $\nabla_X \phi(X,Y,t) = 0$ and $\nabla^2_X\phi(X,Y,t)=0$ the function $u$ satisfies the condition for viscosity solutions in Lemma \ref{Vissolsimp} if and only if it satisfies the asymptotic mean value formula in the viscosity sense at $(X,Y,t)$.

{\bf Proof in the cases $\boldsymbol{p\in (1,\infty)}$.} Adding the first line of \eqref{phi_est} multiplied by $\beta$ and \eqref{impasdfsdaf} multiplied by $\alpha$, we end up with an expression relating $\mathcal{K}_p\phi(X,Y,t)$ to the correct asymptotic mean value formula in the case that $\nabla_X\phi(X,Y,t)\neq 0$. In the case $\nabla_X\phi(X,Y,t) = 0$, $\nabla_X^2\phi(X,Y,t) = 0$, we instead add the first line of \eqref{phi_est} multiplied by $\beta$ and \eqref{thnt} multiplied by $\alpha$ and proceed as before.
\end{proof}

We also have the following version of Theorem \ref{meanvalue1thm}.
\begin{theorem}\label{meanvalue1thmgaian-} Let $D\subset\mathbb R^{M+1}$ be an open set and consider $p$, $1<p\leq\infty$. Let $\alpha$ and $\beta$ be defined as
in \eqref{Conv}.  Let $u\in C(D)$. The asymptotic mean value formula
\begin{align*}
u(X,Y,t)&=\frac \alpha 2\biggl\{\sup_{\tilde X\in{{B_\epsilon(X)}}}u(\tilde X,Y+{\epsilon^2}X/2,t-{\epsilon^2}/2)+\inf_{\tilde X\in{{B_\epsilon(X)}}}u(\tilde X,Y+{\epsilon^2}X/2,t-{\epsilon^2}/2)\biggr \}\notag
\\
&+\beta \barint_{B_\epsilon(X)} u(\tilde X,Y+{\epsilon^2}X/2,t-{\epsilon^2}/{2}) \d \tilde X\notag
\\
&+{o}(\epsilon^2),\mbox{ as }\epsilon\to 0,
\end{align*}
holds for every $(X,Y,t)\in D$ in the viscosity sense  if  and only if $u$ is a viscosity solution to
\begin{eqnarray}\label{solv1+}
  \K_p u(X,Y,t)=0\mbox { in }D.
\end{eqnarray}
\end{theorem}
\begin{proof} The proof follows by retracing the proof of Theorem \ref{meanvalue1thm}. We omit the details.
\end{proof}

Next we also note the following version of Theorem \ref{meanvalue1thm} in which we optimize the function $u(\tilde X,\tilde Y-(\tilde t-t)\tilde X,\tilde t)$ instead of $u(\tilde X,\tilde Y-(\tilde t-t)X,\tilde t)$.

\begin{theorem}\label{meanvalue1thmmodofied} Let $D\subset\mathbb R^{M+1}$ be an open set and consider $p$, $1<p\leq\infty$. Let $\alpha$ and $\beta$ be defined as
in \eqref{Conv}. Let $u\in C(D)$. The asymptotic mean value formula
\begin{align*}
u(X,Y,t) &= \frac \alpha 2\barint_{B_{\epsilon^3}(Y)} \barint_{t-\epsilon^2}^t \biggl\{\sup_{\tilde X\in{{B_\epsilon(X)}}}u(\tilde X,\tilde Y-(\tilde t-t)\tilde X,\tilde t)\biggr \}\d \tilde Y\d \tilde t \notag
\\
&+\frac \alpha 2\barint_{B_{\epsilon^3}(Y)} \barint_{t-\epsilon^2}^t
\biggl\{\inf_{\tilde X\in{{B_\epsilon(X)}}}u(\tilde X,\tilde Y-(\tilde t-t)\tilde X,\tilde t)\biggr \}\d \tilde Y\d \tilde t\notag
\\
&+\beta \barint_{B_\epsilon(X)} \barint_{B_{\epsilon^3}(Y)} \barint_{t-\epsilon^2}^t u(\tilde X,\tilde Y-(\tilde t-t)\tilde X,\tilde t) \d \tilde X \d \tilde Y\d \tilde t\notag
\\
&+{o}(\epsilon^2),\mbox{ as }\epsilon\to 0,
\end{align*}
holds for every $(X,Y,t)\in D$ in the viscosity sense  if  and only if $u$ is a viscosity solution to
\begin{eqnarray}\label{eq:K_pu=0ML}
  \K_p u(X,Y,t)=0\mbox { in }D.
\end{eqnarray}
\end{theorem}
\begin{proof} The proof in the case $p=2$ is essentially identical to the argument in Theorem \ref{meanvalue1thm} as, for a test function $\phi$,
$$\phi(\tilde X,\tilde Y-(\tilde t-t)\tilde X,\tilde t)-\phi(\tilde X,\tilde Y-(\tilde t-t)X,\tilde t)={o}(\epsilon^2)\mbox{ as }\epsilon\to 0.$$
In the following, we only give the complete proof in the case $p=\infty$.  Let $\phi$ be a test function. Consider $(\tilde Y,\tilde t)\in {B_{\epsilon^3}(Y)}\times({t-\epsilon^2},t)$ fixed,  and let now $X_1:=X_1^{\epsilon,\tilde Y,\tilde t} \in \overline{B_\epsilon(X)}$ be a point such that,
\begin{align*}
\phi(X_1, \tilde Y-(\tilde t-t)X_1,\tilde t)&=\min_{\tilde X\in\overline{{B_\epsilon(X)}}}\phi(\tilde X,\tilde Y-(\tilde t-t)\tilde X,\tilde t)=\inf_{\tilde X\in{{B_\epsilon(X)}}}\phi(\tilde X,\tilde Y-(\tilde t-t)\tilde X,\tilde t).
\end{align*}
As before, again using Taylor's formula at $(X,Y,t)$, we deduce
\begin{align*}
\phi(X_1,\tilde Y-(\tilde t-t)X_1,\tilde t)
&=\phi(X,Y,t)+\nabla_X\phi(X,Y,t)\cdot(X_1-X)\notag\\
&+\frac 1 2 \langle \nabla_X^2\phi(X,Y,t)(X_1-X), (X_1-X)\rangle\notag\\
&-(\tilde t-t)\bigl (X\cdot\nabla_Y-\partial_t\bigr)\phi(X,Y,t)+ {o}(\epsilon^2),
\end{align*}
as $\epsilon\to 0$. Similarly, evaluating the Taylor expansion also at $\tilde X_1=2X-X_1$, and adding the two expansions, we see that
\begin{align*}
&\phi(\tilde X_1,\tilde Y-(\tilde t-t)X_1,\tilde t)+\phi(X_1,\tilde Y-(\tilde t-t) X_1,\tilde t)-2\phi(X,Y,t)\notag\\
&= \langle \nabla_X^2\phi(X,Y,t)(X_1-X), (X_1-X)\rangle-2(X\cdot\nabla_Y-\partial_t)\phi(X,Y,t)(\tilde t-t)+ {o}(\epsilon^2), \mbox{ as }\epsilon\to 0.
\end{align*}
Using that $X_1$ is the point where minimum occurs,
\begin{align*}
&\phi(\tilde X_1,\tilde Y-(\tilde t-t)X_1,\tilde t)+\phi(X_1,\tilde Y-(\tilde t-t)X_1,\tilde t)-2\phi(X,Y,t)\notag\\
&\leq \max_{\tilde X\in\overline{{B_\epsilon(X)}}}\phi(\tilde X,\tilde Y-(\tilde t-t)\tilde X,\tilde t)+\min_{\tilde X\in\overline{{B_\epsilon(X)}}}\phi(\tilde X,\tilde Y-(\tilde t-t)\tilde X,\tilde t)-2\phi(X,Y,t),
\end{align*}
and we deduce
\begin{align*}
& \max_{\tilde X\in\overline{{B_\epsilon(X)}}}\phi(\tilde X,\tilde Y-(\tilde t-t)\tilde X,\tilde t)+\min_{\tilde X\in\overline{{B_\epsilon(X)}}}\phi(\tilde X,\tilde Y-(\tilde t-t)\tilde X,\tilde t)-2\phi(X,Y,t)\notag\\
&\geq \langle \nabla_X^2\phi(X,Y,t)(X_1-X), (X_1-X)\rangle-2(X\cdot\nabla_Y-\partial_t)\phi(X,Y,t)(\tilde t-t)+{o}(\epsilon^2).
\end{align*}
Hence, taking averages
\begin{align}\label{impagain}
&\frac 1 2\barint_{B_{\epsilon^3}(Y)} \barint_{t-\epsilon^2}^t \biggl\{\max_{\tilde X\in\overline{{B_\epsilon(X)}}}\phi(\tilde X,\tilde Y-(\tilde t-t)\tilde X,\tilde t)\biggr \}\d \tilde Y\d \tilde t\notag\\
&+\frac 1 2\barint_{B_{\epsilon^3}(Y)} \barint_{t-\epsilon^2}^t
\biggl\{\min_{\tilde X\in\overline{{B_\epsilon(X)}}}\phi(\tilde X,\tilde Y-(\tilde t-t)\tilde X,\tilde t)\biggr \}\d \tilde Y\d \tilde t\notag\\
&-\phi(X,Y,t)\notag\\
&\geq \frac {\epsilon^2}2\barint_{B_{\epsilon^3}(Y)} \barint_{t-\epsilon^2}^t \langle \nabla_X^2\phi(X,Y,t)(\frac{X_1^{\epsilon,\tilde Y,\tilde t}-X}{\epsilon}), (\frac{X_1^{\epsilon,\tilde Y,\tilde t}-X}{\epsilon})\rangle\, \d \tilde Y\d \tilde t\notag\\
&+\frac {\epsilon^2}2(X\cdot\nabla_Y-\partial_t)\phi(X,Y,t)+{o}(\epsilon^2).
\end{align}
This inequality holds for any smooth function $\phi$. Note that the reverse inequality can be derived by considering a point where $\phi$ attains its maximum. Assume that $\nabla_X\phi(X,Y,t)\neq 0$.  Then 
$$\nabla_X(\phi(X,\tilde Y -(\tilde t - t)X,\tilde t))=
\nabla_X\phi(X,\tilde Y -(\tilde t - t)X,\tilde t)-(\tilde t - t)\nabla_Y\phi(X,\tilde Y -(\tilde t - t)X,\tilde t)\neq 0$$ 
for all  $(\tilde Y,\tilde t)\in B_{\epsilon^3}(Y)\times [t-\epsilon^2,t]$ if $\epsilon$ is small enough, which implies that $X_1$ must lie on the boundary of $B_\epsilon(X)$. Furthermore, since $X_1$ is a minimum also on $\partial B_\epsilon(X)$, we can deduce that $\nabla_X(\phi(X_1,\tilde Y -(\tilde t - t)X_1,\tilde t))$ is perpendicular to $\partial B_\epsilon(X)$ at $X_1$ and points inwards. Thus,
\begin{align*}
\lim_{\epsilon\to 0}\frac{X_1^{\epsilon,\tilde Y,\tilde t}-X}{\epsilon} &= \lim_{\epsilon\to 0} -\frac {\nabla_X(\phi(X_1^{\epsilon,\tilde Y,\tilde t},\tilde Y -(\tilde t - t)X_1^{\epsilon,\tilde Y,\tilde t},\tilde t))}{|{\nabla_X(\phi(X_1^{\epsilon,\tilde Y,\tilde t},\tilde Y -(\tilde t - t)X_1^{\epsilon,\tilde Y,\tilde t},\tilde t))}|}\\
&= \lim_{\epsilon\to 0} -\frac {\nabla_X\phi(X_1^{\epsilon,\tilde Y,\tilde t},\tilde Y -(\tilde t - t)X_1^{\epsilon,\tilde Y,\tilde t},\tilde t)-(\tilde t - t)\nabla_Y\phi(X_1^{\epsilon,\tilde Y,\tilde t},\tilde Y -(\tilde t - t)X_1^{\epsilon,\tilde Y,\tilde t},\tilde t)}
{|\nabla_X\phi(X_1^{\epsilon,\tilde Y,\tilde t},\tilde Y -(\tilde t - t)X_1^{\epsilon,\tilde Y,\tilde t},\tilde t)-(\tilde t - t)\nabla_Y\phi(X_1^{\epsilon,\tilde Y,\tilde t},\tilde Y -(\tilde t - t)X_1^{\epsilon,\tilde Y,\tilde t},\tilde t)|}\\
& =-\frac {\nabla_X\phi}{|\nabla_X\phi|}(X,Y,t).
\end{align*}
As a result,
\begin{eqnarray}\label{imp+again}
\quad\quad\quad\lim_{\epsilon\to 0} \barint_{B_{\epsilon^3}(Y)} \barint_{t-\epsilon^2}^t \langle \nabla_X^2\phi(X,Y,t)(\frac{X_1^{\epsilon,\tilde Y,\tilde t}-X}{\epsilon}), (\frac{X_1^{\epsilon,\tilde Y,\tilde t}-X}{\epsilon})\rangle\, \d \tilde Y\d \tilde t=\Delta_{\infty,X}^N\phi(X,Y,t).
\end{eqnarray}
Combining \eqref{imp+again} with \eqref{impagain} and its reverse analogue, we see that
\begin{align}\label{impasdfsdafagain}
&\frac 1 2\barint_{B_{\epsilon^3}(Y)} \barint_{t-\epsilon^2}^t \biggl\{\max_{\tilde X\in\overline{{B_\epsilon(X)}}}\phi(\tilde X,\tilde Y-(\tilde t-t)\tilde X,\tilde t)\biggr \}\d \tilde Y\d \tilde t\notag\\
&+\frac 1 2\barint_{B_{\epsilon^3}(Y)} \barint_{t-\epsilon^2}^t
\biggl\{\min_{\tilde X\in\overline{{B_\epsilon(X)}}}\phi(\tilde X,\tilde Y-(\tilde t-t)\tilde X,\tilde t)\biggr \}\d \tilde Y\d \tilde t\notag\\
&-\phi(X,Y,t)\notag
\\
&= \frac {\epsilon^2}2 \mathcal{K}_\infty \phi(X,Y,t) +{o}(\epsilon^2).
\end{align}
From this, it immediately follows, similarly as in the case $p=2$, that $u$ satisfies the condition for the asymptotic mean value formula at $(X,Y,t)$ if and only if the estimates required by the definition of viscosity solutions hold at $(X,Y,t)$. The case where $\nabla_X\phi(X,Y,t)=0$ and $\nabla_X^2\phi(X,Y,t)=0$ follows exactly as in the proof of Theorem \ref{meanvalue1thm}. We omit further details.\end{proof}

Finally we state following version of Theorem \ref{meanvalue1thmmodofied}, which hence is a version of Theorem \ref{meanvalue1thm}.

\begin{theorem}\label{meanvalue1thmgaian} Let $D\subset\mathbb R^{M+1}$ be an open set and consider $p$, $1<p\leq\infty$. Let $\alpha$ and $\beta$ be defined as
in \eqref{Conv}.  Let $u\in C(D)$. The asymptotic mean value formula
\begin{align*}
u(X,Y,t)&=\frac \alpha 2\biggl\{\sup_{\tilde X\in{{B_\epsilon(X)}}}u(\tilde X,Y+{\epsilon^2}\tilde X/2,t-{\epsilon^2}/2)+\inf_{\tilde X\in{{B_\epsilon(X)}}}u(\tilde X,Y+{\epsilon^2}\tilde X/2,t-{\epsilon^2}/2)\biggr \}\notag
\\
&+\beta \barint_{B_\epsilon(X)} u(\tilde X,Y+{\epsilon^2}\tilde X/2,t-{\epsilon^2}/{2}) \d \tilde X\notag
\\
&+{o}(\epsilon^2),\mbox{ as }\epsilon\to 0,
\end{align*}
holds for every $(X,Y,t)\in D$ in the viscosity sense  if  and only if $u$ is a viscosity solution to
\begin{eqnarray}\label{solv1+ML}
  \K_p u(X,Y,t)=0\mbox { in }D.
\end{eqnarray}
\end{theorem}
\begin{proof} The proof follows by retracing the proof of Theorem \ref{meanvalue1thmmodofied}. We omit the details.
\end{proof}

\section{$(p,\epsilon)$-Kolmogorov functions and tug-of-war games with noise}\label{sec3}

In  Theorems \ref{meanvalue1thm}-\ref{meanvalue1thmgaian} we have established four asymptotic mean value formulas. These formulas are all slightly different but equivalent in the sense that they all characterize viscosity solutions to $\K_p u(X,Y,t)=0$. Motivated by these asymptotic mean value theorems, we next study the functions which satisfy such a mean value
property but without the correction term ${o}(\epsilon^2)$. Strictly speaking we could choose any of the formulas developed as the base for our analysis but
we here take the asymptotic mean value theorem in Theorem \ref{meanvalue1thmgaian}, without the correction term ${o}(\epsilon^2)$, as our starting point. We will call these functions $(p,\epsilon)$-Kolmogorov functions. The very existence and uniqueness of $(p,\epsilon)$-Kolmogorov functions turns out to be a subtle thing, and the complexity stems from what we should actually mean by a $(p,\epsilon)$-Kolmogorov function in terms of measurability, see Remark \ref{borelm} below. In the following, we consider $p\geq 2$ in order to be able to interpret $\alpha$ and $\beta$ as probabilities.

Geometrically, we will work in  product domains $\Omega=U_X\times\mathbb R^m\times I\subset \mathbb R^{M+1}$ and we will assume that $U_X\subset\mathbb R^m$ is a  bounded domain with $C^2$-smooth boundary. In this paper, we will not discuss the extent to which the $C^2$-smoothness can be relaxed. Given $T$, $0<T<\infty$, we let $I:=(0,T)\subset \mathbb R$.  Given $\epsilon>0$ small, we introduce
\begin{align}
    \Gamma_X^\epsilon&:=\{X\in \mathbb R^m\setminus U_X:\  d(X,\partial U_X)\leq \epsilon\},
\end{align}
where $d(\cdot,E)$ denotes the standard Euclidean distance from points in $\mathbb R^m$ to the closed set $E\subset\mathbb R^m$.
Using this notation, we let $U_X^\epsilon:=U_X\cup \Gamma_X^\epsilon$ and
$$\Gamma_\epsilon:=\Gamma_\epsilon^1\cup\Gamma_\epsilon^2,$$
where
\begin{align}
\Gamma_\epsilon^1&:=(\Gamma_X^\epsilon\times \mathbb R^m\times(-\epsilon^2/2,T]),\ \Gamma_\epsilon^2:=(U_X\times \mathbb R^m\times(-\epsilon^2/2,0]).
\end{align}

We say that $F:\Gamma_\epsilon \to \R$ belongs to the function class $\mathcal{G}_\epsilon$ if the following three conditions are met:
\begin{enumerate}
 \item $F: \Gamma_\epsilon \to \R$ is bounded.
 \item\label{cond:G_eps-two} $X \mapsto F(X,Y,t)$ is Borel measurable for every $(Y,t) \in \R^m \times (-\epsilon^2/2, T]$.
 \item\label{cond:G_eps-three} For every $t\in (-\epsilon^2/2,T]$, we have the following uniform continuity condition: For all $\eta>0$, there is a $\delta>0$ such that
 \begin{align*}
 |F(X,Y_1,t)-F(X,Y_2,t)| < \eta, \textnormal{ whenever } |Y_1 - Y_2| < \delta \textnormal{ and } (X,Y_j,t) \in \Gamma_\epsilon.
 \end{align*}
\end{enumerate}
Note that in \eqref{cond:G_eps-two}, if $t\leq 0$, the domain of the map is $U^\epsilon_X$, and if $t>0$, the domain of the map is $\Gamma^\epsilon_X$. Similarly, in \eqref{cond:G_eps-three}, if $t\leq 0$, the points $(X,Y_j,t)$ belong to $\Gamma_\epsilon$ when $X \in U^\epsilon_X$, and if $t>0$, the points $(X,Y_j,t)$ belong to $\Gamma_\epsilon$ when $X\in \Gamma^\epsilon_X$.

We say that a function $v:U^\epsilon_X\times\R^m\times (-\epsilon^2/2,T]\to \R$ belongs to the function class $\mathcal{G}_\epsilon'$ if the following three conditions are met:
\begin{enumerate}
 \item $v: U^\epsilon_X\times\R^m\times (-\epsilon^2/2,T] \to \R$ is bounded.
 \item\label{cond:G_eps'-two} $X \mapsto v(X,Y,t)$ is Borel measurable for every $(Y,t) \in \R^m \times (-\epsilon^2/2, T]$.
 \item\label{cond:G_eps'-three} For every $t\in (-\epsilon^2/2,T]$, we have the following uniform continuity condition: For all $\eta>0$, there is a $\delta>0$ such that
 \begin{align*}
 |v(X,Y_1,t)-v(X,Y_2,t)| < \eta, \textnormal{ whenever } |Y_1 - Y_2| < \delta \textnormal{ and } X \in U^\epsilon_X.
 \end{align*}
\end{enumerate}
The following result will be useful later when we relate $(p,\epsilon)$-Kolmogorov functions to the game.
\begin{lemma}\label{lem:G_eps'_is_Borel-measurable}
 Let $v\in \mathcal{G}_\epsilon'$. For every $t\in (-\epsilon^2/2,T]$, the map
 \begin{align*}
  U^\epsilon_X \times \R^m \ni (X,Y) \mapsto v(X,Y,t)
 \end{align*}
is Borel measurable.
\end{lemma}
\begin{proof}
 Let $t\in (-\epsilon^2/2,T]$ and denote $g(X,Y):= v(X,Y,t)$. Due to property \eqref{cond:G_eps'-three} in the definition of $\mathcal{G}_\epsilon'$, we can approximate $g$ pointwise with a sequence $(g_k)$ of functions that are constant in the $Y$-variable on cubes of side length $2^{-k}$. More precisely, we set
 \begin{align*}
  g_k(X,Y) = g(X, \hat Y),
 \end{align*}
where $\hat Y$ is the unique element in $2^{-k}\Z^m$ for which $Y \in \hat Y + [0, 2^{-k})^m$. Each $g_k$ is Borel measurable due to property \eqref{cond:G_eps'-two} in the definition of $\mathcal{G}_\epsilon'$. In fact, the pre-image $g_k^{-1}(W)$ of any open set $W\subset \R$ is a countable union of sets of the form $A\times Q$, where $A$ is a Borel subset of $U^\epsilon_X$ and $Q$ is a cube in $\R^m$. Thus, the pointwise limit $g$ is also a Borel function.
\end{proof}
Using the notions of $\mathcal{G}_\epsilon$ and $\mathcal{G}_\epsilon'$, we next introduce the following notion of $(p,\epsilon)$-Kolmogorov functions.
\begin{definition}\label{meanvalue1thmgaiandef} Let $p$, $2\leq p\leq\infty$. We say that $u_\epsilon: U_X^\epsilon \times \R^m \times (-\epsilon^2 / 2, T] \to \R$ is a $(p,\epsilon)$-Kolmogorov function in $U_X \times\mathbb R^m\times I$ with boundary values $F\in \mathcal{G}_\epsilon$ if $u_\epsilon\in\mathcal{G}_\epsilon'$ and
\begin{align}\label{cond:p-eps-Kolm1}
\quad u_\epsilon(X,Y,t)&=\frac \alpha 2\biggl\{\sup_{\tilde X\in{{{B_\epsilon(X)}}}}u_\epsilon(\tilde X,Y+{\epsilon^2}\tilde X/2,t-{\epsilon^2}/2)+\inf_{\tilde X\in{{{B_\epsilon(X)}}}}u_\epsilon(\tilde X,Y+{\epsilon^2}\tilde X/2,t-{\epsilon^2}/2)\biggr \}
\\
\notag &+\beta \barint_{B_\epsilon(X)} u_\epsilon(\tilde X,Y+{\epsilon^2}\tilde X/2,t-{\epsilon^2}/2) \d \tilde X\mbox{ for every $(X,Y,t)\in U_X\times\mathbb R^m\times I$},
\\
\label{cond:p-eps-Kolm2} \quad u_\epsilon(X,Y,t)&=F(X,Y,t)\mbox{ for every $(X,Y,t)\in \Gamma_\epsilon$}.
\end{align}
\end{definition}

\begin{remark}\label{borelm} Note that in Definition \ref{meanvalue1thmgaiandef}, we have the function (analogous with $\sup$ replaced by $\inf$)
\begin{align}\label{expr1}
X\to \sup_{\tilde X\in{{{B_\epsilon(X)}}}}u_\epsilon(\tilde X,Y+{\epsilon^2}\tilde X/2,t-{\epsilon^2}/2),
\end{align}
and not
\begin{align}\label{expr2}
X\to\sup_{\tilde X\in{{{B_\epsilon(X)}}}}u_\epsilon(\tilde X,Y+{\epsilon^2}X/2,t-{\epsilon^2}/2).
\end{align}
As we will see in the proof of Lemma \ref{EUKolmo}, it is easy to show that the function defined by \eqref{expr1} is Borel measurable for every fixed $(Y,t)$, and this remains true even if $u_\epsilon$ is replaced by a function which has no regularity in the $Y$-variable. Due to condition \eqref{cond:G_eps'-two} in the definition of $\mathcal{G}_\epsilon'$, one can show that \eqref{expr2} also defines a Borel function for each fixed $(Y,t)$, but this is not necessarily the case if we replace $u_\epsilon$ with a less regular function. The choice of using \eqref{expr1} in the definition of $(p,\epsilon)$-Kolmogorov functions, as well as the choice to base our analysis on Theorem \ref{meanvalue1thmgaian}, can therefore be motivated by the fact that it potentially allows generalizations of the results below for boundary data which is less regular in $Y$.
\end{remark}

\subsection{Existence and uniqueness of $(p,\epsilon)$-Kolmogorov functions} The purpose of the subsection is to prove the following lemma.

\begin{lemma}\label{EUKolmo}
Given boundary values $F\in \mathcal{G}_\epsilon$, there exists a unique $(p,\epsilon)$-Kolmogorov function in $U_X \times\mathbb R^m\times I$ in the sense of Definition \ref{meanvalue1thmgaiandef}.
\end{lemma}
\begin{proof} In the following, $p$, $2\leq p\leq\infty$, and $\epsilon>0$ are fixed. Given $F\in \mathcal{G}_\epsilon$, we have to prove that
there exists a unique function $u_\epsilon: U_X^\epsilon \times \R^m \times (-\epsilon^2 / 2, T] \to \R$, $u_\epsilon \in\mathcal{G}_\epsilon'$, which satisfies \eqref{cond:p-eps-Kolm1} and \eqref{cond:p-eps-Kolm2}.
 We define an operator $\mathcal{T}$ on $\mathcal{G}_\epsilon'$ in the following way. Given a function $v\in \mathcal{G}_\epsilon'$, we let
\begin{align*}
\quad \mathcal{T}v(X,Y,t)&:=\frac \alpha 2\biggl\{\sup_{\tilde X\in{{{B_\epsilon(X)}}}}v(\tilde X,Y+{\epsilon^2}\tilde X/2,t-{\epsilon^2}/2)+\inf_{\tilde X\in{{{B_\epsilon(X)}}}}v(\tilde X,Y+{\epsilon^2}\tilde X/2,t-{\epsilon^2}/2)\biggr \}\notag
\\
&+\beta \barint_{B_\epsilon(X)} v(\tilde X,Y+{\epsilon^2}\tilde X/2,t-{\epsilon^2}/2) \d \tilde X\mbox{ for every $(X,Y,t)\in U_X\times\mathbb R^m\times I$},\notag\\
\quad \mathcal{T}v(X,Y,t)&:=F(X,Y,t)\mbox{ for every $(X,Y,t)\in \Gamma_\epsilon$}.
\end{align*}
It is important  to note that $\mathcal{T}$ preserves the class $\mathcal{G}_\epsilon'$ in the sense that if $v\in \mathcal{G}_\epsilon'$, then $\mathcal{T}v\in \mathcal{G}_\epsilon'$. To see this, note first that $\mathcal{T} v$ is bounded because $v$ is bounded. Furthermore, for $(Y,t)\in \mathbb R^m\times I$ fixed, the functions
\begin{align}\label{borelma}
X\to \sup_{\tilde X\in{{{B_\epsilon(X)}}}}v(\tilde X,Y+{\epsilon^2}\tilde X/2,t-{\epsilon^2}/2),\ X\to \inf_{\tilde X\in{{{B_\epsilon(X)}}}}v(\tilde X,Y+{\epsilon^2}\tilde X/2,t-{\epsilon^2}/2),
\end{align}
are Borel measurable on $U_X$. Indeed, for every $\lambda\in\mathbb R$, the set
\begin{align*}
&\{X\in U_X:\ \sup_{\tilde X\in{{{B_\epsilon(X)}}}}v(\tilde X,Y+{\epsilon^2}\tilde X/2,t-{\epsilon^2}/2)>\lambda\}
\end{align*}
can be expressed as
\begin{align*}
U_X\cap\biggl(\ \bigcup_{\tilde X\in U_X^\epsilon,\ v(\tilde X,Y+{\epsilon^2}\tilde X/2,t-{\epsilon^2}/2)>\lambda}B_\epsilon(\tilde X)\biggr),
\end{align*}
and this set is open as the union of an arbitrary collection of open sets is open. Note that a subtle point here is that $\sup_{\tilde X\in{{{B_\epsilon(X)}}}}$ and $\inf_{\tilde X\in{{{B_\epsilon(X)}}}}$ are used instead of $\sup_{\tilde X\in{{\overline{B_\epsilon(X)}}}}$ and $\inf_{\tilde X\in{{\overline{B_\epsilon(X)}}}}$ in the definition of $(p,\epsilon)$-Kolmogorov functions. In fact, if $\overline{B_\epsilon(X)}$ is used instead of ${B_\epsilon(X)}$, then the Borel measurability of the functions in \eqref{borelma} can fail already for $Y$ and $t$ independent functions, see Example 2.4 in \cite{luirops14}. Moreover, due to the boundedness of $v$ and the continuity in $Y$, one can see that
\begin{align*}
 X \mapsto \barint_{B_\epsilon(X)} v(\tilde X,Y+{\epsilon^2}\tilde X/2,t-{\epsilon^2}/2) \d \tilde X
\end{align*}
is continuous and thus Borel measurable. These observations combined with the fact that $F$ is Borel in the $X$-variable for fixed $(Y,t)$ show that $\mathcal{T}v$ is Borel in the $X$-variable for fixed $(Y,t)$. It remains to verify that $\mathcal{T}v$ satisfies the third condition in the definition of $\mathcal{G}_\epsilon'$. For this purpose, fix $t\in (-\epsilon^2/2,T]$ and let $\eta>0$. For $Y_1,Y_2\in\R^m$, we have that if $(X,Y_1,t) \in \Gamma_\epsilon$, then also $(X,Y_2,t)\in \Gamma_\epsilon$, and since $\mathcal{T}v$ agrees with $F$ on $\Gamma_\epsilon$, we see that
\begin{align*}
 |\mathcal{T}v(X,Y_1,t) - \mathcal{T}v(X,Y_2,t)| = |F(X,Y_1,t) - F(X,Y_2,t)| < \eta,
\end{align*}
provided that $|Y_1 - Y_2| < \delta_0$, where $\delta_0>0$ is a sufficiently small constant related to $F$. If instead $(X,Y_1,t)\notin \Gamma_\epsilon$, we have $(X,Y_1,t),(X,Y_2,t) \in U_X\times\R^m\times (0,T]$. Since $v\in \mathcal{G}_\epsilon'$, there is a $\delta_1>0$ such that $|Y_1-Y_2|< \delta_1$ implies
\begin{align*}
 |v(\tilde X,Y_1,t-\epsilon^2/2) - v(\tilde X,Y_2,t-\epsilon^2/2)| < \eta/3,
\end{align*}
for all $\tilde X \in U_X^\epsilon$. From this, we obtain
\begin{align*}
 \Big|\sup_{\tilde X\in{{{B_\epsilon(X)}}}}v(\tilde X,Y_1+{\epsilon^2}\tilde X/2,t-{\epsilon^2}/2) - \sup_{\tilde X\in{{{B_\epsilon(X)}}}}v(\tilde X,Y_2+{\epsilon^2}\tilde X/2,t-{\epsilon^2}/2)\Big| \leq \eta/3,
\end{align*}
provided that $|Y_1-Y_2|<\delta_1$, and a similar estimate holds also for the infimum. Finally, observe that
\begin{align*}
 \Big|\barint_{B_\epsilon(X)} & v(\tilde X, Y_1 + {\epsilon^2}\tilde X/2, t - {\epsilon^2}/2) \d \tilde X - \barint_{B_\epsilon(X)} v(\tilde X, Y_2 + {\epsilon^2}\tilde X/2, t - {\epsilon^2}/2) \d \tilde X\Big|
 \\
 &\leq \barint_{B_\epsilon(X)} |v(\tilde X, Y_1 + {\epsilon^2}\tilde X/2, t - {\epsilon^2}/2) - v(\tilde X, Y_2 + {\epsilon^2}\tilde X/2, t - {\epsilon^2}/2) |  \d \tilde X
 \\
 & \leq \eta/3,
\end{align*}
if $|Y_1-Y_2| < \delta_1$. Thus, we see that
\begin{align*}
 |\mathcal{T}v(X,Y_1,t) - \mathcal{T}v(X,Y_2,t)| < \eta,
\end{align*}
whenever $X\in U^\epsilon_X$ and  $|Y_1-Y_2|< \delta:=\min\{\delta_0,\delta_1\}$.

Using that  $\mathcal{T}$ preserves the class $\mathcal{G}_\epsilon'$, we set up a recursive scheme as follows to construct $u_\epsilon$. Let $v_0 \in \mathcal{G}_\epsilon'$ be an arbitrary function with boundary values $F$ on $\Gamma_\epsilon$. One, for example, could take $v_0=0$ in all of $U_X\times\R^m\times(0,T]$. We construct a sequence of functions $\{v_i\}$,
$v_i \in \mathcal{G}_\epsilon'$, through  $v_{i+1}:=\mathcal{T}v_i$, $i=0,1,\ldots$. We claim that $\{v_i\}$ converges in a finite number of steps. To establish this, we argue as in the proof of Theorem 5.2 in
\cite{luirops14} and use induction to show that
\begin{align}\label{undu}
v_{i+1}(X,Y,t)-v_{i}(X,Y,t)=0\mbox{ if }t\leq i\eps^2/2.
\end{align}
In particular, $v_{i+1}(X,Y,t)$ can only differ from $v_{i}(X,Y,t)$ if $t>i\eps^2/2$. Obviously, this is clear if $i=0$, since then the operator $\mathcal{T}$ uses the values from $(-\eps^2/2,0]$ which are given by $F$. Suppose now that \eqref{undu} holds for $i\in\{0,1,..,k\}$ for some $k\geq 1$. Consider
$(X,Y,t)\in U_X \times\mathbb R^m\times I$, $t\leq (k+1)\eps^2/2$. Then
\begin{align}\label{undu+}
v_{k+2}(X,Y,t)-v_{k+1}(X,Y,t)=\mathcal{T}(\mathcal{T}v_{k})(X,Y,t)-(\mathcal{T}v_{k})(X,Y,t).
\end{align}
However, by the hypothesis, if $(X,Y,t)\in U_X\times\mathbb R^m\times I$ and $t\leq (k+1)\eps^2/2$, then
\begin{align*}
\mathcal{T}(\mathcal{T} v_{k})(X,Y,t) &:=\frac \alpha 2\biggl\{\sup_{\tilde X\in{{{B_\epsilon(X)}}}}(\mathcal{T} v_{k})(\tilde X,Y+{\epsilon^2}\tilde X/2,t-{\epsilon^2}/2)+\inf_{\tilde X\in{{{B_\epsilon(X)}}}}(\mathcal{T} v_{k})(\tilde X,Y+{\epsilon^2}\tilde X/2,t-{\epsilon^2}/2)\biggr \}\notag
\\
&+\beta \barint_{B_\epsilon(X)} (\mathcal{T} v_{k})(\tilde X,Y+{\epsilon^2}\tilde X/2,t-{\epsilon^2}/2) \d \tilde X\notag
\\
&=\frac \alpha 2\biggl\{\sup_{\tilde X\in{{{B_\epsilon(X)}}}} v_{k}(\tilde X,Y+{\epsilon^2}\tilde X/2,t-{\epsilon^2}/2)+\inf_{\tilde X\in{{{B_\epsilon(X)}}}} v_{k}(\tilde X,Y+{\epsilon^2}\tilde X/2,t-{\epsilon^2}/2)\biggr \}\notag
\\
&+\beta \barint_{B_\epsilon(X)} v_{k}(\tilde X,Y+{\epsilon^2}\tilde X/2,t-{\epsilon^2}/2) \d \tilde X\notag
\\
&=(\mathcal{T}v_{k})(X,Y,t).
\end{align*}
This proves that if $(X,Y,t)\in U_X\times\mathbb R^m\times I$ and $t\leq (k+1)\eps^2/2$, then
\begin{align}\label{undu++}
v_{k+2}(X,Y,t) - v_{k+1}(X,Y,t)=(\mathcal{T}v_{k})(X,Y,t)-(\mathcal{T}v_{k})(X,Y,t)=0.
\end{align}
Hence, by induction, we can conclude that \eqref{undu} holds for all integers $i\geq 0$. That is, the sequence of functions $(v_i)$ does not change for $i \geq 2t/\epsilon^2$. We can thus fix any such integer $i$ and define $u_\epsilon = v_i$. Then we have that $\mathcal{T} u_\epsilon = u_\epsilon$, which, by the definition of $\mathcal{T}$, proves that $u_\epsilon$ is a $(p,\epsilon)$-Kolmogorov function in $U_X \times\mathbb R^m\times I$ with boundary values $F$. Moreover, the uniqueness follows by the same induction argument, or simply from  the comparison principle for $(p,\epsilon)$-Kolmogorov functions stated in Lemma \ref{u4} below. \end{proof}

The following comparison principle for $(p,\epsilon)$-Kolmogorov functions is a consequence of the proof of Lemma \ref{EUKolmo}.

\begin{lemma}\label{u4} Let $v=v_\epsilon$ and $u=u_\epsilon$ be $(p,\epsilon)$-Kolmogorov functions in $ U_X\times\mathbb R^m\times I$ with boundary values
$F_{v}\in \mathcal{G}_\epsilon$ and $F_{u}\in \mathcal{G}_\epsilon$ on $\Gamma_\epsilon$ such that $F_{v}\geq F_{u}$. Then ${v}\geq {u}$ a.e. on
$U_X\times \mathbb R^m\times I$.
\end{lemma}
\begin{proof} Let $v_0, u_0 \in \mathcal{G}_\epsilon'$ be arbitrary functions with boundary values $F_v$ and $F_u$ respectively. Then $\mathcal{T}v_0\ge \mathcal{T}u_0$ for $t\leq \eps^2/2$ where $\mathcal{T}$ was introduced in the proof of Lemma \ref{EUKolmo}. In particular, by iterating this argument, similarly as in the proof  of Lemma \ref{EUKolmo}, we obtain the stated  comparison principle.
\end{proof}

\subsection{Tug-of-war games with noise}\label{subsec:ToW_game_descr} We here formulate an adapted two-player, zero-sum, tug-of-war game with noise, and connect associated value functions to the notion of $(p,\epsilon)$-Kolmogorov functions.  Given $\epsilon$, we let $N$ denote the maximal number of rounds the game is to be played. At the beginning of the tug-of-war game with noise, a token is placed at a point $(X_0,Y_0)\in U_X\times\mathbb R^m$ and the players toss a biased coin with probabilities $\alpha$ and $\beta$, $\alpha + \beta =1$, where $\alpha$ and $\beta$ are as previously defined in \eqref{Conv}.  If they get heads (probability $\alpha$), they play a tug-of-war game in the sense that  a fair coin is tossed and the winner of the toss is allowed to move the velocity coordinate $X_0$ of the game to any $X_1\in B_\eps (X_0)$. The position coordinate then gets updated as $Y_0\to Y_1:=Y_0+{\epsilon^2}X_1/2$, i.e.,  $(X_0,Y_0)\to (X_1,Y_1)$. On the other hand, if they get tails (probability $\beta$), the velocity coordinate $X_0$ of the game moves according to the uniform
probability density to a random point $X_1\in B_\eps(X_0)$, and again $Y_0\to Y_1:=Y_0+{\epsilon^2}X_1/2$. In the next steps, this procedure is repeated at $(X_1,Y_1)$, $(X_2,Y_2)$, and so on, and a sequence of game states $\{(X_k,Y_k)\}$ are constructed according to
\begin{equation*}
X_k\to X_{k+1}, Y_k\to Y_{k+1}:= Y_k+{\epsilon^2} X_{k+1}/2,
\end{equation*}
The game ends when either the $X$-coordinate of the token hits $\Gamma_X^\epsilon$, or the number of rounds played reaches $N$.

We denote by $\tau_N \in\{0,1,...,N\}$ the round at which either the game position reaches $\Gamma_X^\epsilon$ or the number of rounds played equals $N$, whichever happens first, and by $(X_{\tau_N},Y_{\tau_N})\in U_X^\epsilon\times \mathbb R^m$ the endpoint of the game. When no confusion arises, we simply write $\tau$.  The game procedure yields a sequence of game states $(X_0,Y_0), (X_1,Y_1),\ldots, (X_{\tau_N},Y_{\tau_N})$, where every $(X_k,Y_k)$ is a random variable. At the end of the game, Player I earns $\mathcal{F}(X_{\tau_N},Y_{\tau_N},\tau_N)$ while Player II earns $-\mathcal{F}(X_{\tau_N},Y_{\tau_N},\tau_N)$, where
\begin{align}
 \mathcal{F}&:(\Gamma_X^\epsilon\times \mathbb R^m\times\{0,1,...,N\})\cup(U_X\times \mathbb R^m\times\{N\}) \to \mathbb{R} \notag
\end{align}
is a given payoff function that is assumed to be Borel measurable in $(X,Y)$.

The history of the game states up to step $k$ is a vector of the first $k+1$ game states $(X_0,Y_0),\ldots, (X_k,Y_k)$. The space of all possible game state sequences in the case of at most $N$ rounds, and our probability space, is
\[
\begin{split}
H^{N+1}= (X_0,Y_0)\times \left(U_X^\epsilon\times\mathbb R^m\right)\times\ldots\times\left(U_X^\epsilon\times\mathbb R^m\right).
\end{split}
\]
Writing $\om=\left((X_0,Y_0),(X_1,Y_1),\ldots,(X_N,Y_N)\right)\in H^{N+1}$, we can now define $\tau_N$ as the random time variable
\begin{equation*}
\begin{split}
\tau_N(\om):=\min\{N,\inf\{k:\ (X_k,Y_k)\in \Gamma_X^\epsilon\times \mathbb R^m,\ k=0,1,,,N\}\}.
\end{split}
\end{equation*}
$\tau_N=\tau_N(\om)$ is a stopping time relative to the filtration $\{\mathcal{I}_k\}_{k=0}^N$, where
${\mathcal{I}}_0:=\sigma (X_0,Y_0)$ is the $\sigma$-algebra generated by $(X_0,Y_0)$, and
\begin{equation}\label{eq:filtration}
{\mathcal{I}}_k:=\sigma ((X_0,Y_0), (X_1,Y_1),\ldots, (X_k,Y_k))\quad \mbox{for}\quad k\geq 1.
\end{equation}

A strategy $S_{\textrm{I}} = \{S_{\textrm{I}}^k\}_{k=0}^{N}$ for Player I is a collection of  functions that give the next
game position given the history of the game. Strictly speaking, the strategy, as well as the full history, can depend on all the processes of the game, including the previous coin tosses. However, in the arguments presented below, we only use the previous game states. For example, if Player I wins the toss, then
\[
S_{\textrm{I}}^k{\left((X_0,Y_0),(X_1,Y_1),\ldots,(X_k,Y_k)\right)}=X_{k+1}\in  {B_\eps(X_k)}.
\]
Similarly, Player II  plays according to a strategy $S_{{\II}} = \{S_{\textrm{II}}^k\}_{k=0}^{N}$. To be precise, the arguments presented below only use strategies which can be represented by functions $S:  U_X\times\mathbb R^m\to  U_X^\epsilon$ such that if $(X,Y)\in U_X\times\mathbb R^m$, then $S(X,Y)\in B_\epsilon(X)$. Furthermore, concerning measurability, every map $(X,Y)\to S(X,Y)$ is assumed to be Borel measurable.

The fixed starting point $(X_0,Y_0)$, the number of rounds $N$,  the domain $U_X\times \mathbb R^m$ ($U_X^\epsilon\times\mathbb R^m$) and the strategies $S_{\textrm{I}}$ and $S_{{\II}}$ determine a unique probability measure $\mathbb{P}^{(X_0,Y_0),N}_{S_{\textrm{I}},S_{{\II}}}$ on the natural product $\sigma$-algebra. In particular, this measure is defined on the sets of the type $(X_0,Y_0)\times (X_1,Y_1)\times\ldots$, where $\{(X_i,Y_i)\}$, $(X_i,Y_i)\subset U_X^\epsilon\times\mathbb R^m$,  are Borel sets. The probability measure is  built using the initial distribution $\delta_{(X_0,Y_0)}(A_X\times A_Y)$, where $A_X\times A_Y\subset U_X\times\mathbb R^m$, and transition probabilities. Indeed, given a sequence of (random) states $\mathcal{P}_k:=(X_0,Y_0),(X_1,Y_1)\ldots,(X_k,Y_k)$, we define  a family of
transition probabilities   as
\begin{align}\label{eq:meas-steps}
\pi_{S_{\textrm{I}},S_{{\II}}}\left(\mathcal{P}_k,A_X\times A_Y\right)= \biggl (\frac{\alpha}{2} \delta_{S_{\textrm{I}}\left(\mathcal{P}_k\right)}(A_X)+\frac{\alpha}{2} \delta_{S_{{\II}}\left(\mathcal{P}_k\right)}(A_X)+\beta\frac{ |{{A}_X\cap B_\eps(X_k)}|}{|{B_\eps(X_k)}|}\biggr)\delta_{Y_{k+1}}(A_Y),
\end{align}
where, throughout, $\delta$ is the Dirac delta.  Using these transition probabilities, and for sequence of time points $\{t_k\}$, $t_{k+1}-t_k=\epsilon^2/2$,  a family of probability measures $\{\mu_{t_1,...,t_k}\}$ are built on $(U_X^\epsilon\times\mathbb R^m)^k$ satisfying the consistency condition necessary for the Kolmogorov's extension theorem. In particular, the probability measure $\mathbb{P}^{(X_0,Y_0),N}_{S_{\textrm{I}},S_{{\II}}}$
is  built by applying  Kolmogorov's extension theorem to  this family of probability measures (compare to the construction below equation (2.1) in \cite[Section 2]{manfredipr12}).

The expected payoff, when starting at $(X_0,Y_0)$, playing for at most $N$ rounds, and using the strategies $S_{\textrm{I}},S_{{\II}}$, is
\begin{equation}
\label{eq:defi-expectation}
\begin{split}
\mathbb{E}_{S_{\textrm{I}},S_{{\II}}}^{(X_0,Y_0),N}[\mathcal{F}(X_{\tau_N},Y_{\tau_N},{\tau_N})]
&=\int_{H^{N+1}} \mathcal{F}(X_{\tau_N}(\om),Y_{\tau_N}(\om),{\tau_N}(\om)) \d \mathbb{P}^{(X_0,Y_0),N}_{S_{\textrm{I}},S_{{\II}}}(\om).
\end{split}
\end{equation}
The {value of the game for Player I}, when starting at $(X_0,Y_0)$, with the maximum number of rounds $N$, is defined as
\[
u^{\eps,N}_{\textrm{I}}(X_0,Y_0,0):=\sup_{S_{\textrm{I}}}\inf_{S_{{\II}}}\,\mathbb{E}_{S_{\textrm{I}},S_{{\II}}}^{(X_0,Y_0),N}[\mathcal{F}(X_{\tau_N},Y_{\tau_N},{\tau_N})],
\]
while the {value of the game for Player II} is defined as
\[
u^{\eps,N}_{\II}(X_0,Y_0,0):=\inf_{S_{{\II}}}\sup_{S_{\textrm{I}}}\, \mathbb{E}_{S_{\textrm{I}},S_{{\II}}}^{(X_0,Y_0),N}[\mathcal{F}(X_{\tau_N},Y_{\tau_N},{\tau_N})].
\]
More generally, for $k \in\{0,1,...,N\}$ we define the values of the game for the players, when starting at $(X_0,Y_0)$, and playing for a maximum of $h=N-k$ rounds, as
\[
u^{\eps,N}_{\textrm{I}}(X_0,Y_0,k ):=\sup_{S_{\textrm{I}}}\inf_{S_{{\II}}}\,\mathbb{E}_{S_{\textrm{I}},S_{{\II}}}^{(X_0,Y_0),h}[\mathcal{F}(X_{\tau_h},Y_{\tau_h},k+{\tau_h})],
\]
and
\[
u^{\eps,N}_{\II}(X_0,Y_0,k):=\inf_{S_{{\II}}}\sup_{S_{\textrm{I}}}\,\mathbb{E}_{S_{\textrm{I}},S_{{\II}}}^{(X_0,Y_0),h}[\mathcal{F}(X_{\tau_h},Y_{\tau_h},k+{\tau_h})].
\]
Here $\tau_h\in \{0,...,h\}$ is the hitting time of the boundary
\begin{align}
(\Gamma_X^\epsilon\times \mathbb R^m\times\{0,...,N-k\})\cup (U_X\times\mathbb R^m\times\{N-k\}).
\end{align}
For basic properties  of the value functions we refer to \cite{MS}.

\subsection{Value functions and their relations  $(p,\epsilon)$-Kolmogorov functions}\label{subsec:timechange} We here describe the change of time scale that relates values of the tug-of-war games with noise and $(p,\epsilon)$-Kolmogorov functions. The definition of a  $(p,\epsilon)$-Kolmogorov function $u_\epsilon$ given in Definition \ref{meanvalue1thmgaiandef} refers to a forward-in-time parabolic equation as the  $u_\epsilon(\cdot,\cdot, t)$ is  determined by the values $u_\epsilon(\cdot,\cdot, t-\epsilon^2/2)$. In contrast, the value function for the players at step $k$ are determined by the values at future steps.

For $-{\epsilon^2}/{2} < t \leq T$,  let $N(t)$ be the integer defined by $$\frac {t}{\epsilon^2/2}\leq N(t)<\frac {t}{\epsilon^2/2}+1.$$
We use the shorthand notation $N(t)=\lceil{{t}/{(\epsilon^2/2)}}\rceil$. Set $t_0 = t$ and $t_{k+1} = t_k-{\epsilon^2/2}$ for $k = 0,1,...,N(t)-1$, that is, $$t_k = \frac {\epsilon^2}2(N(t)-k)+t_{N(t)}.$$
Observe that $t_{N(t)}\in (-\epsilon^2/2,0]$. When no confusion arises, we simply write $N$ for $N(t)$.

Given $F\in\mathcal{G}_\epsilon$ a boundary value function, we define a payoff function
$$\mathcal{F}_t:(\Gamma_X^\epsilon\times\mathbb R^m\times\{0,...,N(t)\})\cup (U_X\times\mathbb R^m\times\{N(t)\})\to\mathbb R$$ by
$$\mathcal{F}_t(X,Y,k):=F(X,Y,\epsilon^2(N(t)-k)/2+t_{N(t)})=F(X,Y,t_k),$$
and we emphasize that $t$ and $\epsilon$ determine $N$ and $t_N$. Given this notation, if the game begins at $k = 0$, this corresponds to $t_0 = t$ in the time scale. When we play one round $k\to  k + 1$, the clock steps $\epsilon^2/2$ backwards, $t_{k+1} =t_k-\epsilon^2/2$, and we play until we get outside the cylinder when $k =\tau$,  corresponding to $t_\tau$ in the time scale.

Next we define
\begin{align}\label{def:u-eps-I}
 u^\epsilon_{\textrm{I}}(X,Y,t):=u^{\epsilon,N(t)}_{\textrm{I}}(X,Y,0),\  u^\epsilon_\II(X,Y,t):=u^{\epsilon,N(t)}_\II(X,Y,0),
\end{align}
with payoff function $\mathcal{F}_t(X,Y,k)$.
This defines  $u^\epsilon_{\textrm{I}}(X,Y,t)$ and $u^\epsilon_\II(X,Y,t)$ for every instant $t\in(-{\epsilon^2}/{2}, T]$.

\subsection{The existence of a value} Given $F\in \mathcal{G}_\epsilon$,  we established in Lemma \ref{EUKolmo} the existence of a unique $(p,\epsilon)$-Kolmogorov function  $u_\epsilon$  in $U_X \times\mathbb R^m\times I$ with boundary values $F$. The following theorem shows that the $(p,\epsilon)$-Kolmogorov function coincides with the functions $u^\epsilon_I$ and $u^\epsilon_{II}$ defined in \eqref{def:u-eps-I}. Hence, the game has value, and the value is given by $u_\epsilon$.

\begin{theorem}\label{u1}
Let $F\in \mathcal{G}_\epsilon$ and $u_\epsilon$ be the unique $(p,\epsilon)$-Kolmogorov function in $U_X \times\mathbb R^m\times I$ with boundary values $F$ established in Lemma \ref{EUKolmo}. Then
\begin{align*}
 {u_{\textrm{I}}^\epsilon} = u_\epsilon = {u_{\II}^\epsilon} \quad \textnormal{ on }U_X\times\mathbb R^m\times I.
\end{align*}
\end{theorem}
\begin{proof}  We will only prove that
\begin{align}\label{state1}{u_{\II}^\epsilon}\leq {u_\epsilon}\quad\mbox{ on }U_X\times\mathbb R^m\times I.
\end{align}
This is sufficient, as first, the proof that
\begin{align}\label{state2}{u_\epsilon}\leq {u_{\textrm{I}}^\epsilon}\quad\mbox{ on }U_X\times\mathbb R^m\times I
\end{align} is analogous, and second, it always holds that $u_{\textrm{I}}^\epsilon\leq u_{\II}^\epsilon$ by the order of the inf-sup. To start the proof of \eqref{state1}, we assume that Player I follows any strategy and that Player II follows a strategy
$S_{\II}^0$ such that at $(X_{k-1},Y_{k-1})\in U_X\times\mathbb R^m$, he chooses to step to a
point $X_k \in {B_\epsilon(X_{k-1})}$ such that
\begin{align}
u_\epsilon(X_k,Y_{k},t_{k})&=u_\epsilon(X_k,Y_{k-1}+ {\epsilon^2}X_{k}/2,t_{k})\notag\\
&\leq \inf_{\tilde X\in {B_\epsilon(X_{k-1})}} u_\epsilon(\tilde X,Y_{k-1}+ {\epsilon^2}\tilde X/2,t_k)+\eta 2^{-k},
\end{align}
for some fixed $\eta>0$, and where we have recalled that $Y_k=Y_{k-1}+{\epsilon^2}X_{k}/2$ by definition. In other words, Player II tries to almost minimize the value of $u_\epsilon(\cdot,Y_{k-1}+{\epsilon^2}\cdot/2,t_{k})$. When proving \eqref{state2},  Player I instead tries to almost maximize the value of $u_\epsilon(\cdot,Y_{k-1}+{\epsilon^2}\cdot/2,t_{k})$. This type of strategy can be implemented through a Borel measurable function $S_{\II}^0:U_X\times\mathbb R^m\to  U_X^\epsilon$ such that if $(X,Y)\in U_X\times\mathbb R^m$, then $S_{\II}^0(X,Y)\in B_\epsilon(X)$. To see this, we note that due to Lemma \ref{lem:G_eps'_is_Borel-measurable}, the map
\begin{align*}
  X \mapsto u_\epsilon(X, Y + \epsilon^2 X/2, t)
\end{align*}
is Borel measurable for fixed $(Y,t)$. Lemma 3.1 in \cite{luirops14} shows that for any fixed $(Y,t)$ and any $\lambda>0$, one can pick a Borel measurable map $S_Y : U_X \to U^\epsilon_X$ such that $S_Y(X) \in B_\epsilon(X)$ and
\begin{align*}
 u_\epsilon(S_Y(X), Y + \epsilon^2 S_Y(X)/2, t) \leq \inf_{\tilde X\in {B_\epsilon(X)}} u_\epsilon(\tilde X,Y+\epsilon^2\tilde X/2,t) + \lambda/3.
\end{align*}
Let $\delta > 0$. For every $Y\in \R^m$, we denote by $\hat Y$ the unique element in $\delta \Z^m$ for which $Y \in \hat Y + [0,\delta)^m$. Finally, define
\begin{align*}
 S(X,Y) := S_{\hat Y}(X).
\end{align*}
Reasoning as in the proof of Lemma \ref{lem:G_eps'_is_Borel-measurable}, one sees that $S$ is a Borel function. Moreover, we have that
\begin{align*}
 u_\epsilon(S(X,Y), Y + \epsilon^2 S(X,Y)/2,t) &= u_\epsilon(S_{\hat Y}(X), Y + \epsilon^2 S_{\hat Y}(X)/2,t) - u_\epsilon(S_{\hat Y}(X), \hat Y + \epsilon^2 S_{\hat Y}(X)/2,t)
 \\
 &\quad + u_\epsilon(S_{\hat Y}(X), \hat Y + \epsilon^2 S_{\hat Y}(X)/2,t)
 \\
 & \leq u_\epsilon(S_{\hat Y}(X), Y + \epsilon^2 S_{\hat Y}(X)/2,t) - u_\epsilon(S_{\hat Y}(X), \hat Y + \epsilon^2 S_{\hat Y}(X)/2,t)
 \\
 &\quad + \inf_{\tilde X\in {B_\epsilon(X)}} u_\epsilon(\tilde X,\hat Y+\epsilon^2\tilde X/2,t) + \lambda/3
 \\
 & \leq u_\epsilon(S_{\hat Y}(X), Y + \epsilon^2 S_{\hat Y}(X)/2,t) - u_\epsilon(S_{\hat Y}(X), \hat Y + \epsilon^2 S_{\hat Y}(X)/2,t)
 \\
 &\quad + \inf_{\tilde X\in {B_\epsilon(X)}} u_\epsilon(\tilde X,\hat Y+\epsilon^2\tilde X/2,t) - \inf_{\tilde X\in {B_\epsilon(X)}} u_\epsilon(\tilde X, Y+\epsilon^2\tilde X/2,t)
 \\
 & \quad +\inf_{\tilde X\in {B_\epsilon(X)}} u_\epsilon(\tilde X, Y+\epsilon^2\tilde X/2,t) + \lambda/3.
\end{align*}
Due to the uniform continuity property satisfied by $u_\epsilon$ in the $Y$ variable, we see that if we take a sufficiently small $\delta>0$ we have
\begin{align*}
 |u_\epsilon(S_{\hat Y}(X), Y + \epsilon^2 S_{\hat Y}(X)/2,t) - u_\epsilon(S_{\hat Y}(X), \hat Y + \epsilon^2 S_{\hat Y}(X)/2,t)| &< \lambda/3,
 \\
 |\inf_{\tilde X\in {B_\epsilon(X)}} u_\epsilon(\tilde X,\hat Y+\epsilon^2\tilde X/2,t) - \inf_{\tilde X\in {B_\epsilon(X)}} u_\epsilon(\tilde X, Y+\epsilon^2\tilde X/2,t)| &< \lambda/3,
\end{align*}
and thus
\begin{align*}
 u_\epsilon(S(X,Y), Y + \epsilon^2 S(X,Y)/2,t) < \inf_{\tilde X\in {B_\epsilon(X)}} u_\epsilon(\tilde X, Y+\epsilon^2\tilde X/2,t) + \lambda.
\end{align*}
In particular, we could take $\lambda = \eta 2^{-k}$ to obtain the desired strategy $S^0_{\II}$.

 To proceed, recall that $t_{k+1}=t_k-\epsilon^2/2$. We now start from the point $(X_0,Y_0,t_0)\in U_X\times\mathbb R^m\times I$ and we let $N=\lceil{{t_0}/{(\epsilon^2/2)}}\rceil$.  Then, since $u_\epsilon$ is Borel measurable in $(X,Y)$ for fixed $t$ we may estimate,
\begin{align}
&\mathbb{E}_{S_{\textrm{I}}, S^0_{\II}}^{(X_0,Y_0),N}[u_\epsilon(X_k,Y_k,t_k)+\eta 2^{-k}\,|\I_{k-1}]
\notag\\
&\leq \frac{\alpha}{2} \left\{\inf_{ \tilde X\in {B_\epsilon(X_{k-1})}}  u_\epsilon(\tilde X,Y_{k-1}+\epsilon^2\tilde X/2,t_k)+\eta 2^{-k}+\sup_{  \tilde X\in {B_\epsilon(X_{k-1})}}
 u_\epsilon(\tilde X,Y_{k-1}+\epsilon^2\tilde X/2,t_k)\right\}\notag\\
 &+ \beta \barint_{ B_{\eps}(X_{k-1})}
u_\epsilon(\tilde X,Y_{k-1}+\epsilon^2\tilde X/2,t_k) \d \tilde X+\eta 2^{-k},
\end{align}
where we have simply estimated the strategy of Player I by $\sup$ in the definition of the game.  Rewriting
the above display, we have
\begin{align}
&\mathbb{E}_{S_{\textrm{I}}, S^0_{\II}}^{(X_0,Y_0),N}[u_\epsilon(X_k,Y_k,t_k)+\eta 2^{-k}\,|\I_{k-1}]\notag\\
&\leq \frac{\alpha}{2} \left\{\inf_{ \tilde X\in {B_\epsilon(X_{k-1})}}  u_\epsilon(\tilde X,Y_{k-1}+\epsilon^2\tilde X/2,t_{k-1}-\epsilon^2/2)\right\}\notag\\
&+\frac{\alpha}{2} \left\{\sup_{  \tilde X\in {B_\epsilon(X_{k-1})}}
 u_\epsilon(\tilde X,Y_{k-1}+\epsilon^2\tilde X/2,t_{k-1}-\epsilon^2/2)\right\}\notag\\
 &+ \beta \barint_{ B_{\eps}(X_{k-1})}
u_\epsilon(\tilde X,Y_{k-1}+{\epsilon^2}\tilde X/2,t_{k-1}-\epsilon^2/2) \d \tilde X+\eta 2^{-(k-1)}.
\end{align}
Using that $u_\epsilon$ is a $(p,\epsilon)$-Kolmogorov function, we know the value of the right hand side, and we  can conclude that
\begin{align}
\mathbb{E}_{S_{\textrm{I}}, S^0_{\II}}^{(X_0,Y_0),N}[u_\epsilon(X_k,Y_k,t_k)+\eta 2^{-k}\,|\I_{k-1}]\leq u_\epsilon(X_{k-1},Y_{k-1},t_{k-1})+\eta 2^{-(k-1)}.
\end{align}
Thus,
\[
M_k:=u_\epsilon(X_k,Y_k,t_k)+\eta 2^{-k}
\] is a supermartingale with respect to the filtration $\{ {\mathcal I_k}\}_{k\geq 0}$
defined in \eqref{eq:filtration}.  Therefore,  we deduce
\[
\begin{split}
u_{\II}^\epsilon(X_0,Y_0,t_0)&= \inf_{S_{\II}}\sup_{S_{\textrm{I}}}\,\mathbb{E}_{S_{\textrm{I}},S_{{\II}}}^{(X_0,Y_0),N}[F(X_\tau,Y_\tau,t_\tau)]\\
&\leq \sup_{S_{\textrm{I}}}\,\mathbb{E}_{S_{\textrm{I}},S_{{\II}}^0}^{(X_0,Y_0),N}[F(X_\tau,Y_\tau,t_\tau)+\eta 2^{-\tau}]\\
&=\sup_{S_{\textrm{I}}}\,\mathbb{E}_{S_{\textrm{I}},S_{{\II}}^0}^{(X_0,Y_0),N}[u_\epsilon(X_\tau,Y_\tau,t_\tau)+\eta 2^{-\tau}]\\
&\leq \sup_{S_{\textrm{I}}}\,\mathbb{E}_{S_{\textrm{I}},S_{{\II}}^0}^{(X_0,Y_0),N}[M_0]
=u_\epsilon(X_0,Y_0,t_0)+\eta.
\end{split}
\]
In this deduction, we have used that $\tau$ is finite as $T$ is finite, which allowed us to use
 the optional stopping theorem for $M_{k}$. As $\eta$ is arbitrary, the proof is complete.
\end{proof}

\begin{remark}\label{DPP} Using the above results, and reconsidering the tug-of-war game, it follows that the value function for Player I satisfies
\begin{align*}
 u^{\eps,N}_{\textrm{I}}(X,Y,k) &= \frac \alpha 2\biggl\{\sup_{\tilde X\in{{{B_\epsilon(X)}}}}u^{\eps,N}_{\textrm{I}}(\tilde X,Y+\epsilon^2\tilde X/2,k+1)\biggr \}
 \\
 &+\frac \alpha 2\biggl\{\inf_{\tilde X\in{{{B_\epsilon(X)}}}}u^{\eps,N}_{\textrm{I}}(\tilde X,Y+{\epsilon^2}\tilde X/2, k+1)\biggr \}
 \\
&+\beta \barint_{B_\epsilon(X)} u^{\eps,N}_{\textrm{I}}(\tilde X,Y+{\epsilon^2}\tilde X/2, k+1) \d \tilde X,
\end{align*}
for every $(X,Y)\in U_X\times\mathbb R^m$ and $k \in\{0,1,...,N-1\}$, and
\begin{align*}
 u^{\eps,N}_{\textrm{I}}(X,Y,k)&=\mathcal{F}(X,Y,k)
\end{align*}
if $(X,Y)\in \Gamma_X^\epsilon\times\mathbb R^m$ or $k=N$. The value function for Player II, $u^{\eps,N}_{\II}$, satisfies the same statements. This is the Dynamic Programming Principle (DPP) for the tug-of-war game with a maximum number of rounds, and we  note that the expectation is
obtained by summing up the expectations of the following three possible outcomes for the next step
with the corresponding probabilities: Player I chooses the next position (probability
$\alpha/2$), Player II chooses the next (probability $\alpha/2$), or the next position is random (probability
$\beta$). To reiterate Subsection \ref{subsec:timechange}, note that while the definition of a  $(p,\epsilon)$-Kolmogorov function $u_\epsilon$ given in Definition \ref{meanvalue1thmgaiandef} refers to a forward-in-time parabolic equation as $u_\epsilon(\cdot,\cdot, t)$ is  determined by past values $u_\epsilon(\cdot,\cdot, t-\epsilon^2/2)$, in the stated DPP, the values at step $k$ are instead determined by future values at step $k + 1$.
\end{remark}

\section{The Dirichlet problem: existence and uniqueness}\label{sec4}


In this section, we want to investigate what happens to the $(p,\epsilon)$-Kolmogorov function $u_\epsilon$ in the limit $\epsilon\to 0$. In particular, we want to prove the existence of a limit function which is in fact a viscosity solution to \eqref{Dppa}.
To make this operational, we have to establish quantitative continuity estimates with constants that are independent of $\epsilon$, for $\epsilon$ small. 

We let
$$R:=\max\{1,\max_{X\in\overline{U_X}}|X|\}\,$$
and, given
$U_X\times \mathbb R^m\times I$, we introduce
\begin{align*}
\Gamma^1&:=\partial U_X\times\mathbb R^m\times[0,T),\  \Gamma^2:=(U_X\times \mathbb R^m)\times\{0\}.
\end{align*}

Given $F\in C(\Gamma^1\cup\Gamma^2)$, we are concerned with the existence of solutions to the boundary value problem
\begin{align}\label{Dppa}
\begin{cases}
\K_pu=0  ,\quad  &\textrm{for}\quad (X,Y,t)\in U_X\times \mathbb R^m\times I,\\
u   =  F
,\quad  &\textrm{for}\quad  (X,Y,t)\in \Gamma^1\cup\Gamma^2.
\end{cases}
\end{align}
The function $u\in C(\overline{U_X}\times \R^m \times [0,T)$) is said to be a viscosity solution to \eqref{Dppa} if $u$ is a viscosity solution to $\K_pu=0$ in $U_X\times\mathbb R^m\times I$ and if $u(X,Y,t)=F(X,Y,t)$ for all
$(X,Y,t)\in \Gamma^1\cup\Gamma^2$. Concerning $F$, we will initially assume $F\in\mathcal{G}_{\epsilon_0}$, for some $\epsilon_0>0$ fixed, and that
\begin{align}\label{regu}
|F(X,Y,t)-F(\hat X,\hat Y,\hat t)|\leq cd_\K((X,Y,t),(\hat X,\hat Y,\hat t))\approx cd((X,Y,t),(\hat X,\hat Y,\hat t))
\end{align}
for all $(X,Y,t),(\hat X,\hat Y,\hat t)\in \Gamma_{\epsilon_0}$, and for some constant $c$. Recall that $d_\K$ was introduced in
\eqref{e-ps.distint}, see also \eqref{e-ps.dist}, and here,
\begin{align}\label{regu+a}
d((X,Y,t),(\hat X,\hat Y,\hat t)):=|X-\hat X|+|Y-\hat Y-(\hat t-t)\hat X|^{1/3}+|\hat t-t|^{1/2}.
\end{align}
Thus, we first investigate the case where $F$ is defined on the larger set $\Gamma_{\epsilon_0}\supset \Gamma^1\cup \Gamma^2$ and where $F$ is also Lipschitz with respect to the quasi-metric $d_K$. Later, in Corollary \ref{cor:F_only_def_on_par_bdry}, we return to the case where $F$ is only defined a priori on $\Gamma^1 \cup \Gamma^2$.

Let $u_\epsilon$ be the value function of the tug-of-war game with payoff equal to $F$ on $\Gamma_\epsilon$, when $0 < \epsilon \leq \epsilon_0$. We intend to prove, conditioned on the additional regularity of $F$, that
$u_\epsilon \rightarrow u$ as $\epsilon \to 0$ and that $u$ is a viscosity solution to \eqref{Dppa}. Note that the functions $\{u_\epsilon\}$ are in  general not  continuous but, as it turns out, their discontinuities can be controlled, and we will show that the value functions are asymptotically uniformly continuous. We will use the following  lemma which is a  variant of the
classical Arzela-Ascoli's compactness lemma.

\begin{lemma}\label{Ascoli} Let $K \subset \R^{M+1}$ be compact and assume that
$$\left\{u_\eps:K \rightarrow \R,\ \eps>0\right\}$$ is a set of uniformly bounded functions which satisfies the following. Given $\eta>0$,  there exist constants $\rho_0$ and $\eps_0$ such that if  $\eps<\eps_0$, and if $(X,Y,t),(\hat X,\hat Y,\hat t)\in K$ satisfy
$$d((X,Y,t),(\hat X,\hat Y,\hat t))<\rho_0,$$
then
\[
|u_\eps(X,Y,t)-u_\eps(\hat X,\hat Y,\hat t)|<\eta.
\]
Then there exists a uniformly continuous function $v:K \rightarrow \R$ and a sequence $\epsilon_j \downarrow 0$, such that $u_{\epsilon_j} \rightarrow v$ uniformly in $K$ as $j \to \infty$.
\end{lemma}
\begin{proof} Pick an arbitrary sequence $\epsilon_j \downarrow 0$ and let $\mathcal{X}\subset K$ be a dense countable set. By assumption, $\left\{u_\eps:K \rightarrow \R,\ \eps>0\right\}$ is a set of uniformly bounded functions. Hence, a diagonal procedure provides a subsequence, still denoted by $(u_{\epsilon_j})$,  that converges at every point $(X,Y,t)\in \mathcal{X}$. We let $v(X,Y,t)$ denote this limit for $(X,Y,t)\in \mathcal{X}$. The definition of $v$ shows that, given $\eta>0$,  there exist constants $\rho_0$ such that if $(X,Y,t),(\hat X,\hat Y,\hat t)\in \mathcal{X}$ satisfy
$$d((X,Y,t),(\hat X,\hat Y,\hat t))<\rho_0,$$
then
\begin{align*}
|v(X,Y,t)-v(\hat X,\hat Y,\hat t)|<\eta.
\end{align*}
We can thus, by density of $\mathcal{X}$ in $K$, extend $v$  to a uniformly continuous function on all of $K$  by defining
$$v(X,Y,t):=\lim_{(\hat X,\hat Y,\hat t)\in \mathcal{X}, (\hat X,\hat Y,\hat t)\to (X,Y,t)} v(\hat X,\hat Y,\hat t).$$
The next step is to prove that $u_{\epsilon_j}$ converges to $v$ uniformly. We choose
a finite covering
$$K\subset\cup_{i=1}^l \mathcal{B}_{r_i}(X_i,Y_i,t_i),$$
with $(X_i, Y_i, t_i) \in \mathcal{X}$ and $\epsilon_0>0$, such that
$$|u_\epsilon(X,Y,t)-u_\epsilon(X_i,Y_i,t_i)|<\eta/3,\ |v(X,Y,t)-v(X_i,Y_i,t_i)|<\eta/3,$$
for all $(X,Y,t)\in K \cap \mathcal{B}_{r_i}(X_i,Y_i,t_i)$ and $\epsilon<\epsilon_0$, and such that also
$$|u_{\epsilon_j}(X_i,Y_i,t_i)-v(X_i,Y_i,t_i)|<\eta/3,$$
for all $i \in \{1,\dots, l\}$ and $j$ sufficiently large. To obtain the last inequality, we used the fact that $l<\infty$ and the fact that $v$ is the pointwise limit of $u_{\epsilon_j}$. Thus for any
$(X,Y,t)\in K$ we can find $i$, such that $(X,Y,t)\in K \cap \mathcal{B}_{r_i}(X_i,Y_i,t_i)$ and therefore
\begin{align*}
|u_{\epsilon_j}(X,Y,t)-v(X,Y,t)|&\leq |u_{\epsilon_j}(X,Y,t)-u_{\epsilon_j}(X_i,Y_i,t_i)|+|u_{\epsilon_j}(X_i,Y_i,t_i)-v(X_i,Y_i,t_i)|\\
&+|v(X_i,Y_i,t_i)-v(X,Y,t)|<\eta,
\end{align*}
 for $j$ sufficiently large. This proves that $u_{\epsilon_j}$ converges to $v$ uniformly on $K$.
\end{proof}

In Lemma \ref{cruc} and Lemma \ref{cruc++}  below, we establish the
 uniform continuity, i.e., the assumption of Lemma \ref{Ascoli}  near $\Gamma^1$ and $\Gamma^2$, respectively. Using these lemmas and an argument based on
 Theorem \ref{u1} and the comparison principle proved in  Lemma \ref{u4}, we establish the same conclusion on compact subsets of $U_X\times\mathbb R^m\times I$.

  Recall the tug-of-war game and strategies described in Subsection \ref{subsec:ToW_game_descr}. To give a   general outline of the proofs of the lemmas, the idea is to compare the value function $u_\eps(X_0,Y_0,t_0)$, where
  $(X_0,Y_0,t_0)\in U_X\times\mathbb R^m\times I$,
   with the boundary values $F(\hat X_0,\hat Y_0,\hat t_0)$, where
  $(\hat X_0,\hat Y_0,\hat t_0)\in \mathbb R^{M+1}\setminus(U_X\times\mathbb R^m\times I)$, i.e., the values of the game starting at $(X_0,Y_0,t_0)$ and the boundary values at $(\hat X_0,\hat Y_0,\hat t_0)$, respectively, by the construction of appropriate strategies. Based on a stopping rule, basically defined as the first exit time of the game process from $U_X\times\mathbb R^m\times I$, the game will stop at some random time $\tau=\tau_{S_{\textrm{I}}, S_{\II}}$, where we indicate that the stopping time, measured in terms of the number of rounds played, is a function of the strategies of the players.

 Let $(\hat X_0,\hat Y_0,\hat t_0)$ be a point outside the domain where $u_\eps(\hat X_0,\hat Y_0,\hat t_0)=F(\hat X_0,\hat Y_0,\hat t_0)$. Let $S^0_{\textrm{I}}$ and $S^0_{\II}$ refer to some fixed strategies for Player I and Player II respectively. Then we may estimate
\begin{align*}
&u_\eps(X_0,Y_0,t_0)-F(\hat X_0,\hat Y_0,\hat t_0)\notag\\
&\geq  \inf_{S_{\II}}\mathbb{E}^{( X_0,Y_0, t_0)}_{S^0_{\textrm{I}},S_{\II}}[F(X_{\tau_{S^0_{\textrm{I}}, S_{\II}}},Y_{\tau_{S^0_{\textrm{I}}, S_{\II}}},t_0-{\tau_{S^0_{\textrm{I}}, S_{\II}}}\eps^2/2)-F(\hat X_0,\hat Y_0,\hat t_0)].
\end{align*}
Similarly,
\begin{align*}
&u_\eps(X_0,Y_0,t_0)-F(\hat X_0,\hat Y_0,\hat t_0)\notag\\
&\leq \sup_{S_{\textrm{I}}}\mathbb{E}^{( X_0, Y_0, t_0)}_{S_{\textrm{I}},S^0_{\II}}[F(X_{\tau_{S_{\textrm{I}},S^0_{\II}}},Y_{\tau_{S_{\textrm{I}},S^0_{\II}}}, t_0-{\tau_{S_{\textrm{I}},S^0_{\II}}}\eps^2/2)-F(\hat X_0,\hat Y_0,\hat t_0)].
\end{align*}
 Hence, using \eqref{regu}, we obtain
\begin{align}\label{st4}
&|u_\eps(X_0,Y_0,t_0)-F(\hat X_0,\hat Y_0,\hat t_0)|\notag\\
&\leq \sup_{S_{\II}}\mathbb{E}^{( X_0, Y_0, t_0)}_{S^0_{\textrm{I}},S_{\II}}[d((X_{\tau_{S^0_{\textrm{I}}, S_{\II}}},Y_{\tau_{S^0_{\textrm{I}}, S_{\II}}},t_0-{\tau_{S^0_{\textrm{I}}, S_{\II}}}\eps^2/2),
(\hat X_0,\hat Y_0,\hat t_0))]\notag\\
&+\sup_{S_{\textrm{I}}}\mathbb{E}^{( X_0,Y_0, t_0)}_{S_{\textrm{I}},S^0_{\II}}[d((X_{\tau_{S_{\textrm{I}},S^0_{\II}}},Y_{\tau_{S_{\textrm{I}},S^0_{\II}}},t_0-{\tau_{S_{\textrm{I}},S^0_{\II}}}\eps^2/2),
(\hat X_0,\hat Y_0,\hat t_0))],
\end{align}
where $d$ was introduced in \eqref{regu+a}. In essence, to prove estimates, we have to estimate the right hand side in \eqref{st4} by simply developing upper bounds on
\begin{align*}
&\mathbb{E}^{( X_0,Y_0,t_0)}_{S^0_{\textrm{I}},S_{\II}}[d((X_{\tau_{S^0_{\textrm{I}}, S_{\II}}},Y_{\tau_{S^0_{\textrm{I}}, S_{\II}}},t_0-{\tau_{S^0_{\textrm{I}}, S_{\II}}}\eps^2/2),
(\hat X_0,\hat Y_0,\hat t_0))],\notag\\
&\mathbb{E}^{( X_0,Y_0, t_0)}_{S_{\textrm{I}},S^0_{\II}}[d((X_{\tau_{S_{\textrm{I}},S^0_{\II}}},Y_{\tau_{S_{\textrm{I}},S^0_{\II}}},t_0-{\tau_{S_{\textrm{I}},S^0_{\II}}}\eps^2/2),
(\hat X_0,\hat Y_0,\hat t_0))],
\end{align*}
which are independent of $S_{\II}$ and $S_{\textrm{I}}$. By symmetry, it suffices to estimate one of these terms.

\begin{lemma}\label{cruc}
Given $\eta>0$, there exist $\rho_0>0$ and $\eps_1>0$ such that if $(\hat X,\hat Y,\hat t)\in\Gamma^1$,
$(X,Y,t)\in U_X\times\mathbb R^m\times I$,  $\eps<\eps_1$ and $d((X,Y,t),(\hat X,\hat Y,\hat t))<\rho_0,$ then
\[
|u_\eps(X,Y,t)-u_\eps(\hat X,\hat Y,\hat t)|<\eta.
\]
\end{lemma}
\begin{proof} Let $(\hat X,\hat Y,\hat t)\in\Gamma^1=(\partial U_X\times\mathbb R^m\times[0,T))$. We first note that \[
|u_{\eps}(X,Y,t)-u_{\eps}(\hat X,\hat Y,\hat t)|\leq |u_{\eps}(X,Y,t)-u_{\eps}(\hat X,Y,t)|+|u_{\eps}(\hat X,Y,t)-u_{\eps}(\hat X,\hat Y,\hat t)|,
\]
and that
\begin{align*}
|u_{\eps}(\hat X,Y,t)-u_{\eps}(\hat X,\hat Y,\hat t)|&=|F(\hat X,Y,t)-F(\hat X,\hat Y,\hat t)|\notag\\
&\leq c(|Y-\hat Y-(\hat t-t)\hat X|^{1/3}+|\hat t-t|^{1/2})<\eta,
\end{align*}
 if $\rho_0$ is small enough, using the uniform continuity of the boundary data $F$. Hence, we only have to estimate
$|u_{\eps}(X,Y,t)-u_{\eps}(\hat X, Y,t)|$ and we let $(X_0,Y_0,t_0):=(X,Y,t)$. Recall that
$$u_{\eps}(X_0,Y_0,t_0)= \sup_{S_{\textrm{I}}}\inf_{S_{{\II}}}\,\mathbb{E}_{S_{\textrm{I}},S_{{\II}}}^{(X_0,Y_0,t_0)}[F(X_\tau,Y_\tau,t_\tau)].$$
Using the regularity of $U_X$, we can conclude that
$U_X$ satisfies the uniform exterior sphere condition, i.e., $\hat X\in \partial B_\delta(Z)$ for some $B_\delta(Z)\subset \R^m\setminus U_X$ and for some $\delta\in (0,1)$ independent of $\hat X$. 
We now start the game at $(X_0,Y_0,t_0)$. Player I uses a strategy $S^0_\text{I}$ which implies that Player I is consistently trying to pull the game towards $Z$ and and Player II uses a strategy $S_\text{II}$. We first want to estimate
\begin{align*}
\mathbb{E}& ^{(X_0,Y_0,t_0)}_{S^0_\text{I},S_\text{II}}[d((X_k,Y_k,t_k),(Z,Y,t))|\mathcal{F}_{k-1}].
\end{align*}
Let $w_k:=(X_{k-1}-Z)/|Z-X_{k-1}|$.  Note that
$$X_{k-1}-Z-\epsilon w_k=w_k(|Z-X_{k-1}|-\epsilon).$$
Using this observation, the above, and the construction of the game, which at $X_{k-1}$ implies that if Player I is to modify the game then he adds  vector $-\epsilon w_k$, to the current game velocity state, so that $|X_{k-1}-\epsilon w_k-Z| = |X_{k-1}-Z|-\epsilon$ and we deduce
\begin{align*}
\mathbb{E}^{(X_0,Y_0,t_0)}_{S^0_\text{I},S_\text{II}}[|X_k-Z||\mathcal{F}_{k-1}]&\leq \frac \alpha 2\biggl (|X_{k-1}-Z|-\epsilon+|X_{k-1}-Z|+\epsilon\biggr)+\beta(|X_{k-1}-Z|+c\epsilon)\notag\\
&\leq |X_{k-1}-Z|+c\epsilon.
\end{align*}
By construction,
\begin{align*}
\mathbb{E}^{(X_0,Y_0,t_0)}_{S^0_\text{I},S_\text{II}}[|Y_{k}-Y-(t-t_{k})Z||\mathcal{F}_{k-1}]&=|Y_{k-1}+{\epsilon^2} X_{k}/2-Y-(t-t_{k-1}+{\epsilon^2}/2)Z|\\
&\leq |Y_{k-1}-Y-(t-t_{k-1})Z|+{\epsilon^2}|X_{k}-Z|/2\\
&\leq |Y_{k-1}-Y-(t-t_{k-1})Z|+c\epsilon^2R.
\end{align*}
In particular, if we let
$$M_k^X:=|X_k-Z|-ck\epsilon,\ M_k^Y:=|Y_{k}-Y-(t-t_{k})Z|-ck{\epsilon^2}R, $$
then our conclusions can be stated
$$\mathbb{E}^{(X_0,Y_0,t_0)}_{S^0_\text{I},S_\text{II}}[M_k|\mathcal{F}_{k-1}]\leq M_{k-1},$$
where either $M_k=M_k^X$ or  $M_k^Y$, i.e., both $M_k^X$ and $M_k^Y$ are supermartingales. Using the optional stopping theorem and H{\"o}lder's inequality, we can therefore conclude that
\begin{align*}
\mathbb{E}^{(X_0,Y_0,t_0)}_{S^0_\text{I},S_\text{II}}[d((X_\tau,Y_\tau,t_\tau),(Z,Y,t))]&\leq \bigl(|X_0-Z|+c\epsilon\mathbb{E}^{(X_0,Y_0,t_0)}_{S^0_\text{I},S_\text{II}}[\tau]\bigr )\\
&+\bigl(|Y_{0}-Y-(t-t_{0})Z|+c{\epsilon^2}R\mathbb{E}^{(X_0,Y_0,t_0)}_{S^0_\text{I},S_\text{II}}[\tau]\bigr)^{1/3}\\
&+\epsilon\bigl(\mathbb{E}^{(X_0,Y_0,t_0)}_{S^0_\text{I},S_\text{II}}[\tau]\bigr)^{1/2}.
\end{align*}
Hence, using that
 $(a+b)^{1/q}\leq a^{1/q}+b^{1/q}$ whenever $a$ and $b$ are non-negative real numbers, for all $q\geq 1$, we have
 \begin{align*}
\mathbb{E}^{(X_0,Y_0,t_0)}_{S^0_\text{I},S_\text{II}}[d((X_\tau,Y_\tau,t_\tau),(Z,Y,t))]&\leq
|X-Z|+c\bigl(\epsilon^2\mathbb{E}^{(X_0,Y_0,t_0)}_{S^0_\text{I},S_\text{II}}[\tau]\bigr )^{1/2}\\
&+c\bigl(\epsilon^2R\mathbb{E}^{(X_0,Y_0,t_0)}_{S^0_\text{I},S_\text{II}}[\tau]\bigr )^{1/3}.
\end{align*}
We need to estimate $\mathbb{E}^{(X_0,Y_0,t_0)}_{S^0_\text{I},S_\text{II}}[\tau]=\mathbb{E}^{(X_0,Y_0,t_0)}_{S^0_\text{I},S_\text{II}}[\tau_{S^0_\text{I},S_\text{II}}]$. Let
$\hat \tau=\hat\tau_{S^0_\text{I},S_\text{II}}$ be the corresponding stopping time in the case $U_X\times\mathbb R^m$ is replaced by $(B_{R_0}(Z)\setminus\overline{B_\delta(Z)})\times \mathbb R^m$,  where $R_0>0$ is chosen so that $U_X\subset B_{R_0}(Z)$. Here it is understood that the strategy $S^0_\text{I}$ has been extended so that it still pulls the token towards $Z$, and the extension of $S_\text{II}$ is arbitrary. Then $\tau=\tau_{S^0_\text{I},S_\text{II}}\leq \hat \tau_{S^0_\text{I},S_\text{II}}=\hat\tau$ is a conservative upper bound. Next, we note that if a pure tug-of-war game was played instead of a tug-of-war game with noise,
then we would have
\begin{align}\label{esta1}
\mathbb{E}^{(X_0,Y_0,t_0)}_{S^0_\text{I},S_\text{II}}[|X_k-Z||\mathcal{F}_{k-1}]&\leq |X_{k-1}-Z|,
\end{align}
while if a pure random walk occurred, then, in this case, the strategies lack impact,
\begin{align}\label{esta2}
\mathbb{E}^{(X_0,Y_0,t_0)}_{S^0_\text{I},S_\text{II}}[|X_k-Z||\mathcal{F}_{k-1}]&=\barint_{B_\epsilon(X_{k-1})} |X-Z|\, \d X\notag\\
&=\barint_{B_\epsilon(0)} |X+(X_{k-1}-Z)|\, \d X\geq |X_{k-1}-Z|.
\end{align}
Let $\tau^\ast$ be the first exit time from $(B_{R_0}(Z)\setminus\overline{B_\delta(Z)})\times \mathbb R^m$ of the random walk process. We claim that
\eqref{esta1} and \eqref{esta2} imply that
\begin{align*}
\mathbb{E}^{(X_0,Y_0,t_0)}_{S^0_\text{I},S_\text{II}}[\hat\tau_{S^0_\text{I},S_\text{II}}]\leq \mathbb{E}^{(X_0,Y_0,t_0)}_{S^0_\text{I},S_\text{II}}[\tau^\ast].
\end{align*}
Therefore, to estimate $\mathbb{E}^{(X_0,Y_0,t_0)}_{S^0_\text{I},S_\text{II}}[\tau]$ we can use a conservative and simply bound $\mathbb{E}^{(X_0,Y_0,t_0)}_{S^0_\text{I},S_\text{II}}[\tau^\ast]$. To do this, we can immediately reuse the parabolic result, and in particular the elliptic estimate stated in Lemma 14 in \cite{manfredipr10c}, to conclude the conservative estimate
\begin{equation*}
\mathbb{E}^{(X_0,Y_0,t_0)}_{S^0_\text{I},S_\text{II}}[\tau]\leq \min\{\frac{c(R_0/\delta)\dist(\partial B_\delta (Z),X_0)+o(1)}{\eps^2},N\},
\end{equation*}
where $R_0>0$ is chosen so that $U_X\subset B_{R_0}(Z)$, and $o(1)\rightarrow 0$ when $\eps\rightarrow 0$. Put together, noting that $\dist(\partial B_\delta (Z),X_0)\leq |X-\hat X|$, we have
\begin{align*}
\mathbb{E}^{(X_0,Y_0,t_0)}_{S^0_\text{I},S_\text{II}}[d((X_\tau,Y_\tau,t_\tau),(Z,Y,t))]&\leq
|X-Z|+\bigl(\min\{cR(c(R_0/\delta)|X-\hat X|+o(1)),\epsilon^{2}N\}\bigr )^{1/3},
\end{align*}
if $c(R_0/\delta)|X-\hat X|$ is small. To use the estimate, we note that
\begin{align*}
u_{\eps}(X,Y,t)-u_{\eps}(\hat X,Y,t) &\geq -c\delta+\sup_{S_{\textrm{I}}}\inf_{S_{{\II}}}\,\mathbb{E}_{S_{\textrm{I}},S_{{\II}}}^{(X_0,Y_0,t_0)}[F(X_\tau,Y_\tau,t_\tau)-F(Z, Y,t)]
\\
&\geq -c\delta+\inf_{S_{{\II}}}\,\mathbb{E}_{S^0_{\textrm{I}},S_{{\II}}}^{(X_0,Y_0,t_0)}[F(X_\tau,Y_\tau,t_\tau)-F(Z, Y,t)]
\\
&\geq -c\delta-c\inf_{S_{{\II}}}\,\mathbb{E}_{S^0_{\textrm{I}},S_{{\II}}}^{(X_0,Y_0,t_0)}[d((X_\tau,Y_\tau,t_\tau),(Z,Y,t))]
\\
&\geq -c\delta-c\bigl(\min\{cR(c(R_0/\delta)|X-\hat X|+o(1)),\epsilon^{2}N\}\bigr )^{1/3}.
\end{align*}
Similarly, interchanging the roles of Player I and Player II, we deduce that
\begin{align*}
u_{\eps}(X,Y,t)-u_{\eps}(\hat X,Y,t)&\leq c\delta+c\bigl(\min\{cR(c(R_0/\delta)|X-\hat X|+o(1)),\epsilon^{2}N\}\bigr )^{1/3},
\end{align*}
and, hence,
\begin{align*}
|u_{\eps}(X,Y,t)-u_{\eps}(\hat X,Y,t)|&\leq c\delta+c\bigl(\min\{cR(c(R_0/\delta)|X-\hat X|+o(1)),\epsilon^{2}N\}\bigr )^{1/3}.
\end{align*}
Note that $\epsilon^{2}N\leq 2T$. Let $\eta>0$ be small, and let $\delta=\eta/(16c)$, $\rho_0=\eta^4/(16c^4RR_08^3)$. Then, assuming $d((X,Y,t),(\hat X,\hat Y,\hat t))<\rho_0$,
\begin{align*}
|u_{\eps}(X,Y,t)-u_{\eps}(\hat X,Y,t)|&\leq \frac \eta 8+\frac \eta 8+(o(1))^{1/3}<\eta,
\end{align*}
if $\epsilon<\epsilon_1$, and the proof is complete in this case. \end{proof}

\begin{lemma}\label{cruc++}
Given $\eta>0$, there exist $\rho_0>0$ and $\eps_1>0$ such that if $(\hat X,\hat Y,\hat t)\in\Gamma^2$,
$(X,Y,t)\in U_X\times\mathbb R^m\times I$,  $\eps<\eps_1$ and $d((X,Y,t),(\hat X,\hat Y,\hat t))<\rho_0,$ then
\[
|u_\eps(X,Y,t)-u_\eps(\hat X,\hat Y,\hat t)|<\eta.
\]
\end{lemma}
\begin{proof} Let $(\hat X,\hat Y,\hat t)\in\Gamma^2=(U_X\times \mathbb R^m)\times\{0\}$.  Let  $(X_0,Y_0,t_0):=(X,Y,t)$. We start the game at $(X_0,Y_0,t_0)$  and we fix for Player I a strategy $S^0_\text{I}$ which in this case implies that Player I is consistently trying to pull the game towards $\hat X$. Player II uses a strategy $S_\text{II}$. Arguing as in Lemma \ref{cruc} we deduce that
\begin{align*}
\mathbb{E}^{(X_0,Y_0,t_0)}_{S^0_\text{I},S_\text{II}}[d((X_\tau,Y_\tau,t_\tau),(\hat X,\hat Y,\hat t))]&\leq \bigl(|X-\hat X|^2+c\epsilon^2\mathbb{E}^{(X_0,Y_0,t_0)}_{S^0_\text{I},S_\text{II}}[\tau]\bigr )^{1/2}\\
&+\bigl(|Y-\hat Y-(\hat t-t)\hat X|+c{\epsilon^2}R\mathbb{E}^{(X_0,Y_0,t_0)}_{S^0_\text{I},S_\text{II}}[\tau]\bigr)^{1/3}\\
&+\epsilon\bigl(\mathbb{E}^{(X_0,Y_0,t_0)}_{S^0_\text{I},S_\text{II}}[\tau]\bigr)^{1/2}.
\end{align*}
Since the stopping time is bounded by $t/(\epsilon^2/2) + 1$, we obtain
\begin{align*}
\mathbb{E}^{(X_0,Y_0,t_0)}_{S^0_\text{I},S_\text{II}}[d((X_\tau,Y_\tau,t_\tau),(\hat X,\hat Y,\hat t))]&\leq \bigl(|X-\hat X|^2+c(t+\epsilon^2)\bigr )^{1/2}\\
&+\bigl(|Y-\hat Y-(\hat t-t)\hat X|+c(t+\epsilon^2)\bigr)^{1/3}+\bigl(t+\epsilon^2\bigr)^{1/2}\notag\\
&\leq c|X-\hat X|+c|Y-\hat Y-(\hat t-t)\hat X|^{1/3}+c(t^{1/3}+\epsilon^{2/3}).
\end{align*}
Using this, and arguing as in the final part of the proof of Lemma \ref{cruc}, we deduce
\begin{align*}
|u_{\eps}(X,Y,t)-u_{\eps}(\hat X,\hat Y,\hat t)|&\leq c(|X-\hat X|+|Y-\hat Y-(\hat t-t)\hat X|^{1/3}+t^{1/3}+\epsilon^{2/3}),
\end{align*}
based on which we can  complete the proof of the lemma.\end{proof}

\begin{lemma} \label{conver} Let $\{u_{\eps}\}$, $\eps >0$ be the value functions of the tug-of-war game with payoff equal to $F$ on $\Gamma_\epsilon$. Then there is a sequence $\eps_j \downarrow 0$ for which the corresponding functions $u_{\eps_j}$ converge uniformly on compact subsets of $\overline{U_X}\times \R^m \times [0,T]$ to a continuous limit function $u: \overline{U_X}\times\mathbb R^m\times [0,T] \to \R$. Moreover, $u(X,Y,t)=F(X,Y,t)$ whenever $(X,Y,t)\in \Gamma^1\cup\Gamma^2$.
\end{lemma}
\begin{proof} To prove the lemma, we will use Lemma \ref{Ascoli}, Lemma \ref{cruc}, Lemma \ref{cruc++} and the comparison principle established in Lemma \ref{u4}. In particular, we need to verify the assumptions of Lemma \ref{Ascoli}. We first note that as $|u_\eps|\leq \max |F|$, the functions $\{u_{\eps}\}$ are uniformly bounded.  Let $\eta > 0$. Suppose that  $(X,Y,t),(\hat X,\hat Y,\hat t)\in \overline{U_X}\times\mathbb R^m\times [0,T]$.
For $\rho>0$ small, we introduce the notation
$$U_X(\rho):=\{X\in {U_X}:\ d(X,\partial U_X)>\rho\},\ I^{\rho}:=\{t\in I:\ t>\rho^2\},$$
and the strip
$$S_{\rho}:= (\overline{U_X}\times \mathbb R^m\times [0,T])\setminus (U_X(\rho) \times \mathbb R^m\times I^\rho).$$
 Using Lemma \ref{cruc}, Lemma \ref{cruc++}, recalling that $u_\epsilon$ coincides with $F$ on $\Gamma^1 \cup \Gamma^2$ and simply comparing boundary values using \eqref{regu}, we see that  given $\eta>0$, there exist $\rho_1>0$ and $\eps_1>0$ such that if  $\eps<\eps_1$ and $d((X,Y,t),(\hat X, \hat Y, \hat t)) < \rho_1$ then
\[
|u_\eps(X,Y,t)-u_\eps(\hat X,\hat Y,\hat t)|<\eta
\]
whenever $(X,Y,t)\in S_{\rho_1}$ or $(\hat X,\hat Y,\hat t)\in S_{\rho_1}$. By taking a smaller $\epsilon_1$ if necessary, we may also assume that $\epsilon_1 < \rho_1/2$. It remains to consider the case where the points $(X,Y,t), (\hat X,\hat Y,\hat t)$ both belong to $U_X(\rho_1)\times\mathbb R^m\times I^{\rho_1}$. Define
$$(\hat Z,\hat W,\hat \tau):=(\hat X,\hat Y,\hat t)\circ(X,Y,t)^{-1}.$$
Without loss of generality, we may assume that $\hat t \leq t$ so that $\hat \tau \leq 0$. A calculation shows that if $\rho_0 > 0$ and if
\begin{align}\label{distcond_rho0}
 d((X,Y,t),(\hat X, \hat Y, \hat t)) < \rho_0
\end{align}
then
\begin{align*}
\|(\hat Z,\hat W,\hat \tau)\|\leq c(m,R,T)d^{1/3}\leq c \rho_0^{1/3}.
\end{align*}
From this it follows that if $(Z,W,\tau)\in E:= U_X(\rho_1)^\epsilon \times\mathbb R^m\times (\rho^2_1 - \epsilon^2 / 2, T]$, and if \eqref{distcond_rho0} holds with a sufficiently small $\rho_0$ then $(\hat Z,\hat W,\hat \tau)\circ(Z,W,\tau) \in U_X \times \R^m \times (0,T]$. Here, the superscript $\epsilon$ takes the same meaning as in Section \ref{sec3}, i.e., we expand the set $U_X(\rho_1)$ by $\epsilon$ in all directions. Note that in this argument we also use the fact that $\epsilon < \epsilon_1 < \rho_1/2$, so that $E$ is at a positive distance (depending only on $\rho_1$) from the parabolic boundary of $U_X \times \R^m \times (0,T]$. The assumption $\hat \tau \leq 0$ was needed to guarantee that $\tau + \hat \tau \leq T$. Thus, we can define
\begin{align*}\tilde u_\epsilon(Z,W,\tau)&:=u_\epsilon((\hat Z,\hat W,\hat \tau)\circ(Z,W,\tau)) + \eta\\
&=u_\epsilon(Z+\hat Z,W+\hat W-\tau\hat Z,\hat\tau+\tau) + \eta,
\end{align*}
for all $(Z,W,\tau)\in E$. Note that the domain of $\tilde u_\epsilon$ is the union of $U_X(\rho_1)\times \R^m \times (\rho_1^2, T]$ and its parabolic $\epsilon$-boundary $\tilde \Gamma_\epsilon$, defined in a way which is analogous to Section \ref{sec3}. To be explicit, we define
\begin{align*}
 \tilde \Gamma_\epsilon &:= \tilde \Gamma^1_\epsilon \cup \tilde \Gamma^2_\epsilon,
 \\
 \tilde \Gamma^1_\epsilon &:= \tilde \Gamma^\epsilon_X \times \R^m \times (\rho_1^2 - \epsilon^2 / 2, T],
 \\
 \tilde \Gamma^2_\epsilon &:= U_X(\rho_1) \times \R^m \times (\rho_1^2 - \epsilon^2 /2, \rho_1^2],
 \\
 \tilde \Gamma^\epsilon_X &:= \{ X \in \R^m \setminus U_X(\rho_1) : d(X, \partial U_X(\rho_1)) \leq \epsilon \}.
\end{align*}
We furthermore set $\tilde F := \tilde u_\epsilon |_{\tilde \Gamma_\epsilon}$. Note that since $u_\epsilon \in \mathcal{G}_\epsilon'$, the properties defining this function class are inherited by $\tilde u_\epsilon$, with appropriate modifications taking into account the domain of $\tilde u_\epsilon$. Thus also $\tilde F$ satisfies properties similar to those in the definition of $\mathcal{G}_\epsilon$. We claim that $\tilde u_\epsilon$ is the unique
$(p,\epsilon)$-Kolmogorov function in $U_X(\rho_1)\times\mathbb R^m\times I^{\rho_1}$ with boundary values defined by $\tilde F$.
To see that $\tilde u_\epsilon$ is a $(p,\epsilon)$-Kolmogorov function we take $(Z,W,\tau)\in U_X(\rho_1)\times\mathbb R^m\times I^{\rho_1}$, and note that since
$u_\epsilon$ is a $(p,\epsilon)$-Kolmogorov function in $U_X\times\mathbb R^m\times I$, we have
\begin{align*}
\quad \tilde u_\epsilon(Z,W,\tau)&=\frac \alpha 2\biggl\{\sup_{\tilde X\in{{B_\epsilon(\hat Z+Z)}}}u_\epsilon(\tilde X,\hat W + W - \tau\hat Z+{\epsilon^2}\tilde X/2,(\hat\tau+\tau)-{\epsilon^2}/2)\biggr \}\notag
\\
&\hphantom{=}+\frac \alpha 2\biggl\{\inf_{\tilde X\in{{B_\epsilon(\hat Z+Z)}}}u_\epsilon(\tilde X,\hat W + W - \tau\hat Z+{\epsilon^2}\tilde X/2,(\hat\tau+\tau)-{\epsilon^2}/2)\biggr \}\notag
\\
&\hphantom{=}+\beta \barint_{B_\epsilon(\hat Z+Z)} u_\epsilon(\tilde X,\hat W + W - \tau\hat Z+{\epsilon^2}\tilde X/2,(\hat\tau+\tau)-{\epsilon^2}/2) \d \tilde X + \eta \notag
\\
&=\frac \alpha 2\biggl\{\sup_{\tilde X\in{{B_\epsilon(Z)}}}u_\epsilon(\tilde X+\hat Z,\hat W + W - (\tau-{\epsilon^2}/2)\hat Z+{\epsilon^2}\tilde X/2,(\hat\tau+\tau)-{\epsilon^2}/2)\biggr \}\notag
\\
&\hphantom{=}+\frac \alpha 2\biggl\{\inf_{\tilde X\in{{B_\epsilon(Z)}}}u_\epsilon(\tilde X+\hat Z,\hat W + W -(\tau-{\epsilon^2}/2)\hat Z+{\epsilon^2}\tilde X/2,(\hat\tau+\tau)-{\epsilon^2}/2)\biggr \}\notag
\\
&\hphantom{=}+\beta \barint_{B_\epsilon(Z)} u_\epsilon(\tilde X+\hat Z,\hat W + W - (\tau-{\epsilon^2}/2)\hat Z+{\epsilon^2}\tilde X/2,(\hat\tau+\tau)-{\epsilon^2}/2) \d \tilde X + \eta
\\
& = \frac \alpha 2\biggl\{\sup_{\tilde X\in{{B_\epsilon(Z)}}} \tilde u_\epsilon(\tilde X, W + \epsilon^2 \tilde X/2, \tau -{\epsilon^2}/2) + \inf_{\tilde X\in{{B_\epsilon(Z)}}} \tilde u_\epsilon(\tilde X, W + \epsilon^2 \tilde X/2, \tau -{\epsilon^2}/2)\biggr \}\notag
\\ &\hphantom{=}+ \beta \barint_{B_\epsilon(Z)} \tilde u_\epsilon(\tilde X, W + \epsilon^2 \tilde X/2, \tau -{\epsilon^2}/2) \d \tilde X.
\end{align*}
Note that for $\rho_1$ sufficiently small, $U_X(\rho_1)$ has the same regularity as $U_X$. Thus, we are precisely in the situation of Section \ref{sec3} modulo a translation of the time interval. Hence, uniqueness follows from Lemma \ref{EUKolmo}, or Lemma \ref{u4}.

We now want to compare the values of $\tilde u_\epsilon$ and $u_\epsilon$ on $\tilde \Gamma_\epsilon$ in order to conclude a relation between these functions also in $U(\rho_1)\times \R^m \times I^{\rho_1}$. A direct calculation shows that for all $(Z, W, \tau)\in \tilde \Gamma_\epsilon$, we have
\begin{align*}
d((Z,W,\tau) , (\hat Z, \hat W, \hat \tau) \circ (Z,W,\tau)) = \| (\hat Z, \hat W +(\hat\tau-\tau)\hat Z + \hat \tau Z, \hat \tau ) \| < \rho_1,
\end{align*}
if $\rho_0$ is sufficiently small. Since $\tilde \Gamma_\epsilon \subset S_{\rho_1}$, this means that we can use the first part of the proof to conclude that
\begin{align*}
 \tilde u_\epsilon(Z,W,\tau) -u_\epsilon(Z,W,\tau) = u_\epsilon((\hat Z, \hat W, \hat \tau) \circ (Z,W,\tau)) - u_\epsilon(Z,W,\tau) + \eta > 0,
\end{align*}
for $(Z, W, \tau)\in \tilde \Gamma_\epsilon$ when $\rho_0$ is chosen small enough. Hence, using the comparison principle from Lemma \ref{u4}, we have  $\tilde u_\epsilon \geq u_\epsilon$ on the set $U_X(\rho_1)\times \R^m \times I^{\rho_1}$, and hence,
\begin{align}
u_\epsilon(X,Y,t)\leq\tilde u_\epsilon(X,Y,t)=u_\epsilon(\hat X,\hat Y,\hat t) + \eta.
\end{align}
The lower bound follows by a similar argument. Thus, we have verified the assumptions of Lemma \ref{Ascoli} for the functions $u_\epsilon$ on the set $\overline{U_X}\times \R^m\times [0,T]$. In particular, the assumptions hold on the compact subsets $K_j := \overline{U_X}\times \bar B_j(0)\times [0,T]$. We can thus apply Lemma \ref{Ascoli} to find a uniformly convergent subsequence on every set $K_j$. By another diagonalization argument, we obtain a sequence which converges uniformly on every compact subset of $\overline{U_X}\times \R^m\times [0,T]$.
\end{proof}

\begin{lemma}\label{raja} Let $\{u_{\eps_j}\}$, $\eps_j\rightarrow 0$, be a sequence of value functions of the tug-of-war game with payoff equal to $F$ on $\Gamma_\epsilon$. Suppose that
$u_{\eps_j}\rightarrow u$ as $j\to\infty$ as stated in Lemma \ref{conver}. Then $u$ is a viscosity solution to \eqref{Dppa}.
\end{lemma}

\begin{proof} We will only give the proof of the statement that $u$ is a viscosity supersolution to \eqref{Dppa} as the proof of the
statement that $u$ is a viscosity subsolution to \eqref{Dppa} is analogous. Let $(\hat X,\hat Y,\hat t)\in U_X\times \mathbb R^m\times I$ and assume that $\phi\in C^2$ touches $u$ from below at $(\hat X,\hat Y,\hat t)$. In light of Lemma \ref{Vissolsimp}, we need to prove that at $(\hat X,\hat Y,\hat t)$,
\begin{eqnarray*}
(i)&& (m+p)(\partial_t\phi-\hat X\cdot\nabla_Y\phi))\geq ((p-2)\Delta_{\infty,X}^N+\Delta_X)\phi,\mbox{ if $\nabla_X \phi(\hat X,\hat Y,\hat t)\neq 0$},\notag\\
(ii)&& (m+p)(\partial_t\phi-\hat X\cdot\nabla_Y\phi))\geq 0,\mbox{ if $\nabla_X\phi(\hat X,\hat Y,\hat t)=0$ and $\nabla_X^2\phi(\hat X, \hat Y, \hat t) = 0$}.
\end{eqnarray*}
We note that the following inequality, a version of \eqref{impagain}, holds for all smooth function $\phi$,
\begin{eqnarray}\label{meanvalue3ach}
&&\frac \alpha 2\biggl\{\max_{\tilde X\in\overline{{B_\epsilon(\hat X)}}}\phi(\tilde X,\hat Y+{\epsilon^2}\tilde X/2,\hat t-{\epsilon^2}/2)+\min_{\tilde X\in\overline{{B_\epsilon(\hat X)}}}\phi(\tilde X,\hat Y+{\epsilon^2}\tilde X/2,\hat t-{\epsilon^2}/2)\biggr \}\notag
\\
&&+\beta \barint_{B_\epsilon(\hat X)} \phi(\tilde X,\hat Y+{\epsilon^2}\tilde X/2,\hat t-{\epsilon^2}/2) \d \tilde X -\phi(\hat X,\hat Y,\hat t)\notag
\\
&&\geq \alpha \frac {\epsilon^2}2 \langle \nabla_X^2\phi(\hat X,\hat Y,\hat t)(\frac{X_1^{\epsilon,\hat Y,\hat t - \epsilon^2/2}-\hat X}{\epsilon}), (\frac{X_1^{\epsilon,\hat Y,\hat t - \epsilon^2/2}-\hat X}{\epsilon})\rangle\ \notag
\\
&& \hphantom{\geq} +\frac {\epsilon^2}2(\hat  X\cdot\nabla_{Y}-\partial_{t} + \tfrac{1}{m+p}\Delta_X)\phi(\hat X,\hat Y,\hat t)+{o}(\epsilon^2).
\end{eqnarray}
Here, $X_1^{\epsilon,\tilde Y,\tilde t}$ is defined as in the deductions leading up to \eqref{impagain}. We introduce the auxiliary functions
\begin{align*}
 f(X,Y,t) &:= u(X,Y,t) - \phi(X,Y,t),\\
 f_j(X,Y,t) &:= u_{\epsilon_j}(X,Y,t) - \phi(X,Y,t).
\end{align*}
Since $\phi$ touches $u$ from below at $(\hat X, \hat Y, \hat t)$, we have
\begin{align*}
 f(\hat X, \hat Y, \hat t) &= 0,
 \\
 f(X,Y,t) &> 0, \textrm{ if } (X,Y,t) \neq (\hat X, \hat Y, \hat t).
\end{align*}
For $\rho>0$ we define
\begin{align*}
 W_\rho &:= B_\rho(\hat X) \times B_\rho(\hat Y) \times (\hat t - \rho, \hat t + \rho),\  K_\rho := \overline{W_{2\rho}}\setminus W_\rho.
\end{align*}
Let $\rho>0$ be so small that $\overline{W_{2\rho}}\subset U_X \times \R^m\times I$. By the continuity of $f$ and the compactness of $K_\rho$, there exists a $c_\rho > 0$ such that $f \geq c_\rho$ on $K_\rho$. By the uniform convergence of $u_{\epsilon_j}$, we see that $f_j > c_\rho/2$ on $K_\rho$ for large $j$. Also, $f_j(\hat X, \hat Y, \hat t) < c_\rho/4$ for large $j$. If $j$ is sufficiently large, we also have $\epsilon_j^3 < c_\rho/4$. These observations show that for large $j$,
\begin{align}\label{est:f_j}
 f_j > c_\rho/2 = c_\rho/4 + c_\rho/4 > f_j(\hat X, \hat Y, \hat t) + \epsilon_j^3 \geq \inf_{W_{2\rho}} f_j + \epsilon_j^3, \textrm{ on } K_\rho.
\end{align}
Pick a point $(X_j, Y_j, t_j)\in W_{2\rho}$ such that
\begin{align}\label{X_jY_jt_j}
 f_j(X_j,Y_j,t_j) <\inf_{W_{2\rho}} f_j + \epsilon_j^3.
\end{align}
By \eqref{est:f_j}, we see that, in fact, $(X_j, Y_j, t_j)\in W_\rho$. Hence, for large $j$ we have
\begin{align}\label{inclusion:W_2rho}
 B_{\epsilon_j}(X_j) \times \big\{ Y_j + {\epsilon_j^2}X_j/2\big\} \times \big\{ t_j -{\epsilon_j^2}/{2}\big\} \subset W_{2\rho}.
\end{align}
From \eqref{X_jY_jt_j} and the definition of $f_j$, it immediately follows that
\begin{align}\label{est:u_eps_j-phi_j}
 u_{\epsilon_j}(X,Y,t) > \phi_j(X,Y,t) - \epsilon_j^3, \hspace{7mm} (X,Y,t) \in W_{2\rho},
\end{align}
where
\begin{align*}
 \phi_j(X,Y,t) := \phi(X,Y,t) + u_{\epsilon_j}(X_j, Y_j, t_j) - \phi(X_j,Y_j,t_j).
\end{align*}
Using the fact that  $u_{\eps_j}$ is the value function of a tug-of-war game, i.e.~ $u_{\eps_j}$ is a $(p,\eps_j)$-Kolmogorov function, and the estimate \eqref{est:u_eps_j-phi_j} combined with \eqref{inclusion:W_2rho}, we thus have for large $j$
\begin{align*}
  \phi_j(X_j, Y_j, t_j) &=u_{\epsilon_j}(X_j, Y_j, t_j)
  \\
  &= \frac{\alpha}{2}\big\{ \sup_{X\in B_{\epsilon_j}(X_j)} u_{\epsilon_j} (X, Y_j + {\epsilon_j^2}X/2, t_j -{\epsilon_j^2}/{2}) + \inf_{X\in B_{\epsilon_j}(X_j)} u_{\epsilon_j} (X, Y_j + {\epsilon_j^2}X/2, t_j -{\epsilon_j^2}/{2})\big\}
 \\
 &\hphantom{=} + \beta \barint_{B_{\epsilon_j}(X_j)} u_{\epsilon_j}(X, Y_j + {\epsilon_j^2}X/2, t_j -{\epsilon_j^2}/{2}) \d X
 \\
 &\geq \frac{\alpha}{2}\big\{ \max_{X\in \overline{B_{\epsilon_j}(X_j)}} \phi_j (X, Y_j + {\epsilon_j^2}X/2, t_j -{\epsilon_j^2}/{2}) + \min_{X\in\overline{B_{\epsilon_j}(X_j)}} \phi_j (X, Y_j + {\epsilon_j^2}X/2, t_j -{\epsilon_j^2}/{2})\big\}
 \\
 &\hphantom{=} + \beta \barint_{B_{\epsilon_j}(X_j)} \phi_j(X, Y_j + {\epsilon_j^2}X/2, t_j -{\epsilon_j^2}/{2}) \d X - \epsilon_j^3.
\end{align*}
Noting that $\rho$ in the previous argument can be chosen arbitrarily small, we see that passing to subsequences we may assume that $(X_j, Y_j, t_j) \to (\hat X, \hat Y, \hat t)$. Combining the last estimate with \eqref{meanvalue3ach} with $\phi$ replaced by $\phi_j$, with $\epsilon = \epsilon_j$ and with the point $(X_j, Y_j, t_j)$ replacing $(\hat X, \hat Y, \hat t)$, we obtain
\begin{align}\label{est:phi}
 (p-2) \langle \nabla_X^2\phi(X_j, Y_j, t_j)(\frac{X_1^{\epsilon, Y_j, t_j - \epsilon_j^2/2} - X_j}{\epsilon_j}), (\frac{X_1^{\epsilon, Y_j, t_j - \epsilon_j^2/2}-\hat X}{\epsilon_j})\rangle\ + \Delta_X \phi(X_j, Y_j, t_j)
 \\
 \notag +(m+p)(X_j\cdot \nabla_Y - \partial_t)\phi(X_j,Y_j,t_j) \leq \epsilon_j^{-2}o(\epsilon_j^2)
\end{align}
where we also used the fact that the derivatives of $\phi$ and $\phi_j$ coincide. Since $\phi_j$ and $\phi$ only differ by a constant, the map $X\mapsto \phi(X, Y_j + \epsilon_j^2 X/2, t_j -\epsilon_j^2/2)$ attains its minimum in the ball $\overline{B_{\epsilon_j}(X_j)}$ at  $X_1^{\epsilon, Y_j, t_j - \epsilon_j^2/2}$. In the case that $\nabla_X \phi(\hat X, \hat Y, \hat t) \neq 0$, we can pass to the limit $j\to \infty$ reasoning as in the proof of Theorem \ref{meanvalue1thm} to obtain condition $(i)$ as desired. Consider now the case where $\nabla_X \phi(\hat X, \hat Y, \hat t) = 0$ and $\nabla_X^2\phi(\hat X, \hat Y, \hat t) = 0$. Since the norm of $(X_1^{\epsilon, Y_j, t_j - \epsilon_j^2/2} - X_j)/\epsilon_j$ stays bounded, we see that in this case the first two terms on the left-hand side of \eqref{est:phi} vanish in the limit $j\to \infty$ and we end up with $(ii)$.
\end{proof}

\begin{lemma}\label{raja+} Let $u$ and $v$ be two bounded viscosity solutions to \eqref{Dppa}. Then $u\equiv v$.
\end{lemma}
\begin{proof}  Assume that $u$ and $v$ are viscosity solutions to \eqref{Dppa} with the same boundary values. It suffices to prove that
$u\leq v$ on $U_X\times\mathbb R^m\times I$. By initially considering a smaller time interval, we may assume that $u$ and $v$ are continuously defined up to time $t=T$. We introduce the auxiliary functions
$$w(Y,t)=w_{\lambda,A,\theta,\eta}(Y,t):=\frac\theta 2 e^{-\lambda (T-t)}(|Y|^2+A)+\eta/(T-t),$$
and
$$\tilde u(X,Y,t):=u(X,Y,t)-w_{\lambda,A,\theta,\eta}(Y,t),$$
where $\theta$, $\lambda$, $A$, and $\eta$ are positive degrees of freedom. We intend to prove that
\begin{align}
\label{ineq}
u(X,Y,t)-w_{\lambda,A,\theta,\eta}(Y,t)=\tilde u(X,Y,t)\leq v(X,Y,t),
\end{align}
for all $(X,Y,t)\in U_X\times\mathbb R^m\times I$. With $\lambda$ and $A$ fixed, we can then let $\theta\to 0$ and $\eta\to 0$ in \eqref{ineq} and as a consequence $u\leq v$ on $U_X\times\mathbb R^m\times I$. Note that
\begin{align*}
(\partial_t-X\cdot\nabla_Y)w(X,Y,t)&=\frac\theta 2 e^{-\lambda (T-t)}\bigl (-2X\cdot Y+\lambda (|Y|^2+A)\bigr )+\eta/(T-t)^2\notag\\
&\geq \frac\theta 2 e^{-\lambda T}(\lambda (|Y|^2+A)-2R|Y|),
\end{align*}
as $|X|\leq R$ on $U_X$. Hence, if we let $A:=2R^2/\lambda^2$ then
\begin{align}\label{super}
(\partial_t-X\cdot\nabla_Y)w(X,Y,t)\geq \frac\theta 2 e^{-\lambda T}(\lambda (|Y|^2+A)-2R|Y|)\geq \frac\theta 2 e^{-\lambda T}R^2/\lambda>0.
\end{align}
In the following, $A$ is fixed as above.

To prove \eqref{ineq},  we argue by contradiction and we assume that there is a point $(X^*, Y^*, t^*) \in U_X \times \R^m \times I$ such that
\begin{align*}
\tilde u(X^*, Y^*, t^*) - v(X^*, Y^*, t^*) > 0.
\end{align*}
For $j\in \N$, we introduce the function
\begin{align}\label{contra2gg}
\notag w_j&: \bar U_X \times \R^m \times [0,T) \times \bar U_X \times \R^m \times [0,T] \to \R,
\\
w_j(X,Y,t,\tilde X,\tilde Y,\tilde t)&:=\tilde u(X,Y,t)-v(\tilde X,\tilde Y,\tilde t)
-\bigl(\frac {j^4}{4}|X-\tilde X|^4+\frac {j^4}{4}|Y-\tilde Y|^4+\frac j2|t-\tilde t|^2\bigr ).
\end{align}
Note that
\begin{align*}
 \sup w_j \geq w_j(X^*, Y^*, t^*,X^*, Y^*, t^*) = \tilde u(X^*, Y^*, t^*) - v(X^*, Y^*, t^*) > 0.
\end{align*}
The boundedness of $u$ and $v$ and the form of the last term in the definition of $w(Y,t)$ show that $w_j$ is negative if $t$ is close to $T$. Considering the last term of \eqref{contra2gg}, we note that $w_j$ is also negative if $j$ is large and $\tilde t$ is close to $T$. Let $B<\infty$ be such that $|u| \leq B$ and $|v|\leq B$. Then
\begin{align*}
 w_j \leq 2B - \frac\theta 2 e^{-\lambda T}(|Y|^2+A) - \frac{1}{4}|Y- \tilde Y|^4
\end{align*}
Due to the second term on the right-hand side, $w_j$ is negative if $|Y|$ is sufficiently large. The last term shows that in order for $w_j$ to be positive also $\tilde Y$ must be confined in some ball. Thus there is $\tilde R> 0$ such that $w_j$ is negative unless $Y, \tilde Y \in \overline{B(0, \tilde R)}$.
Finally, if $X\in \partial U_X$ or if $t=0$,  the fact that $u$ and $v$ coincide on the parabolic boundary gives that
\begin{align}\label{w_j:when_X_on_bdry}
 w_j(X,Y,t,\tilde X, \tilde Y,\tilde t) =& v(X,Y,t) - v(\tilde X, \tilde Y, \tilde t) - w(Y,t)\notag\\
  &- \bigl(\frac {j^4}{4}|X-\tilde X|^4+\frac {j^4}{4}|Y-\tilde Y|^4+\frac j2|t-\tilde t|^2\bigr ).
\end{align}
Note that $-w(Y,t) \leq - c$ for some positive constant $c$. Hence, if the distance of the points $(X,Y,t)$ and $(\tilde X, \tilde Y, \tilde t)$ is smaller than some fixed limit $d$, then the right-hand side of \eqref{w_j:when_X_on_bdry} is negative due to the uniform continuity of $v$. On the other hand, if the distance is larger than $d$, then the right-hand side is also negative provided that $j$ is sufficiently large. Similarly we can show that $w_j$ is negative for large $j$ if $\tilde X \in \partial U_X$ or if $\tilde t = 0$. All in all, these observations show that there exist, for large $j\in \N$ and $\tilde R>0$,  points $(X_j,Y_j,t_j,\tilde X_j,\tilde Y_j,\tilde t_j)$ in $(U_X \times \overline{B(0,\tilde R)} \times I)\times (U_X \times \overline{B(0,\tilde R)} \times I)$ at which the supremum of $w_j$ is attained.

As in the proof of Proposition 3.7 in \cite{crandallil92}, we see that
\begin{eqnarray}\label{contra4-gg}
\frac {j^4}{4}|X_j-\tilde X_j|^4\rightarrow 0,\ \frac {j^4}{4}|Y_j-\tilde Y_j|^4\rightarrow 0,\mbox{ and }\frac j2|t_j-\tilde t_j|^2\rightarrow 0.
\end{eqnarray}
and since the points $(X_j,Y_j,t_j,\tilde X_j,\tilde Y_j,\tilde t_j)$ are confined to a compact set, we can after passing to a subsequence deduce that
they converge to a point which in light of \eqref{contra4-gg} must be of the form $(\hat X, \hat Y, \hat t, \hat X, \hat Y, \hat t)$.


Assume that $X_{j_l}=\tilde X_{j_l}$ for an infinite sequence $\{j_l\}_l$ with $j_l\geq j_0$.  Let
\begin{eqnarray*}
\quad\varphi_{j}(X,Y,t,\tilde X,\tilde Y,\tilde t):=\frac {{j^4}}{4}|X-\tilde X|^4+\frac {{j^4}}{4}|Y-\tilde Y|^4+\frac {j}2|t-\tilde t|^2.
\end{eqnarray*}
Then
\begin{eqnarray*}
(\tilde X,\tilde Y,\tilde t)\to v(\tilde X,\tilde Y,\tilde t)+\varphi_{j_l}(X_{j_l},Y_{j_l},t_{j_l},\tilde X,\tilde Y,\tilde t)
\end{eqnarray*}
has a local minimum at $(\tilde X_{j_l},\tilde Y_{j_l},\tilde t_{j_l})$, and
\begin{eqnarray*}
(X, Y,t)\to \tilde u(X,Y,t)-\varphi_{j_l}(X,Y,t,\tilde X_{j_l},\tilde Y_{j_l},\tilde t_{j_l})
\end{eqnarray*}
has a local maximum at $(X_{j_l},Y_{j_l},t_{j_l})$. The first statement implies that there exists an open set $D\subset\mathbb R^{M+1}$, containing
$(\tilde X_{j_l},\tilde Y_{j_l},\tilde t_{j_l})$, such that on $D$,  $v$ is touched from below at $(\tilde X_{j_l},\tilde Y_{j_l},\tilde t_{j_l})$ by the function
\begin{eqnarray*}
\quad\quad\quad\phi_1(\tilde X,\tilde Y,\tilde t):=v(\tilde X_{j_l},\tilde Y_{j_l},\tilde t_{j_l})+\varphi_{j_l}(X_{j_l},Y_{j_l},t_{j_l},\tilde X_{j_l},\tilde Y_{j_l},\tilde t_{j_l})-\varphi_{j_l}(X_{j_l},Y_{j_l},t_{j_l},\tilde X,\tilde Y,\tilde t).
\end{eqnarray*}
The second statement implies that there exists an open set $D\subset\mathbb R^{M+1}$, containing $(X_{j_l},Y_{j_l},t_{j_l})$, such that on $D$,  $\tilde u$ is touched from above at $(X_{j_l},Y_{j_l},t_{j_l})$ by the function
\begin{eqnarray*}
\quad\quad\quad\phi_2(X,Y,t):=\tilde u(X_{j_l},Y_{j_l},t_{j_l})-\varphi_{j_l}( X_{j_l}, Y_{j_l},t_{j_l},\tilde X_{j_l},\tilde Y_{j_l},\tilde t_{j_l})+\varphi_{j_l}(X,Y,t,\tilde X_{j_l},\tilde Y_{j_l},\tilde t_{j_l}).
\end{eqnarray*}
Note that the last statement implies that on $D$,  $u$ is touched from above at $(X_{j_l},Y_{j_l},t_{j_l})$ by the function
\begin{align*}
\tilde\phi_2(X,Y,t)&:=u(X_{j_l},Y_{j_l},t_{j_l})-w(Y_{j_l},t_{j_l})-\varphi_{j_l}( X_{j_l}, Y_{j_l},t_{j_l},\tilde X_{j_l},\tilde Y_{j_l},\tilde t_{j_l})\notag\\
&+\varphi_{j_l}(X,Y,t,\tilde X_{j_l},\tilde Y_{j_l},\tilde t_{j_l})+w(Y,t),
\end{align*}
where  $w(Y,t)=w_{\lambda,A,\theta,\eta}(Y,t)$ was introduced above.

Using that $X_{j_l}=\tilde X_{j_l}$ we deduce from the first statement, as $v$ is a viscosity solution, that
\begin{align*}
0\leq (\partial_{\tilde t}\phi_1-\tilde X\cdot\nabla_{\tilde Y}\phi_1)(\tilde X_{j_l},\tilde Y_{j_l},\tilde t_{j_l})&=((\tilde X\cdot\nabla_{\tilde Y}-\partial_{\tilde t})\varphi_{j_l}(X_{j_l},Y_{j_l},t_{j_l},\cdot,\cdot,\cdot))(\tilde X_{j_l},\tilde Y_{j_l},\tilde t_{j_l})\notag\\
&=- {{j_l^4}}(X_{j_l}\cdot (Y_{j_l}-\tilde Y_{j_l})) |Y_{j_l}-\tilde Y_{j_l}|^2+{j_l}(t_{j_l}-\tilde t_{j_l}).
\end{align*}
I.e., we deduce
\begin{align}\label{Conc1}
{{j_l^4}}(X_{j_l}\cdot (Y_{j_l}-\tilde Y_{j_l})) |Y_{j_l}-\tilde Y_{j_l}|^2-{j_l}(t_{j_l}-\tilde t_{j_l})\leq 0.
\end{align}
From the second statement, we deduce, as $u$ is a viscosity solution and by using \eqref{super}, that
\begin{align*}
0\geq (\partial_{t}\tilde\phi_2-X\cdot\nabla_{Y}\tilde\phi_2)(X_{j_l},Y_{j_l},t_{j_l})=&((\partial_{t}-X\cdot\nabla_{Y})\varphi_{j_l}(\cdot,\cdot,\cdot,\tilde X_{j_l},\tilde Y_{j_l},\tilde t_{j_l}))(X_{j_l},Y_{j_l},t_{j_l})\notag\\
&+(\partial_{t}-X_{j_l}\cdot\nabla_{Y})w(Y_{j_l},t_{j_l})\notag\\
\geq &- {{j_l^4}}(X_{j_l}\cdot (Y_{j_l}-\tilde Y_{j_l})) |Y_{j_l}-\tilde Y_{j_l}|^2+{j_l}(t_{j_l}-\tilde t_{j_l})\notag\\
&+\frac\theta 2 e^{-\lambda T}R^2/\lambda.
\end{align*}
I.e., we deduce
\begin{align}\label{Conc2}
{{j_l^4}}(X_{j_l}\cdot (Y_{j_l}-\tilde Y_{j_l})) |Y_{j_l}-\tilde Y_{j_l}|^2-{j_l}(t_{j_l}-\tilde t_{j_l})\geq \frac\theta 2 e^{-\lambda T}R^2/\lambda.
\end{align}
From \eqref{Conc1} and \eqref{Conc2} we conclude a contradiction and therefore either our original assumption must be incorrect, and then we are done, or $X_j\neq \tilde X_j$ for all $j\geq j_0$ and for some $j_0\gg 1$.

Assume now that there is a $j_0 \in \N$ such that $X_j\neq \tilde X_j$ for all $j\geq j_0$. In this case we use Theorem 3.2 in \cite{crandallil92} with the choices $k=2$, $u_1=\tilde u$, $u_2 = -v$, $\hat x = P_j:=(X_j, Y_j, t_j, \tilde X_j, \tilde Y_j, \tilde t_j)$. In our case, the function $w$ in Theorem 3.2 in \cite{crandallil92} is $w_j$ and the function $\phi$ corresponds to $\varphi_j$. Similarly as in the proof of Lemma \ref{Vissolsimp} this allows us to conclude that there are symmetric $(M+1)\times (M+1)$ matrices $E, H$ such that
\begin{align*}
 (\nabla_{X,Y,t}\varphi_j(P_j), H) &\in \bar J^{2,+}(u-w)(X_j,Y_j,t_j),
 \\
 (\nabla_{\tilde X, \tilde Y, \tilde t} \varphi_j(P_j), E) &\in \bar J^{2,+}(-v)(\tilde X_j, \tilde Y_j, \tilde t_j),
 \\
 H + E &\leq 0.
\end{align*}
Hence, we have sequences
\begin{align*}
 (X^k_j, Y^k_j, t^k_j) &\to (X_j, Y_j, t_j),\hspace{5mm} \eta_k \to \nabla_{X,Y,t}\varphi_j(P_j), \hspace{5mm} H_k \to H,
 \\
 (\tilde X^k_j, \tilde Y^k_j, \tilde t^k_j) &\to (\tilde X_j, \tilde Y_j, \tilde t_j), \hspace{5mm} \xi_k \to \nabla_{\tilde X, \tilde Y, \tilde t} \varphi_j(P_j),\hspace{5mm} E_k \to E,
\end{align*}
such that
\begin{align*}
 (\eta_k, H_k) &\in J^{2,+}(u-w)(X^k_j, Y^k_j,t^k_j),
 \\
 (\xi_k, E_k) &\in J^{2,+}(-v)(\tilde X^k_j, \tilde Y^k_j, \tilde t^k_j).
\end{align*}
As noted in \cite{crandallil92}, we can find a $C^2$ function $f$ touching $v$ from below at $(\tilde X^k_j,\tilde Y^k_j,\tilde t^k_j)$, and a $C^2$ function $g$ touching $\tilde u = u -w$ from above at $(X^k_j, Y^k_j, t^k_j)$, such that
\begin{align}
 \label{f_xi_E}(\nabla_{X,Y,t} f, \nabla_{X,Y,t}^2 f)(\tilde X^k_j, \tilde Y^k_j, \tilde t^k_j) &= (-\xi_k, -E_k),
 \\
 \label{g_eta_H}(\nabla_{X,Y,t} g, \nabla_{X,Y,t}^2 g)(X^k_j, Y^k_j, t^k_j) &= (\eta_k, H_k).
\end{align}
By the assumption $X_j \neq \tilde X_j$ and the convergence of $\xi_k$, we see that $\nabla_X f\neq 0$ for large $k$ and thus, since $v$ is a viscosity solution we have
\begin{align*}
 (m+p)(\partial_tf - \tilde X^k_j\cdot \nabla_Yf)(\tilde X^k_j, \tilde Y^k_j, \tilde t^k_j) \geq ((p-2)\Delta^N_{\infty,X} + \Delta_X)f(\tilde X^k_j, \tilde Y^k_j, \tilde t^k_j),
\end{align*}
which using \eqref{f_xi_E} can be written as
\begin{align*}
 (m+p)(-\xi^t_k + \tilde X^k_j\cdot \xi^Y_k) \geq -(p-2)(\hat \xi^X_k)^T E^X_k \hat \xi^X_k - \tr(E^X_k).
\end{align*}
The superscripts $X,Y,t$  of $\xi$ indicate the components and the hat indicates taking the unit vector. $E_k^X$ refers to the subminor of $E_k$ corresponding to the $X$-components. Passing to the limit $k\to \infty$ we end up with
\begin{align}\label{k_to_infty_v}
 (m+p)(\partial_t \varphi_j(P_j) - \tilde X_j \cdot \nabla_Y \varphi_j(P_j)) \geq -(p-2)\widehat{\nabla_X\varphi_j(P_j)}^T E^X \widehat{\nabla_X\varphi_j(P_j)} - \tr(E^X),
\end{align}
where $\widehat{\nabla_X\varphi_j(P_j)}$ equals $\nabla_X\varphi_j(P_j)/|\nabla_X\varphi_j(P_j)|$. Similarly, as $g+w$ touches $u$ from above,
\begin{align*}
 (m+p)(\partial_t(g+w) - X^k_j\cdot \nabla_Y(g+w))(X^k_j,Y^k_j,t^k_j) \leq ((p-2)\Delta^N_{\infty,X} + \Delta_X)g( X^k_j, Y^k_j, t^k_j).
\end{align*}
Combining this estimate with \eqref{super} and \eqref{g_eta_H} we obtain
\begin{align*}
 \frac\theta 2 e^{-\lambda T}R^2/\lambda + (m+p)(\eta^t_k - X^k_j\cdot \eta^Y_k) \leq (p-2)(\hat \eta^X_k)^T H^X_k \hat \eta^X_k + \tr(H^X_k),
\end{align*}
and passing to the limit $k\to\infty$ we have
\begin{align}\label{k_to_infty_u}
 &\frac\theta 2 e^{-\lambda T}R^2/\lambda + (m+p) (\partial_t \varphi_j(P_j) - X_j \cdot \nabla_Y \varphi_j(P_j))\notag\\
  &\leq (p-2) \widehat{\nabla_X\varphi_j(P_j)}^T H^X \widehat{\nabla_X\varphi_j(P_j)} + \tr(H^X).
\end{align}
Combining \eqref{k_to_infty_v} and \eqref{k_to_infty_u} we see that
\begin{align}\label{lexp}
 \frac\theta 2 e^{-\lambda T}R^2/\lambda + (m+p)(\tilde X_j - X_j)\cdot \nabla_Y\varphi_j(P_j) \leq& (p-2)\widehat{\nabla_X\varphi_j(P_j)}^T (H^X + E^X) \widehat{\nabla_X\varphi_j(P_j)}\notag\\
 & + \tr(H^X + E^X).
\end{align}
Denoting the eigenvalues of $H^X + E^X$ by $(\lambda_j)_{j=1}^m$ we see that they must all be nonpositive since $H^X + E^X \leq 0$. Letting $z \in \mathbb{S}^{m-1}$ denote the coordinate vector of $\widehat{\nabla_X\varphi_j(P_j)}$ with respect to a corresponding orthonormal basis of eigenvectors. Then  we may express the right-hand side in \eqref{lexp} as
\begin{align*}
 (p-2)\sum^m_{j=1} \lambda_j z_j^2 + \sum^m_{j=1} \lambda_j = (p-1)\sum^m_{j=1} \lambda_j z_j^2 + \sum^m_{j=1} \lambda_j(1-z_j^2) \leq 0.
\end{align*}
Thus, we have proved that
\begin{align*}
 \frac\theta 2 e^{-\lambda T}R^2/\lambda + (m+p)(\tilde X_j - X_j)\cdot \nabla_Y\varphi_j(P_j) \leq 0.
\end{align*}
As in the proof of Lemma \ref{Vissolsimp} we see that the second term on the left-hand side vanishes in the limit $j\to \infty$ and we have again reached a contradiction.
\end{proof}

Using the uniqueness result in Lemma \ref{raja+}, the convergence result of Lemma \ref{conver}, and Lemma \ref{raja}, we can finally state and prove the main theorem of this section.
\begin{theorem}\label{thm:existence_uniqueness}
Suppose that $F:\Gamma_\epsilon \to \R$ is a bounded function satisfying \eqref{regu}. Let $\{u_{\eps}\}$, $\eps >0$ be the value functions of the tug-of-war game with payoff equal to $F$ on $\Gamma_\epsilon$. Then $u_\epsilon$ converges uniformly on compact subsets of $\overline{U_X}\times \R^m \times [0,T]$ as $\epsilon \to 0$ to the unique bounded viscosity solution $u$ to \eqref{Dppa}. Moreover, $u$ is continuous.
\end{theorem}
\begin{proof}
 First note that the boundedness of $F$ and \eqref{regu} together imply that $F\in \mathcal{G}_\epsilon$, so that we can apply the previous results of this section. Combining Lemma \ref{conver} and Lemma \ref{raja} we see that every sequence $u_{\epsilon_j}$ has a subsequence which converges uniformly on compact subsets to a bounded viscosity solution to \eqref{Dppa}. By Lemma \ref{raja+}, these solutions all coincide, and we denote this solution by $u$. By Lemma \ref{conver}, $u$ is continuous. To show that the family $u_\epsilon$ converges uniformly to $u$ on compact subsets, suppose the contrary. Then we have a compact subset $K$, a number $\rho>0$ and a sequence $\epsilon_j \to 0$ such that
 \begin{align*}
  \sup_K |u_{\epsilon_j} - u| \geq \rho, \hspace{7mm}j \in \N.
 \end{align*}
But then there would be no subsequence of $u_{\epsilon_j}$ converging to $u$, which is a contradiction.
\end{proof}
The following corollary provides some sufficient conditions for the existence and uniqueness of solutions to \eqref{Dppa} in the case that $F$ is only a priori defined on the parabolic boundary $\Gamma^1\cup \Gamma^2$. Note the exceptional exponent $1/2$ rather than $1/3$ in the middle term on the right-hand side of \eqref{cond:F_Lip_wrt_weird_metric}. For example, all bounded functions which are Lipschitz with respect to the Euclidean metric satisfy the condition \eqref{cond:F_Lip_wrt_weird_metric}.
\begin{corollary}\label{cor:F_only_def_on_par_bdry}
 Let $F:\Gamma^1\cup \Gamma^2\to \R$ be bounded and suppose that
 \begin{align}\label{cond:F_Lip_wrt_weird_metric}
  |F(X,Y,t)-F(\tilde X, \tilde Y, \tilde t)|\leq c\big(|X - \tilde X| + |Y - \tilde Y|^\frac12 + |t - \tilde t|^\frac12\big).
 \end{align}
Then there is a unique viscosity solution to \eqref{Dppa}.
\end{corollary}
\begin{proof}
 Note that
 \begin{align*}
  \hat d((X,Y,t), (\tilde Y, \tilde X, \tilde t)) := |X - \tilde X| + |Y - \tilde Y|^\frac12 + |t - \tilde t|^\frac12,
 \end{align*}
defines a metric in $\R^{M+1}$. By \eqref{cond:F_Lip_wrt_weird_metric}, $F$ is Lipschitz with respect to $\hat d$ so by the McShane-Whitney extension lemma we can extend $F$ to a function which is Lipschitz with respect to $\hat d$ on $\R^{M+1}$. Furthermore, recalling that the original $F$ was bounded we see that truncating the extended $F$ from above and below results in a bounded $\hat d$-Lipschitz extension of $F$ defined on all of $\R^{M+1}$. By Theorem \eqref{thm:existence_uniqueness}, it is sufficient to show that $F$ satisfies \eqref{regu} on some $\Gamma_{\epsilon}$. In order to do this, fix an arbitrary $\epsilon>0$ and note that due to the boundedness of $U_X$ we have that
\begin{align*}
 |(t-\tilde t) \tilde X| \leq \kappa < \infty,
\end{align*}
whenever $(X,Y,t), (\tilde X, \tilde Y, \tilde t) \in \Gamma_\epsilon$. In the case that $|Y - \tilde Y| > 2\kappa$, we thus have
\begin{align*}
 d((X,Y,t),(\tilde X, \tilde Y, \tilde t)) \geq |Y-\tilde Y + (t-\tilde t)\tilde X|^\frac13 \geq \kappa^\frac13 \geq (2||F||_\infty + 1)^{-1} \kappa^\frac13 |F(X,Y,t) - F(\tilde X,\tilde Y, \tilde t)|.
\end{align*}
On the other hand, if $|Y - \tilde Y| \leq 2\kappa$ we may estimate
\begin{align*}
 |Y- \tilde Y|^\frac12 &\leq |Y-\tilde Y + (t-\tilde t)\tilde X|^\frac12 + |(t-\tilde t)|^\frac12|\tilde X|^\frac12
 \\
 &\leq |Y-\tilde Y + (t-\tilde t)\tilde X|^\frac16 |Y-\tilde Y + (t-\tilde t)\tilde X|^\frac13 + c|t-\tilde t|^\frac12
 \\
 &\leq (3\kappa)^\frac16 |Y-\tilde Y + (t-\tilde t)\tilde X|^\frac13 + c|t-\tilde t|^\frac12.
\end{align*}
Combining the last estimate with \eqref{cond:F_Lip_wrt_weird_metric} we can again bound $|F(X,Y,t) - F(\tilde X, \tilde Y, \tilde t)|$ by the quantity $d((X,Y,t),(\tilde X, \tilde Y, \tilde t))$ modulo a multiplicative constant. Since $d$ and $d_\mathcal{K}$ are comparable in size, we have proved \eqref{regu}.
\end{proof}
We have opted to work with the quasi-metric $d_{\mathcal{K}}$ which reflects the natural scaling in the equation. It seems that the previous arguments could also be carried out if one modifies the exponents $1, \tfrac13, \tfrac12$ present in the definition of $d_{\mathcal{K}}$, which at the cost of increased technicality would allow for a weaker assumption than \eqref{cond:F_Lip_wrt_weird_metric}.

\section{Future research and open problems}\label{OP}
In this paper, we have  completed one version  of Tug-of-war with Kolmogorov in domains of the form $U_X\times\mathbb R^m\times I$. There are many open research problems to pursue and we here just state three of them.\\

\noindent{Problem 1}: Develop a more complete theory concerning the existence and uniqueness of viscosity solutions to
\begin{align}\label{Dppaintr+}
\begin{cases}
\K_p u(X,Y,t)=0  ,\quad  &\textrm{for}\quad (X,Y,t)\in \Omega,\\
u(X,Y,t)  =  F(X,Y,t)
,\quad  &\textrm{for}\quad  (X,Y,t)\in \partial_\K\Omega,
\end{cases}
\end{align}
in potentially  velocity $(X)$, position $(Y)$ as well as time $(t)$ dependent domains $\Omega\subset\mathbb R^{M+1}$. Already the cylindrical cases
$\Omega=U_X\times U_Y\times I$, with $U_Y\subset\mathbb R^m$ smooth and bounded,  and $\Omega=\mathbb R^m\times U_Y\times I$, are very interesting.\\

\noindent{Problem 2}: Consider $\Omega=U_X\times U_Y\times I$, with $U_Y\subset\mathbb R^m$ smooth and bounded. Construct a tug-of-war game in $\Omega$ such that
the (fair) value function of the  game $u_\epsilon$ satisfies  $u_\epsilon\to u$, as $\epsilon\to 0$, where $u$ is the unique viscosity solution to
\eqref{Dppaintr+}.\\

\noindent{Problem 3}: Let $u$ be a non-negative viscosity solutions to $\K_p u=0$ in $\Omega$. Prove a Harnack inequality. Are viscosity solutions locally H{\"o}lder continuous?\\

\noindent{\bf Acknowledgements}. We would like to thank an anonymous referee for emphasizing the subtle measure theoretical issues underlying the analysis
in Section \ref{sec3} and for encouraging us to reread \cite{luirops14}. This helped us to improve  the analysis in Section \ref{sec3}. The second author thanks Juan Manfredi for some preliminary communications concerning this project and Svante Janson for discussions concerning (Borel) measurability.

\end{document}